\documentclass{article}
\usepackage[cp1252]{inputenc}
\usepackage[OT1]{fontenc}
\usepackage[english]{babel}
\usepackage{amsmath}
\usepackage{amsfonts}
\usepackage{amssymb}
\usepackage{amsthm}
\usepackage[left=2cm,right=2cm,top=2cm,bottom=2cm]{geometry}
\usepackage{enumitem,multicol}
\usepackage{float, graphicx}

\usepackage[matrix,arrow,curve]{xy}
\usepackage{lscape,comment}

\numberwithin{equation}{section}

\newtheorem{theorem}{Theorem}
\newtheorem*{remark}{Remark}

\newtheorem{corollary}{Corollary}

\newtheorem{prop}{Proposition}

\def\ad{\operatorname{ad}}
\def\const{\operatorname{const}}
\def\spann{\operatorname{span}}
\def\sgn{\operatorname{sgn}}
\def\Id{\operatorname{Id}}

\def\card{\operatorname{card}}

\def\cut{\operatorname{cut}}
\def\ds{\displaystyle}

\newcommand{\eq}[1]{$(\protect\ref{#1})$}
\newcommand{\be}[1]{\begin{equation}\label{#1}}
\newcommand{\ee}{\end{equation}}

\renewcommand\Vec{\operatorname{Vec}}

\newcommand{\onefiglabelsize}[4]
{
\begin{figure}[htbp]
\begin{center}
\includegraphics[width=#4\textwidth]{#1}
\\
\parbox[t]{#4\textwidth}{\caption{#2}\label{#3}}
\end{center}
\end{figure}
}

\newcommand{\twofiglabelsize}[8]
{
  \begin{minipage}{\linewidth}
 \centering
      \begin{minipage}{0.45\linewidth}
          \begin{figure}[H]
					 \centering
              \includegraphics[width=#4\linewidth]{#1}
              \caption{#2}\label{#3}
          \end{figure}
      \end{minipage}
      \hspace{0.05\linewidth}
      \begin{minipage}{0.45\linewidth}
          \begin{figure}[H]
					 \centering
              \includegraphics[width=#8\linewidth]{#5}
              \caption{#6}\label{#7}
          \end{figure}
      \end{minipage}
 \end{minipage}
}

\title{Sub-Finsler geodesics on the Cartan group\footnote{Sections 1, 2 and 6 of the paper are written by E. Le Donne, 
and Sections 3--5 and 7, 8 are written by A. Ardentov  and Yu. Sachkov. 
The work of A. Ardentov  and Yu. Sachkov is supported by the Russian Science Foundation 
under grant 17-11-01387 and performed in Ailamazyan Program Systems Institute 
of Russian Academy of Sciences.
E. Le Donne was partially supported by the Academy of Finland (grant
288501
`\emph{Geometry of subRiemannian groups}')
and by the European Research Council
 (ERC Starting Grant 713998 GeoMeG `\emph{Geometry of Metric Groups}').}}

\author{A.Ardentov\footnote{Program Systems Institute, Pereslavl-Zalessky, Russia, \tt{aaa@pereslavl.ru}}, E. Le Donne\footnote{Department of Mathematics and Statistics, P.O. Box 35,
FI-40014, University of Jyv\"askyl\"a, Finland,
\tt{ledonne@msri.org}}, 
Yu. Sachkov\footnote{Program Systems Institute, Pereslavl-Zalessky, Russia, \tt{yusachkov@gmail.com}}}

\begin{document}

\maketitle


\begin{abstract}
This paper is a continuation of the work by the same authors on the Cartan group equipped with the sub-Finsler  $\ell_\infty$ norm.
We start by giving a  detailed presentation of the structure of bang-bang extremal trajectories.
Then we prove upper bounds on the number of switchings on bang-bang minimizers.
We prove that any normal extremal is either bang-bang, or singular, or mixed. 
Consequently, we study   mixed extremals.
In particular, we prove that every two points can be connected by a piecewise smooth
minimizer, and we give a uniform bound on the number of such pieces.
\end{abstract}

\section{Introduction}
There are several motivations for studying sub-Finsler geometry on Lie groups, especially in geometric group theory and in harmonic analysis.
We only mention the prominent articles \cite{pansu,cowlingmartini13, Breuillard-LeDonne1} and then we refer to the introductions of \cite{bbds, SFCartan1} for a broad explanation of
the reasons and for several references of the state-of-the-art.

On the one hand, as in sub-Riemannian geometry, distributions of step 2 are easier to study and there is already some good understanding of the lower dimensional cases, see  \cite{bbds}. On the other hand, sub-Finsler structures defined by smooth norms have a similar theory that in the sub-Riemannian case. For these reasons the challenge is to study step-3 sub-Finsler groups with a non-strictly convex norm. The lower dimensional examples are the Engel group and the Cartan group, which both have step 3 and rank 2. 

In this paper we study the Cartan group, since it is the  free-nilpotent group of rank 2 and step 3 (so the Engel group is a quotient of this group), equipped with the   $\ell_\infty$ sub-Finsler   structure.
In our previous paper~\cite{SFCartan1},     adopting the point of view of   time-optimal control theory, we 
  characterized extremal curves via Pontryagin maximum principle,  we described abnormal and singular arcs, and we constructed the bang-bang flow.
 
 The Cartan  distribution can be expressed   by the span of two vector fields $X_1$, $X_2$.
We consider the $\ell_\infty$   norm with respect to   $X_1$, $X_2$.  Hence, every admissible trajectory is characterized by two controls.
A summary of the results of this paper is given by the following statements.

\begin{theorem}\label{th:result}
In the $\ell_\infty$ sub-Finsler structure on the Cartan group
 the length-minimizing  trajectories are of three  not-mutually-exclusive  types:
\begin{itemize}
\item[(i)] one component of the control is constantly equal to $1$ or $-1$, 
\item[(ii)]  bang-bang trajectory,
\item[(iii)] piecewise smooth concatenation of trajectories of types (i) and (ii). 
\end{itemize}
The length-minimizers  that are of  type (ii) but not of type (i) have at most $12$ arcs.
The length-minimizers  of type (iii) have at most $14$ arcs.
All curves of type (i) are length-minimizers.
 Moreover,  for every trajectory of type (i) there exists a piecewise-smooth length-minimizing trajectory connecting the same two points and having at most $5$ smooth pieces.
\end{theorem}
 
As a corollary, we deduce that
any pair of points  
can  be connected by an optimal piecewise-smooth trajectory with at most 14 arcs.

The paper has the following structure.
In Sec. 2 we recall the problem statement and the main results on it obtained in previous paper~\cite{SFCartan1}.
Section 3 is devoted to detailed study of structure of bang-bang extremal trajectories implied by Pontryagin Maximum Principle.
In Sec. 4 we prove upper bounds on the number of switchings on bang-bang minimizers. In Sec. 5 we prove that any normal extremal is either bang-bang, or singular, or mixed. Further, Sec. 6 is devoted to the study of mixed extremals, including upper bound on the number of switchings. Finally, in Sec. 7 we obtain a uniform bound on the number of smooth pieces on minimizers connecting arbitrary points in the Cartan group.

\section{Problem statement and previous results}\label{sec:problem}
Consider the 5-dimensional free nilpotent  Lie algebra with 2 generators, of step 3. 
There exists a basis 
$L = \spann (X_1, \dots, X_5)$ in which the product rule in $L$ takes the form
$$
[X_1, X_2] = X_3, \quad [X_1, X_3] = X_4, \quad [X_2, X_3] = X_5, \quad \ad X_4 = \ad X_5 = 0.
$$
The Lie algebra $L$ is
called the Cartan algebra, and the corresponding 
 connected simply connected Lie group  $M$ is called the Cartan group. We will use the following model:
$$ M = \mathbb{R}^5_{x,y,z,v,w}, $$
with the Lie algebra $L$ modeled by left-invariant vector fields on $\mathbb{R}^5$
\begin{align*}
&X_1 = \frac{\partial}{\partial x} - \frac y2 \frac {\partial}{\partial z} - \frac{x^2 + y^2}{2} \frac{\partial}{\partial w}, \\
&X_2 = \frac{\partial}{\partial y} + \frac x2 \frac {\partial}{\partial z} + \frac{x^2 + y^2}{2} \frac{\partial}{\partial v}, \\
&X_3 = \frac{\partial}{\partial z} + x \frac{\partial}{\partial v} + y \frac{\partial}{\partial w}, \\
&X_4 = \frac{\partial}{\partial v}, \\
&X_5 = \frac{\partial}{\partial w}.
\end{align*}
The product rule in the Cartan group $M$ in this model is given in \cite{dido_exp}.

Left-invariant $\ell_{\infty}$ sub-Finsler problem on the Cartan group is stated as the following time-optimal problem:
\begin{align}
&\dot{q} = u_1 X_1 + u_2 X_2, \quad q \in M, \quad u \in U = \{ u \in \mathbb{R}^2 \mid  {\lVert u \rVert}_\infty \le 1 \}, \label{sys}\\
\nonumber
&\lVert u \rVert_\infty = \max (|u_1|, |u_2|), \\
&q(0) = q_0 = \Id = (0, \dots, 0), \quad q(T) = q_1, \label{bound}\\
&T \to \min. \label{T}
\end{align}

Problem~\eq{sys}--\eq{T} was considered first in paper~\cite{SFCartan1}. We recall the main results of that paper.
 
Existence of optimal controls follows from Rashevsky-Chow and  Filippov theorem \cite{notes}.

Pontryagin Maximum Principle implies that optimal abnormal controls are constant.

Introduce linear-on-fibers Hamiltonians $h_i (\lambda) = \langle \lambda, X_i \rangle$, $\lambda \in T^*M$, $i = 1, \dots, 5$.
A normal extremal arc $\lambda_t, t \in I = (\alpha, \beta) \subset [0, T]$ is called:
\begin{itemize}
\item a bang-bang arc if
$$ \card \{ t \in I \mid h_1 h_2 (\lambda_t) = 0 \} < \infty, $$
\item a singular arc if one of the condition holds:
\begin{align*}
&h_1 (\lambda_t) \equiv 0, \quad t \in I \quad (\textrm{$h_1$-singular arc}), \text{ or}\\
&h_2 (\lambda_t) \equiv 0, \quad t \in I \quad (\textrm{$h_2$-singular arc}),
\end{align*}
\item a mixed arc if it consists of a finite number of bang-bang and singular arcs.
\end{itemize}
 
Singular controls have one of components constantly equal to 1 or $-1$, thus they are optimal. The fix-time attainable set along singular trajectories was explicitly described and was shown to be semi-algebraic.

Bang-bang extremal trajectories satisfy the Hamiltonian system 
with the Hamiltonian function $H = |h_1| + |h_2|$:
\begin{equation}\label{Ham_bang}
\begin{cases} 
\dot{h}_1 = -s_2 h_3, \\
\dot{h}_2 = s_1 h_3, \\
\dot{h}_3 = s_1 h_4 + s_2 h_5, \\
\dot{h}_4 = \dot{h}_5 = 0, \\
\dot{q} = s_1 X_1 + s_2 X_2.
\end{cases}
\end{equation}

The dual of the Lie algebra $L^* = T_{\Id}^*M$ has Casimir functions $h_4$, $h_5$, $E = \frac{h_3^2}{2} + h_1 h_5 - h_2 h_4$, thus Hamiltonian system \eqref{Ham_bang} has integrals $h_4$, $h_5$, $E$, and $H$.

The mapping $(\lambda, q) \mapsto (k\lambda, q), k> 0$, preserves extremal trajectories, thus  we can consider only the reduced case 
\begin{gather*}
H(\lambda) \equiv 1.
\end{gather*} 

With the use of the coordinate 
$\theta \in S^1 = \mathbb{R}/2\pi \mathbb{Z}$:
$$h_1 = \sgn (\cos \theta) \cos^2 \theta, \quad h_2 = \sgn (\sin \theta) \sin^2 \theta,$$
the vertical part of Hamiltonian system \eq{Ham_bang} reduces to the following system:
\begin{equation}\label{Hamtheta}
\begin{cases} 
\dot{\theta} = \frac{h_3}{|\sin 2\theta|}, \quad \theta \ne \frac{\pi n}{2}, \\
\dot{h}_3 = s_1 h_4 +s_2 h_5, \quad s_1 = \sgn \cos \theta, \quad s_2  = \sgn \sin \theta.
\end{cases}
\end{equation}

System~\eq{Hamtheta} is preserved by the group of symmetries of the square 
$\{ (h_1, h_2) \in \mathbb{R}^2  \mid  |h_1| + |h_2| = 1 \}$.  
Thus in the study of system \eq{Hamtheta} we can restrict ourselves by the case $h_4 \geq h_5 \geq 0$. This case obviously decomposes into the following sub-cases:
\begin{enumerate}
\item[$1)$] $h_4 > h_5 >0$,
\item[$2)$] $h_4 > h_5 = 0$,
\item[$3)$] $h_4 = h_5 > 0$,
\item[$4)$] $h_4 = h_5 = 0$.
\end{enumerate}

\section{Structure of bang-bang trajectories}\label{sec:structure}
In this section we consider, case by case, the structure of bang-bang trajectories implied by Pontryagin Maximum Principle.

\subsection{Case $1)$}
Let $h_4 > h_5 > 0$. Then system~\eq{Hamtheta} has the phase portrait given in Fig.~\ref{fig:1)h3th}, see Subsubsec.~7.2.1~\cite{SFCartan1}.

The domain $\{ \lambda \in C \mid h_4 > h_5 > 0 \}$ of the cylinder $C = L^* \cap \{H = 1\}$ admits a decomposition defined by the energy integral $E$:
\begin{align*}
&\{ \lambda \in C \mid h_4 > h_5 > 0 \} = \cup_{i=1}^8 C_i, \\
&C_1 = E^{-1} (-h_4),  \qquad C_2 = E^{-1} (-h_4, -h_5), \qquad C_3 = E^{-1} (-h_5), \qquad C_4 = E^{-1} (-h_5, h_5),\\ 
&C_5 = E^{-1} (h_5),  \qquad C_6 = E^{-1} (h_5, h_4), \qquad C_7 = E^{-1} (h_4), \qquad C_8 = E^{-1} (h_4, +\infty).
\end{align*}

\onefiglabelsize
{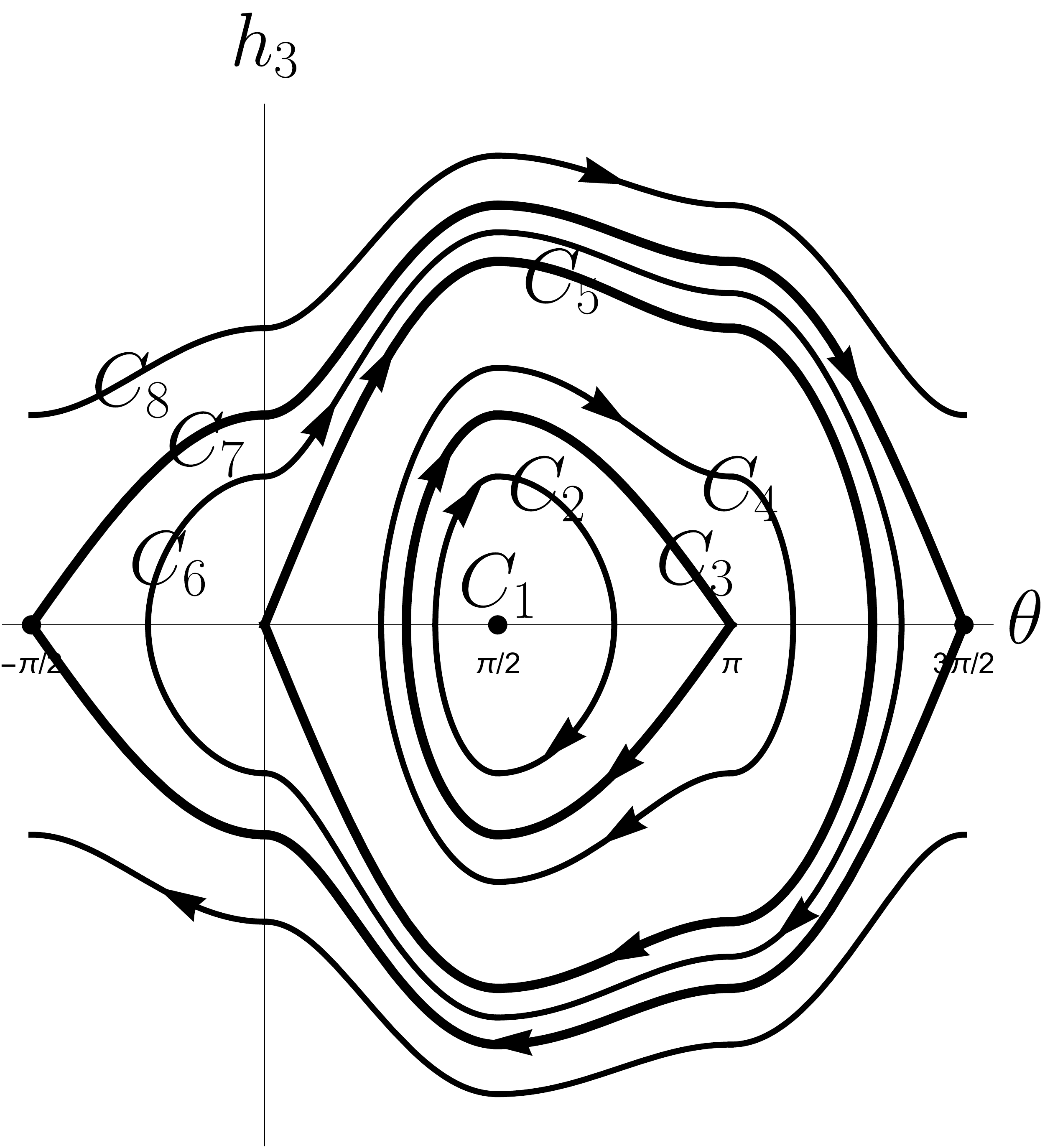}{Phase portrait of system \eq{Hamtheta} in case 1)}{fig:1)h3th}{0.3}

\subsubsection{Case $1)$, level set $C_1$}
Let $h_4 > h_5 > 0$, $E = -h_4$. Then $\theta \equiv \frac{\pi}{2}, h_3 \equiv 0$. Thus $h_1 (\lambda_t) \equiv 0$, $h_2 (\lambda_t) \equiv 1$, i.e., the extremal $\lambda_t$ is $h_1$-singular and not bang-bang.

\subsubsection{Case $1)$, domain $C_2$} \label{subsec:1)I1}
Let $h_4 > h_5 > 0$, $E \in (-h_4, -h_5)$. The corresponding bang-bang extremal is shown in Figs.~\ref{fig:1)I1}, \ref{fig:xy1)C2}.

\twofiglabelsize
{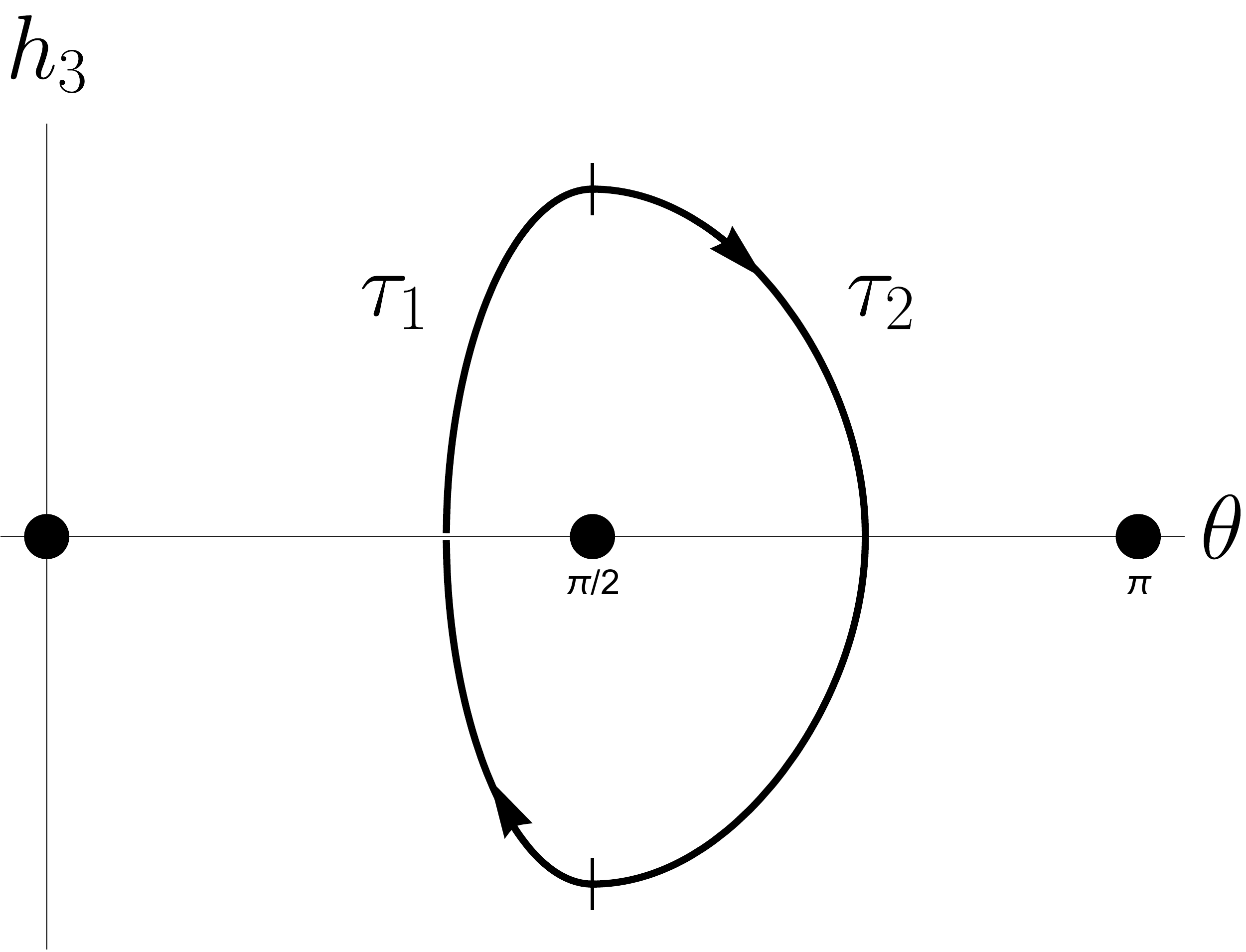}{$(\theta(t), h_3(t))$: Case 1), domain $C_2$}{fig:1)I1}{0.6}
{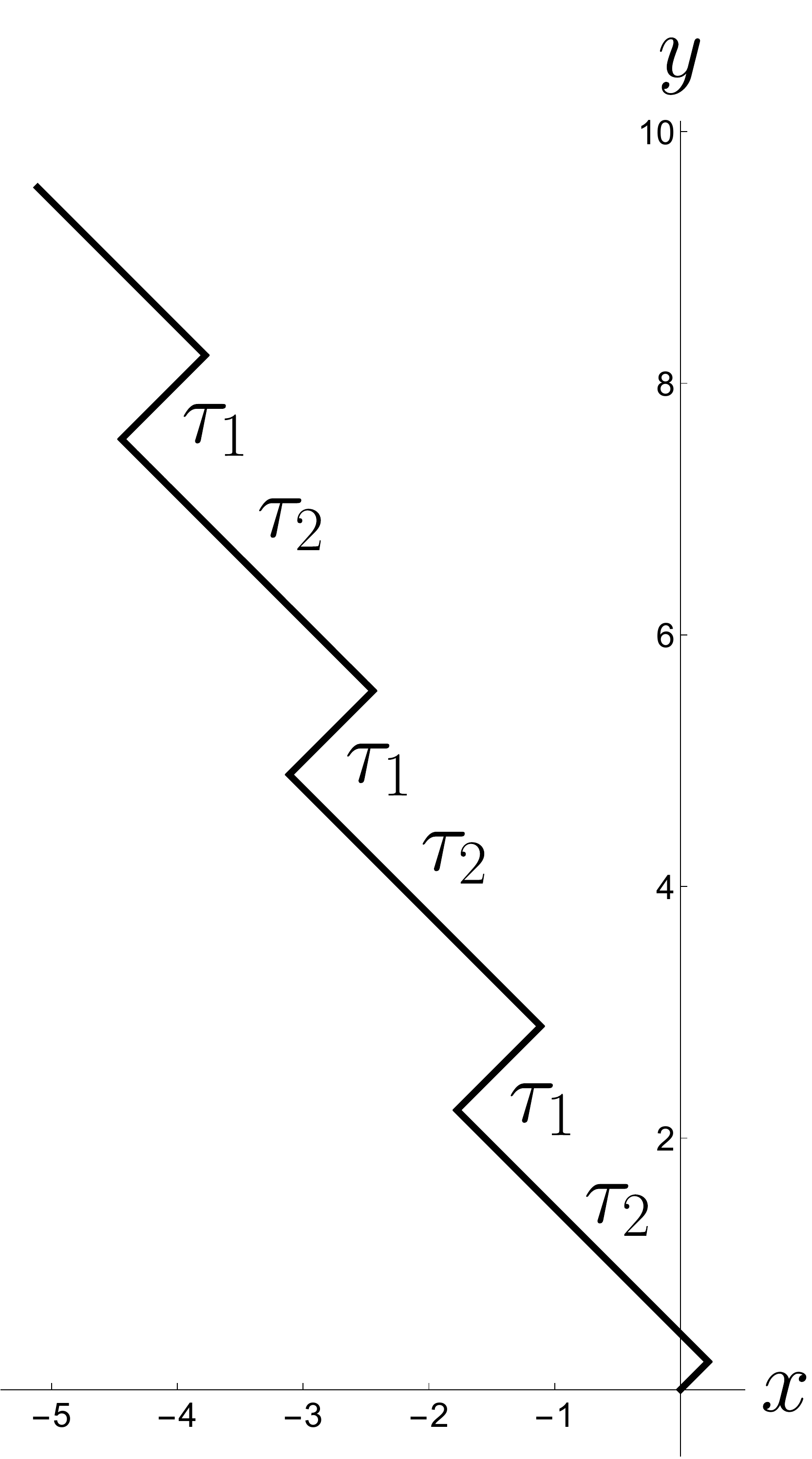}{$(x(t), y(t))$: Case 1), domain $C_2$}{fig:xy1)C2}{0.4}

Denote by $\tau_1, \tau_2$, the time intervals
between successive switchings of control: 
\begin{gather*}
\theta (0) = \frac{\pi}{2}, \quad \theta (t) \in \left(0, \frac{\pi}{2}\right) \quad \textrm{for} \quad t \in (0, \tau_1), \quad \theta (\tau_1) = \frac{\pi}{2}, \\
\theta (0) = \frac{\pi}{2}, \quad \theta (t) \in \left(\frac{\pi}{2}, \pi\right) \quad \textrm{for} \quad t \in (0, \tau_2), \quad \theta (\tau_2) = \frac{\pi}{2}.
\end{gather*}
Compute the interval $\tau_1$ via the Casimirs $E$, $h_4$, $h_5$:
\begin{align*}
&\theta \in \left[0, \frac{\pi}{2}\right] \Rightarrow E = \frac{h_3^2}{2} + h_5 \cos^2 \theta - h_4 \sin^2 \theta, \\
&\theta = \frac{\pi}{2} \Rightarrow E = \frac{h_3^2}{2} - h_4, \quad h_3 = \pm \sqrt{2(E+h_4)},\\
&h_3(0) = -\sqrt{2(E+h_4)}, \quad h_3 (\tau_1) = \sqrt{2(E+h_4)}, \\
&h_3(t) = h_3(0) + (h_4 + h_5)t,\\
&\tau_1 = \frac{h_3(\tau_1) - h_3(0)}{h_4 + h_5} = \frac{2\sqrt{2(E+h_4)}}{h_4 + h_5}.
\end{align*}
We compute similarly $\ds\tau_2 = \frac{2\sqrt{2(E+h_4)}}{h_4 - h_5} > \tau_1$. The rule of switchings of controls $u_i (t) = s_i = \sgn h_i (t)$, $i = 1,2$, follows from the phase portrait in Fig.~\ref{fig:1)I1}. Thus we obtain a general form of
a bang-bang trajectory in the case 1), $C_2$.

\begin{prop}
In case $1)$, domain $C_2: h_4 > h_5 > 0, E \in (-h_4, -h_5)$, a bang-bang control is obtained by choosing a finite segment of the following two-side infinite periodic sequence:
\begin{table}[h]
\begin{center}
\begin{tabular}{cccccccc}
$(u_1, u_2):$ & $\quad$ & $\dots$ & $(+, +)$ & $(-, +)$ & $(+, +)$ & $(-, +)$ & \dots \\
& & & $\tau_1$ & $\tau_2$ & $\tau_1$ & $\tau_2$ &
\end{tabular}
\end{center}
\end{table}

Duration of all segments of constancy of controls (except the first and the last ones) is equal to
$$ \tau_1 = \frac{2\sqrt{2(E+h_4)}}{h_4 + h_5}, \qquad \tau_2 = \frac{2\sqrt{2(E+h_4)}}{h_4 - h_5} > \tau_1.$$
The first and the last segments may take arbitrary values in the corresponding intervals $(0, \tau_i]$.
\end{prop}
The subsequent analysis of the structure of bang-bang trajectories is completely analogous to the preceding one, thus we omit analogous computations and arguments in the following subsubsections.

\subsubsection{Case $1)$, level line $C_3$}
\label{subsec:1)C2}

Let $h_4 > h_5 > 0$, $E = -h_5$, the corresponding extremal is shown in Figs.~\ref{fig:1)C2}, \ref{fig:xy1)C3}.

\twofiglabelsize
{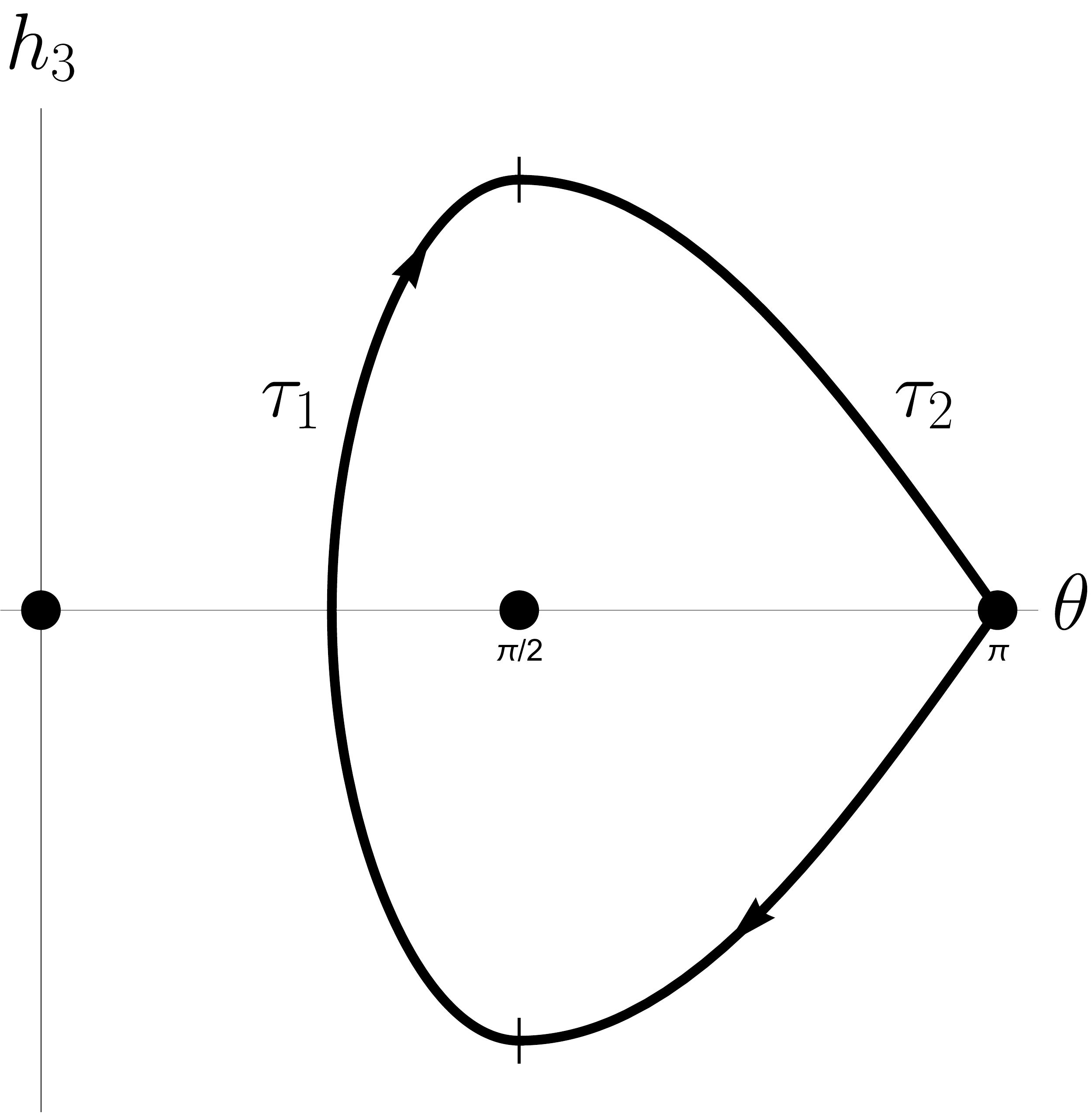}{$(\theta(t), h_3(t))$: Case 1), level line $C_3$}{fig:1)C2}{0.6}
{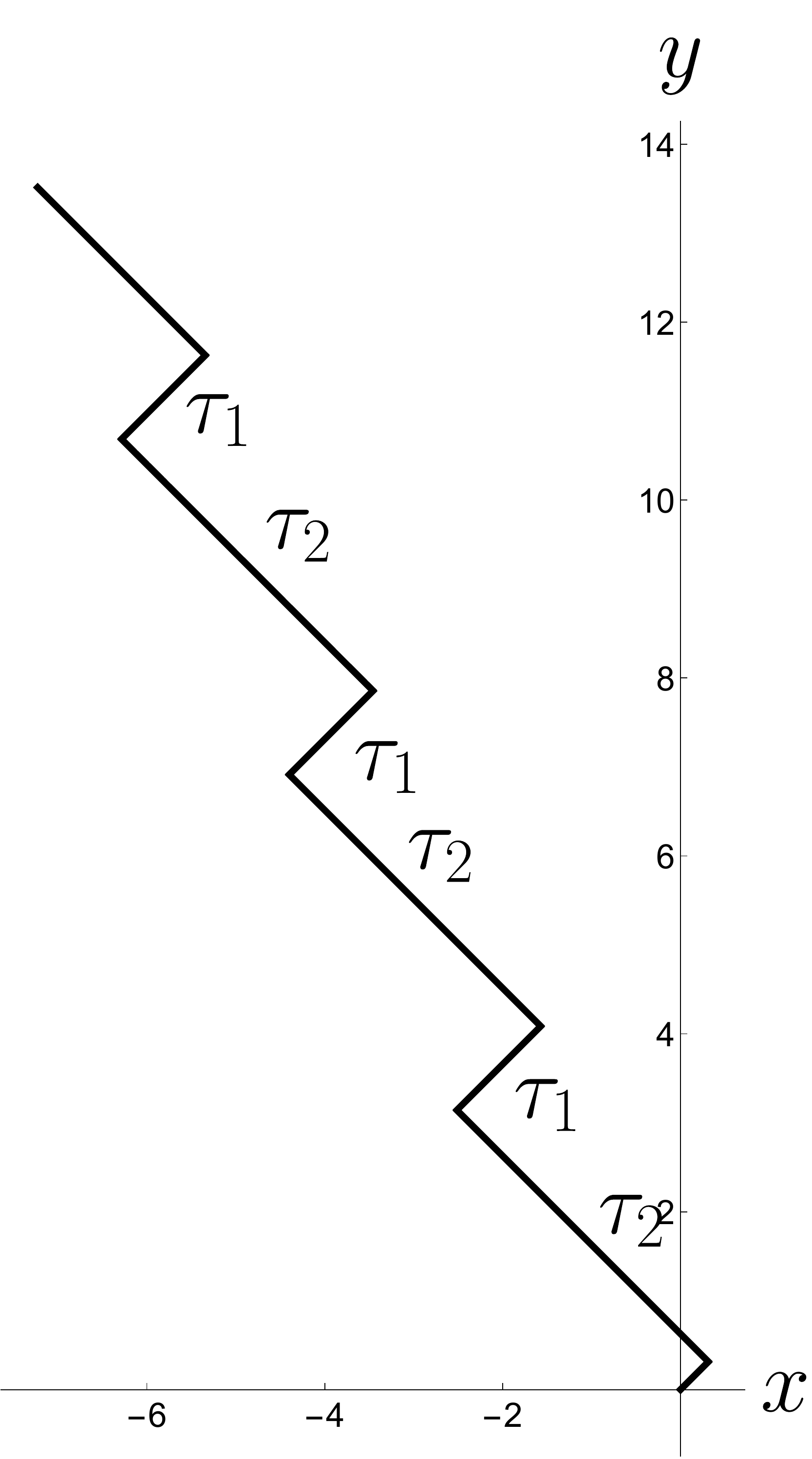}{$(x(t), y(t))$: Case 1), level line $C_3$}{fig:xy1)C3}{0.4}

Then the control has the form 
\begin{table}[H]
\begin{center}
\begin{tabular}{cccccccc}
$(u_1, u_2):$ & $\quad$ & $\dots$ & $(+, +)$ & $(-, +)$ & $(+, +)$ & $(-, +)$ & \dots \\
& & & $\tau_1$ & $\tau_2$ & $\tau_1$ & $\tau_2$ &
\end{tabular}
\end{center}
\end{table}
with
$$ \tau_1 = \frac{2\sqrt{2(h_4 - h_5)}}{h_4 + h_5}, \qquad \tau_2 = \frac{2\sqrt{2(h_4 - h_5)}}{h_4 - h_5} = 2 \sqrt{\frac{2}{h_4 - h_5}} > \tau_1.$$
Notice that despite the fact that $h_2(\lambda_t)$ vanishes when $\theta = \pi, h_3 = 0$, the control $u_2(t)$ does not switch at such points since  $h_2(\lambda_t)$ preserves sign near these points.

\subsubsection{Case $1)$, domain $C_4$}
\label{subsec:1)I2}
Let $h_4 > h_5 > 0, E \in (-h_5, h_5)$, the corresponding trajectory is shown in Figs.~\ref{fig:1)I2}, \ref{fig:xy1)I2}.

\twofiglabelsize
{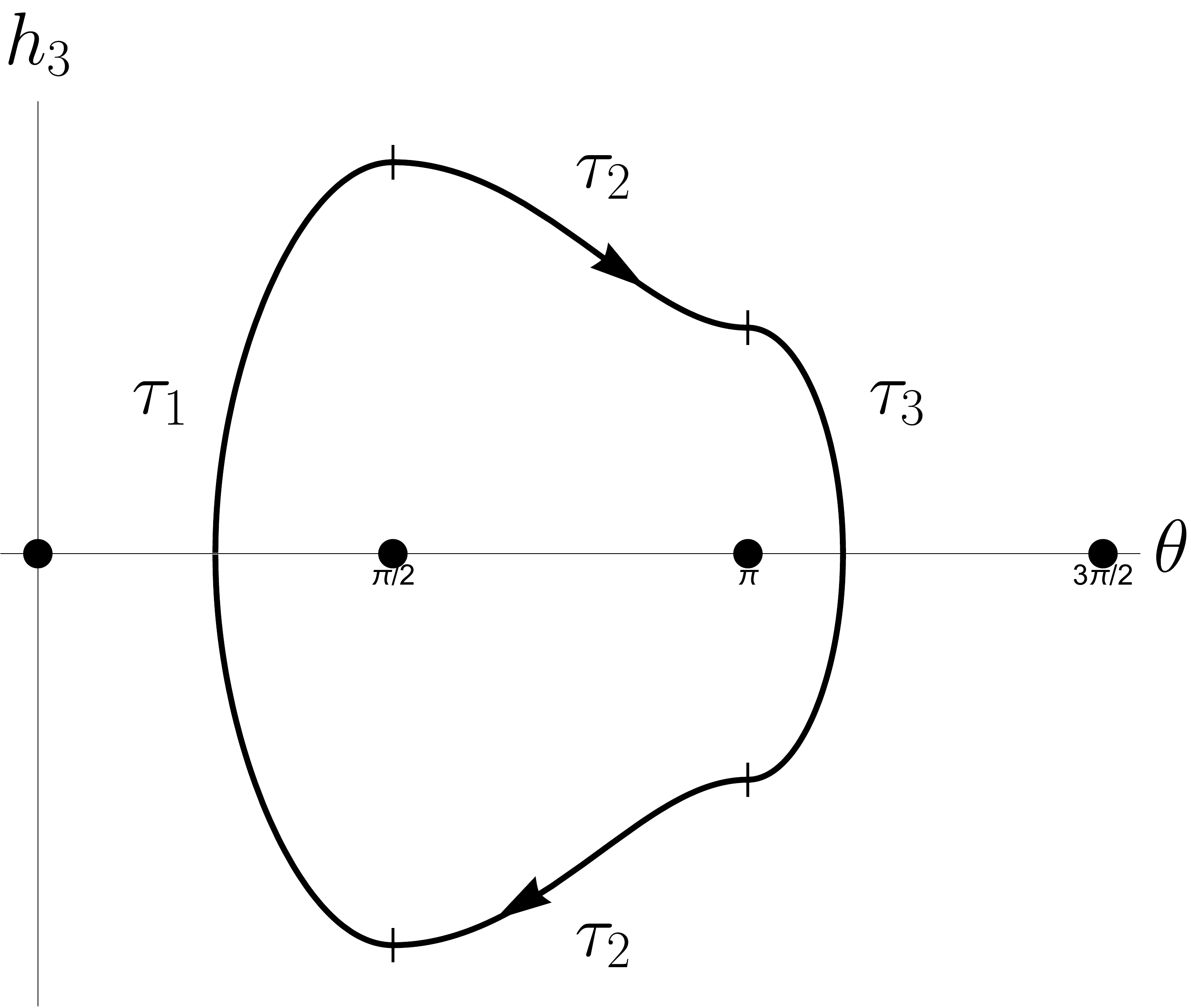}{$(\theta(t), h_3(t))$: Case $1)$, domain $C_4$}{fig:1)I2}{0.7}
{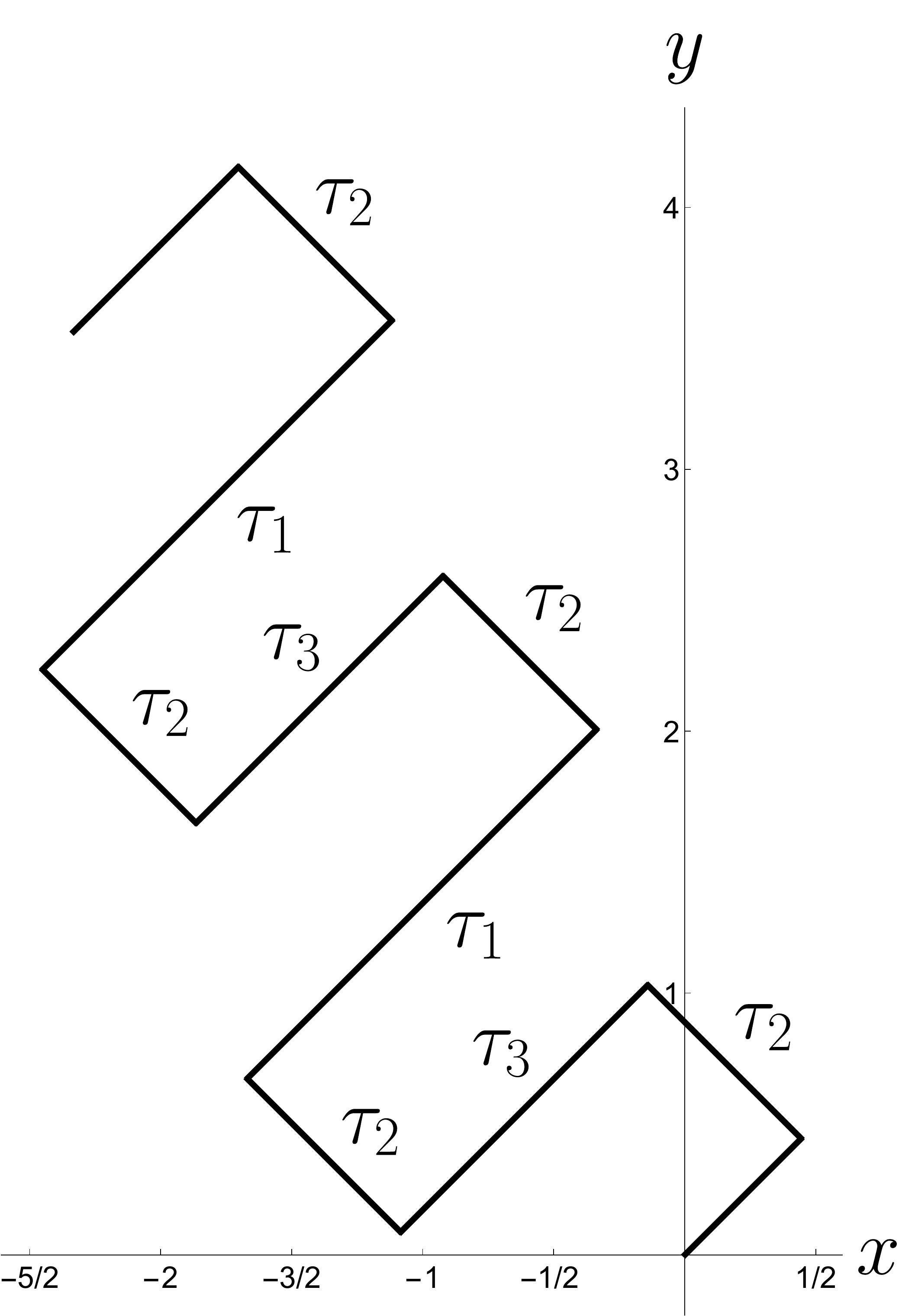}{$(x(t), y(t))$: Case $1)$, domain $C_4$}{fig:xy1)I2}{0.5}

Then the control has the form:
\begin{table}[H]
\begin{center}
\begin{tabular}{ccccccccc}
$(u_1, u_2):$ & $\quad$ & $\dots$ & $(+, +)$ & $(-, +)$ & $(-, -)$ & $(-, +)$ & $(+, +)$ &\dots \\
& & & $\tau_1$ & $\tau_2$ & $\tau_3$ & $\tau_2$ & $\tau_1$ &
\end{tabular}
\end{center}
\end{table}
with
\begin{gather*}
\tau_1 = \frac{2\sqrt{2(E+h_4)}}{h_4 + h_5}, \qquad \tau_2 = \frac{\sqrt{2(E+h_4)} - \sqrt{2(E+h_5)}}{h_4 - h_5} = \frac{2}{\sqrt{2(E+h_4)} + \sqrt{2(E+h_5)}},\\
\tau_3 = \frac{2\sqrt{2(E+h_5)}}{h_4 + h_5} < \tau_1.
\end{gather*}

\subsubsection{Case $1)$, level line $C_5$}
Let $h_4 > h_5 > 0$, $E = h_5$,  the corresponding trajectory is shown in Figs.~\ref{fig:1)C3}, \ref{fig:xy1)C5}.

\twofiglabelsize
{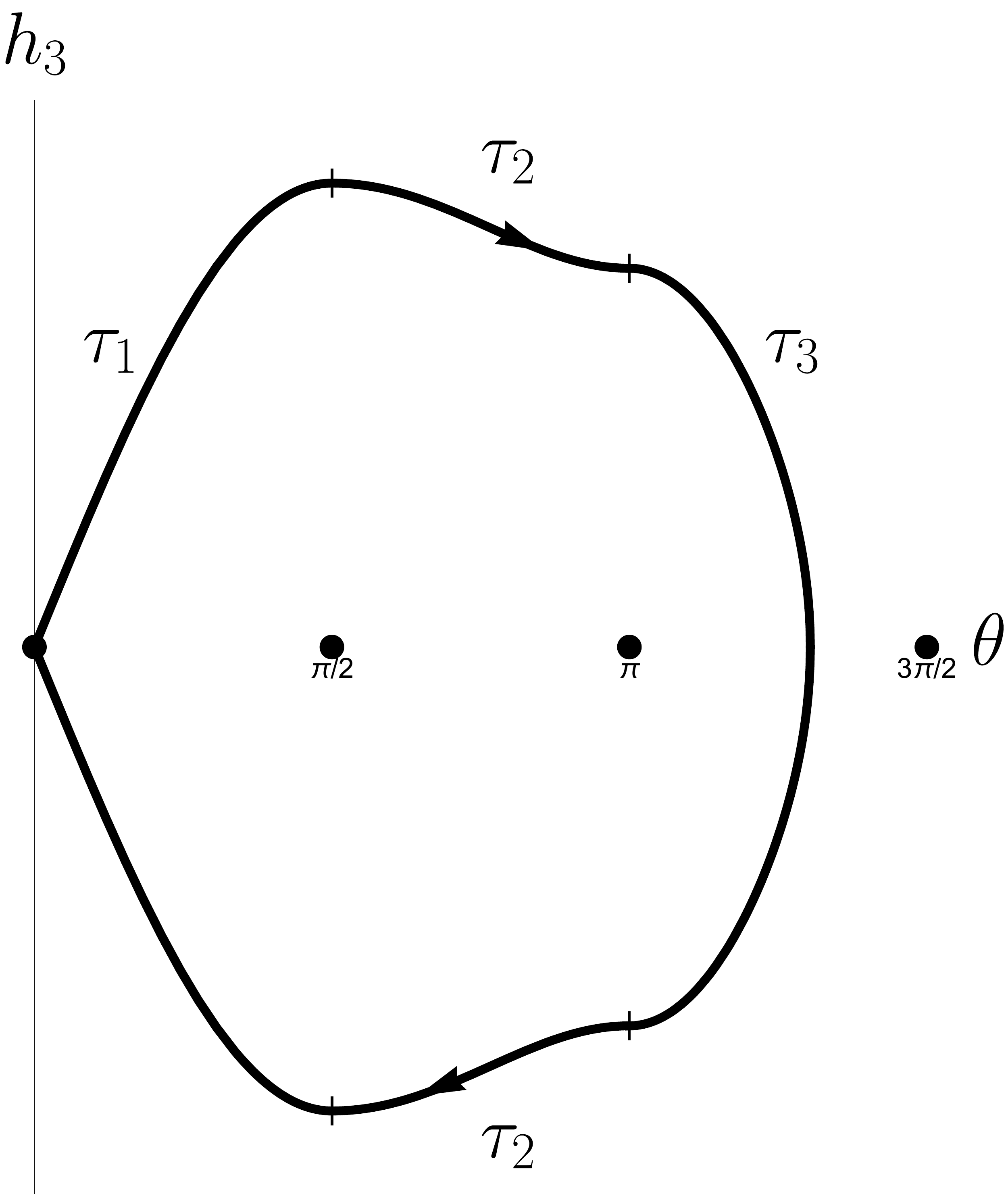}{$(\theta_t, h_3(t))$: Case 1), level line $C_5$}{fig:1)C3}{0.6}
{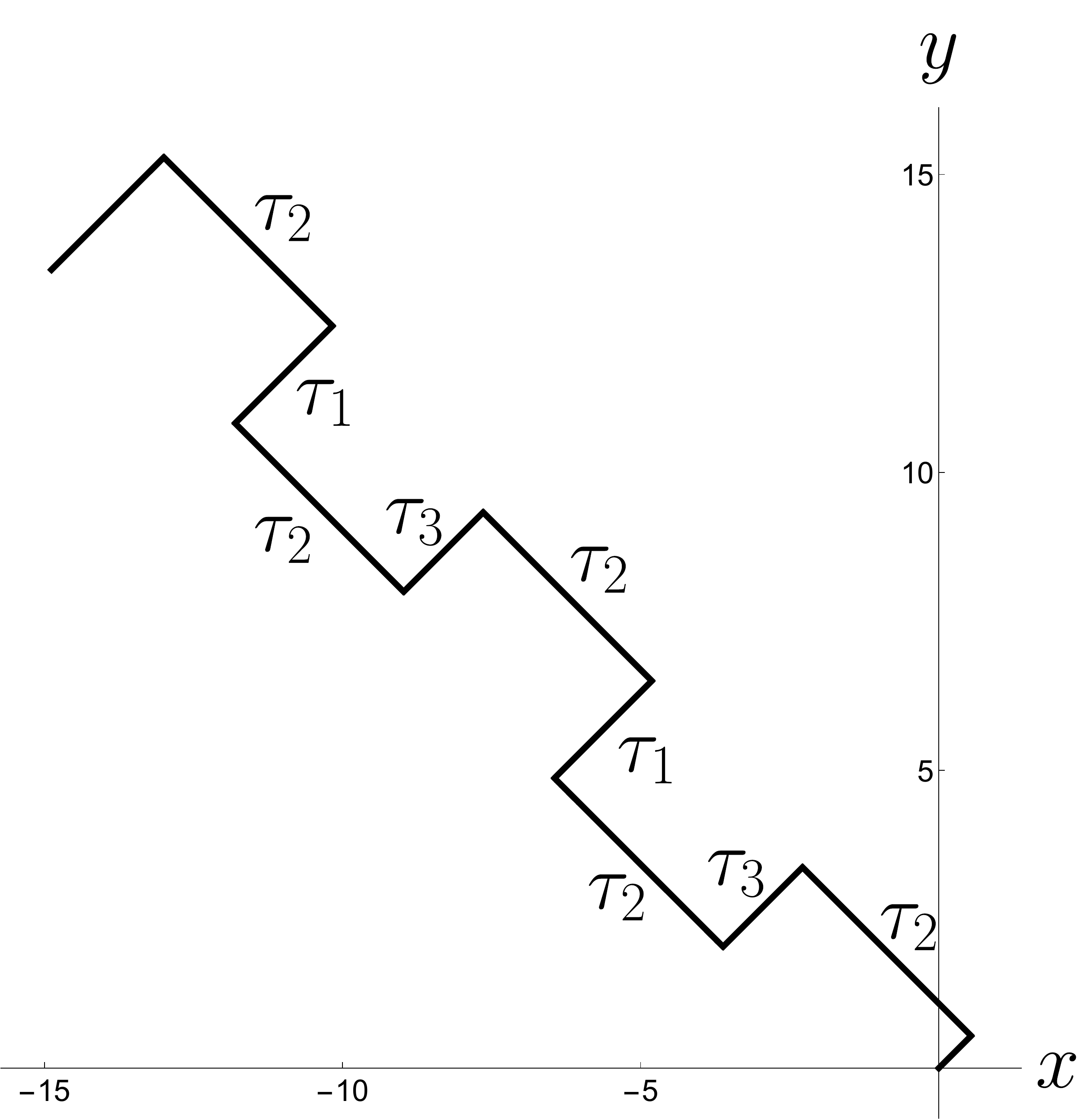}{$(x(t), y(t))$: Case 1), level line $C_5$}{fig:xy1)C5}{0.6}

Then the control has the form:
\begin{table}[H]
\begin{center}
\begin{tabular}{ccccccccc}
$(u_1, u_2):$ & $\quad$ & $\dots$ & $(+, +)$ & $(-, +)$ & $(-, -)$ & $(-, +)$ & $(+, +)$ &\dots \\
& & & $\tau_1$ & $\tau_2$ & $\tau_3$ & $\tau_2$ & $\tau_1$ &
\end{tabular}
\end{center}
\end{table}
with
\begin{gather*}
\tau_1 = 2\sqrt{\frac{2}{h_4 + h_5}}, \qquad  \tau_2 = \frac{\sqrt{2(h_4 + h_5)} - 2\sqrt{h_5}}{h_4 - h_5}, \qquad \tau_3 = \frac{4\sqrt{h_5}}{h_4 + h_5}.
\end{gather*}
The control $u_1 (t)$ does not switch when $\theta = 0$, $h_3 = 0$.

\subsubsection{Case $1)$, domain $C_6$}
Let $h_4 > h_5 > 0$, $E \in (h_5, h_4)$,  the corresponding trajectory is shown in Figs.~\ref{fig:1)I3}, \ref{fig:xy1)C6}.

\twofiglabelsize
{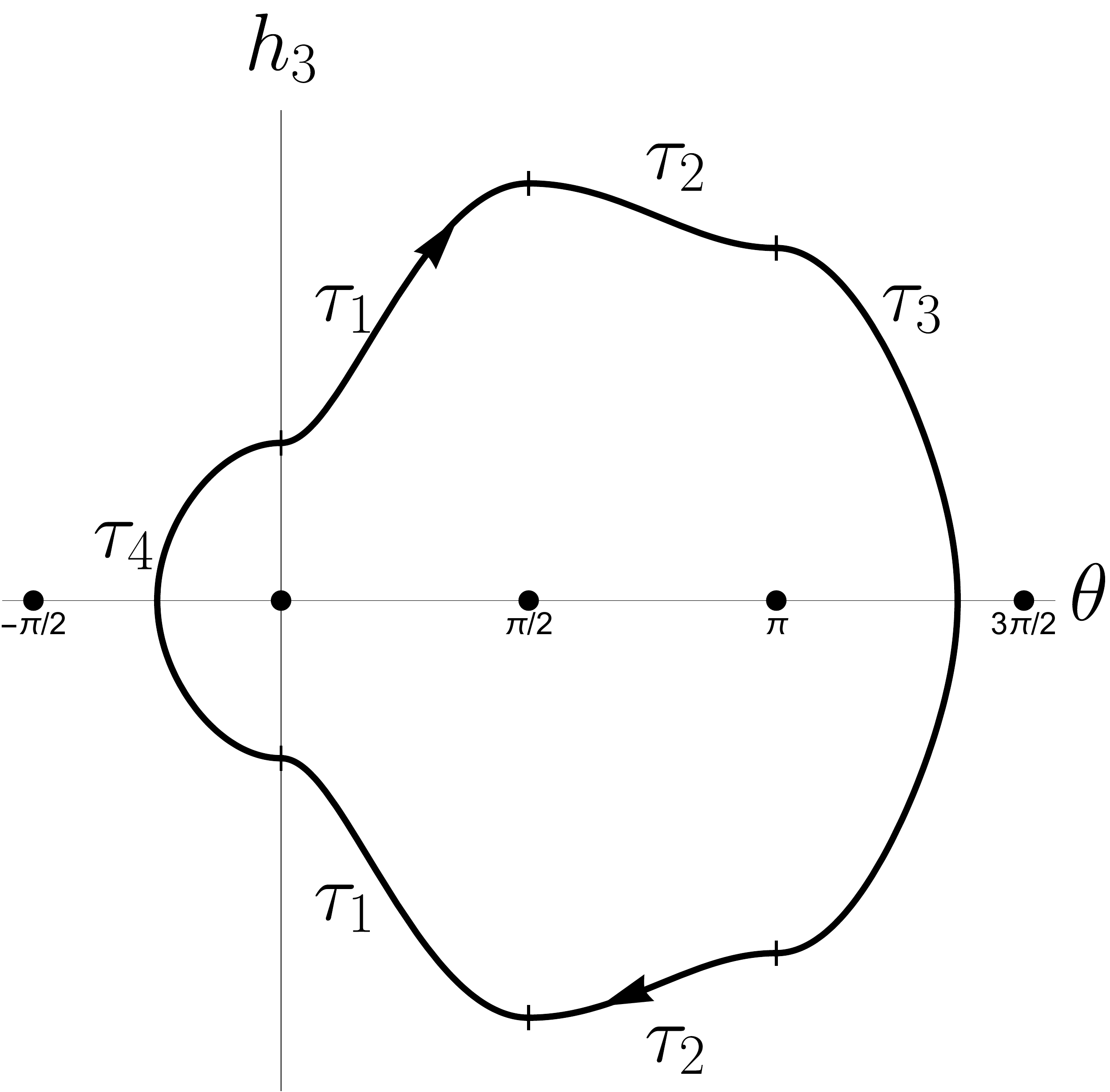}{$(\theta(t), h_3(t))$: Case 1), domain $C_6$}{fig:1)I3}{0.6}
{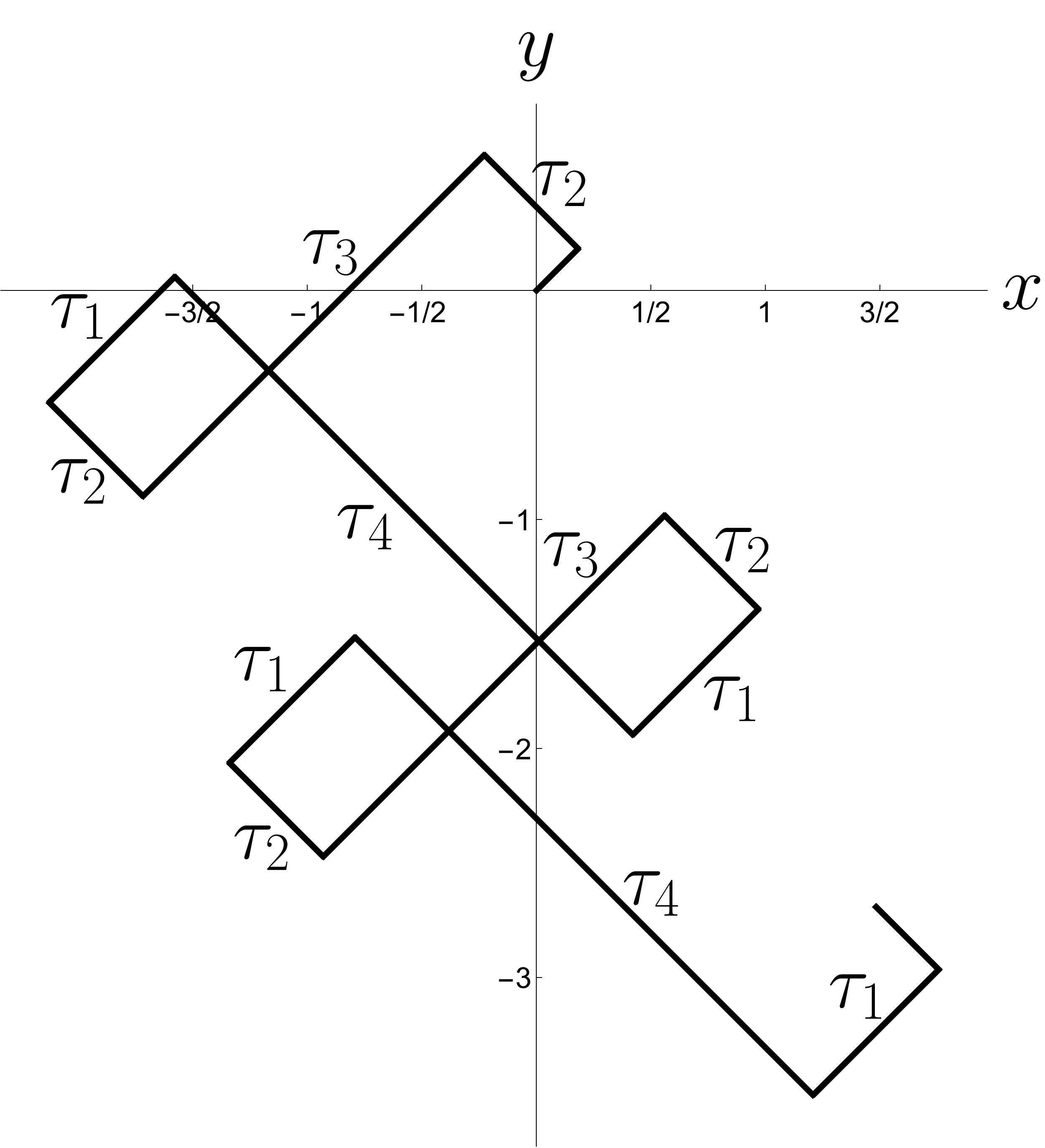}{$(x(t), y(t))$: Case 1), domain $C_6$}{fig:xy1)C6}{0.6}

Then the control has the form:
\begin{table}[H]
\begin{center}
\begin{tabular}{ccccccccccc}
$(u_1, u_2):$ & $\quad$ & $\dots$ & $(+, +)$ & $(-, +)$ & $(-, -)$ & $(-, +)$ & $(+, +)$ & $(+, -)$ & $(+, +)$ &\dots \\
& & & $\tau_1$ & $\tau_2$ & $\tau_3$ & $\tau_2$ & $\tau_1$ & $\tau_4$ & $\tau_1$&
\end{tabular}
\end{center}
\end{table}

with
\begin{gather*}
\tau_1 = \frac{\sqrt{2(E + h_4)} - \sqrt{2(E - h_5)}}{h_4 + h_5} = \frac{2}{\sqrt{2(E+h_4)} + \sqrt{2(E-h_5)}}, \\
\tau_2 = \frac{\sqrt{2(E + h_4)} - \sqrt{2(E + h_5)}}{h_4 - h_5} = \frac{2}{\sqrt{2(E+h_4)} + \sqrt{2(E+h_5)}}, \\
\tau_3 = \frac{2\sqrt{2(E + h_5)}}{h_4 + h_5}, \quad \tau_4 = \frac{2\sqrt{2(E - h_5)}}{h_4 - h_5}.
\end{gather*}

\subsubsection{Case $1)$, level line $C_7$}
Let $h_4 > h_5 > 0$, $E = h_4$,  the corresponding phase portrait is shown in Fig.~\ref{fig:1)C4}.

\twofiglabelsize
{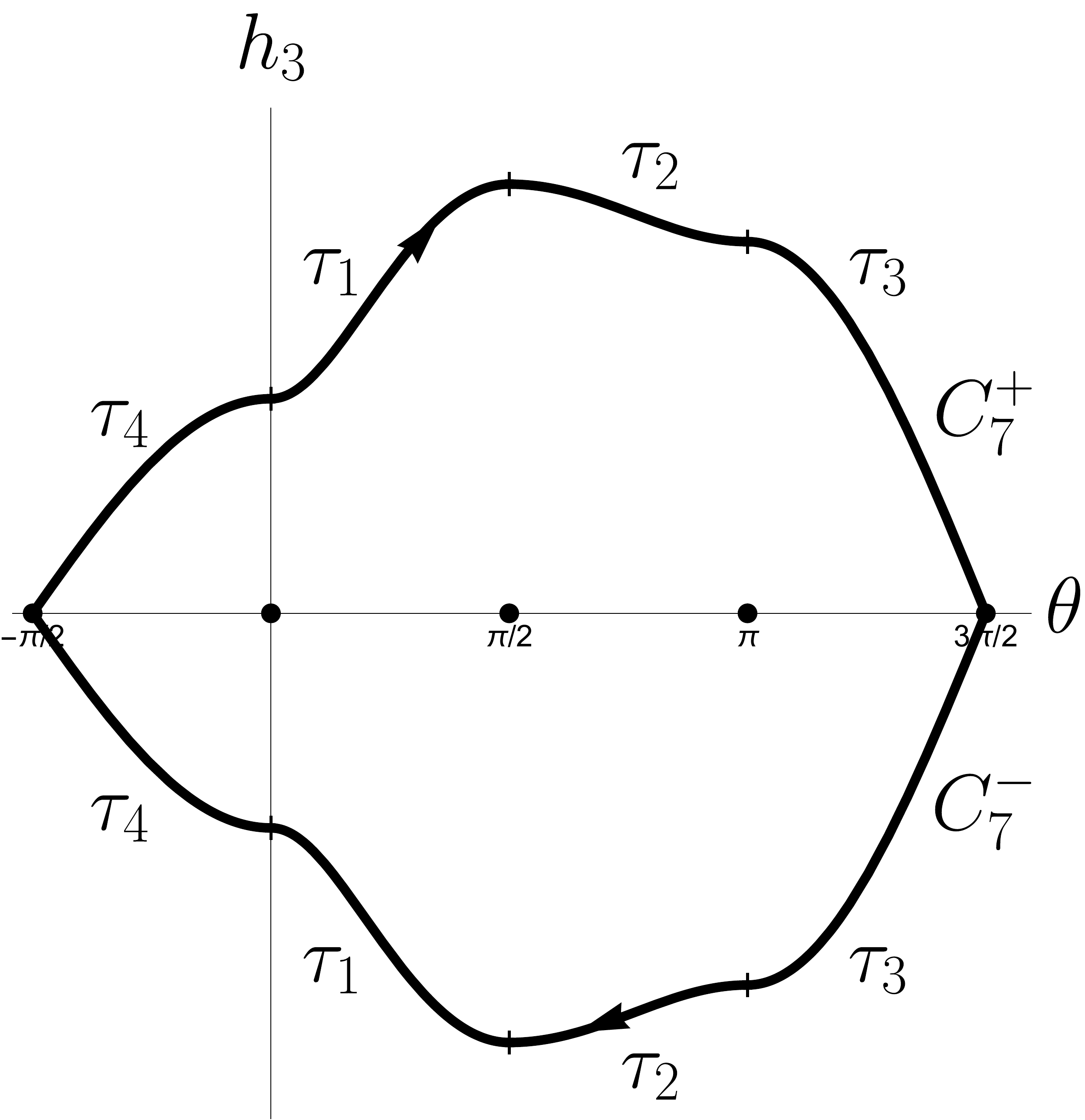}{$(\theta(t), h_3(t))$: Case 1), level line $C_7$}{fig:1)C4}{0.6}
{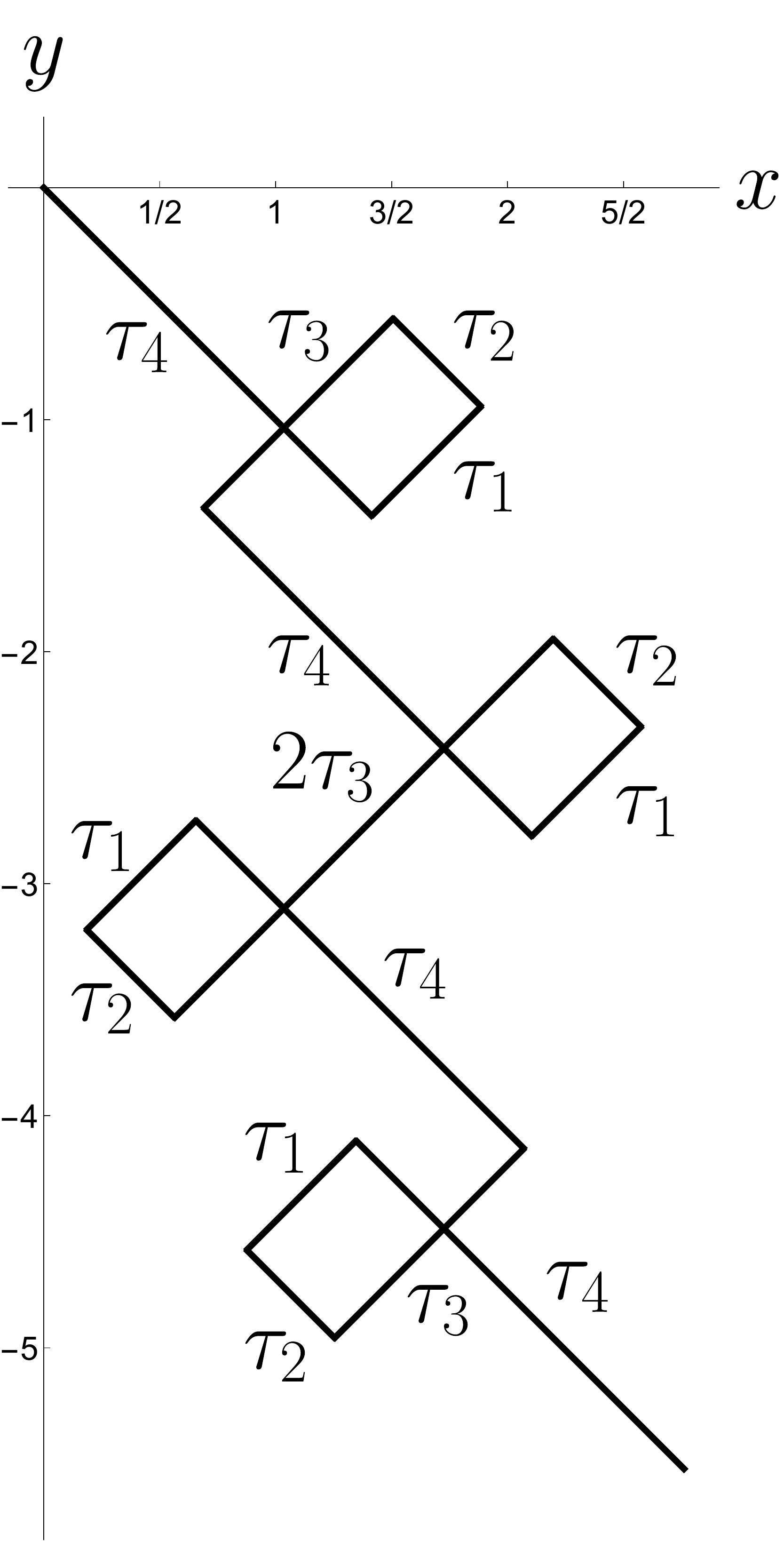}{$(x(t), y(t))$:  Case 1), level line $C_7$}{fig:xy1)C7}{0.4}

The  level line $C_7$ is defined in the domains $\{ \theta \in (-\frac{\pi}{2}, 0) \}$ and $\{ \theta \in (\pi, \frac{3\pi}{2}) \}$ by the equations 
$$h_3 = \pm \sqrt{2(h_4 - h_5)} \cos \theta \text{ and } h_3 = \pm \sqrt{2(h_4 + h_5)} \cos 2\theta.
$$ 
Thus the level line $C_7$ is homeomorphic to figure 8, with self-intersection at the point 
$(\theta, h_3) = (\frac{3\pi}{2}, 0)$. The intersections $C_7^+ = C_7 \cap \{ h_3 \ge 0 \}, C_7^- = C_7 \cap \{ h_3 \le 0 \}$ are continuous curves homeomorphic to $S^1$, with the only singularity --- the corner point $(\theta, h_3) = (\frac{3\pi}{2}, 0)$.

Each bang-bang control is obtained by choosing a finite segment from the following infinite periodic graph:
\medskip
$$
\xymatrix{
&&\tau_1 & \tau_2 & \tau_3 & \tau_4 & \tau_1  & \\
C_7^+ & \cdots \ar@{>}[r] & (+,+) \ar@{>}[r]& (-,+) \ar@{>}[r]& (-,-)\ar@{>}[r] \ar@{>}[dr]& (+,-)\ar@{>}[r]& (+,+)\ar@{>}[r]& \cdots  \\
C_7^- & \cdots \ar@{>}[r] & (-,+) \ar@{>}[r]& (+,+) \ar@{>}[r]& (+,-) \ar@{>}[r] \ar@{>}[ur] & (-,-) \ar@{>}[r]& (-,+)\ar@{>}[r]&\cdots  \\
&&\tau_2 & \tau_1 & \tau_4 & \tau_3 & \tau_2  & 
}
$$
\medskip

We have
\begin{gather*}
\tau_1 = \frac{2\sqrt{h_4} - \sqrt{2(h_4 - h_5)}}{h_4 + h_5}, \qquad \tau_2 = \frac{2\sqrt{h_4} - \sqrt{2(h_4 + h_5)}}{h_4 - h_5}, \\
\tau_3 =  \sqrt{\frac{2}{h_4 + h_5}}, \qquad  \tau_4 =  \sqrt{\frac{2}{h_4 - h_5}}.
\end{gather*}

The  curves $(x(t), y(t))$ corresponding to the curves $C_7^+$ and $C_7^-$ are shown in Figs.~\ref{fig:xy1)C7+} and  \ref{fig:xy1)C7-} respectively. An example of curve $(x(t), y(t))$ corresponding to two curves $C_7^+$ and two curves  $C_7^-$ is given in Fig.~\ref{fig:xy1)C7}.

\twofiglabelsize
{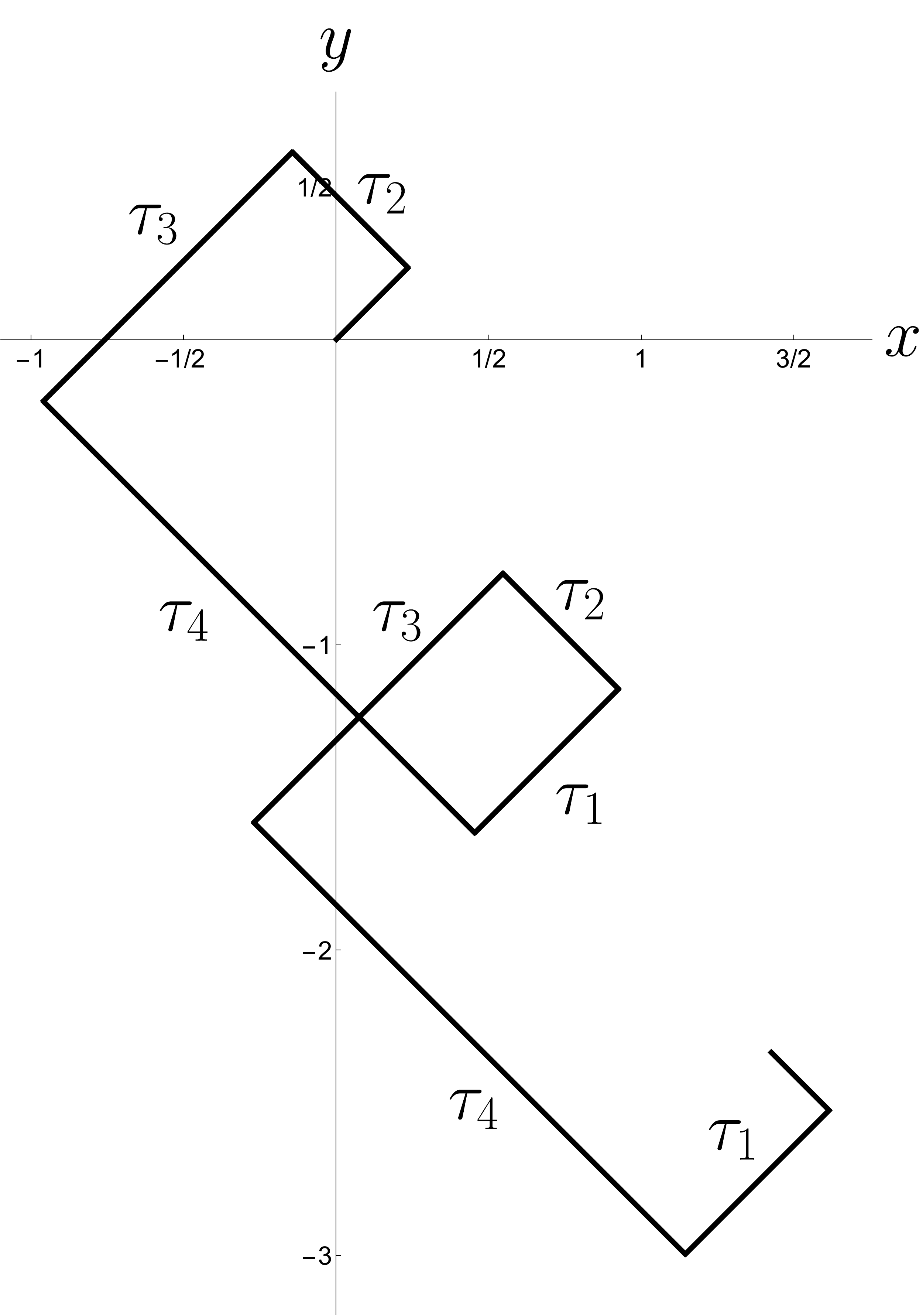}{$(x(t), y(t))$: Case 1), curve $C_7^+$}{fig:xy1)C7+}{0.6}
{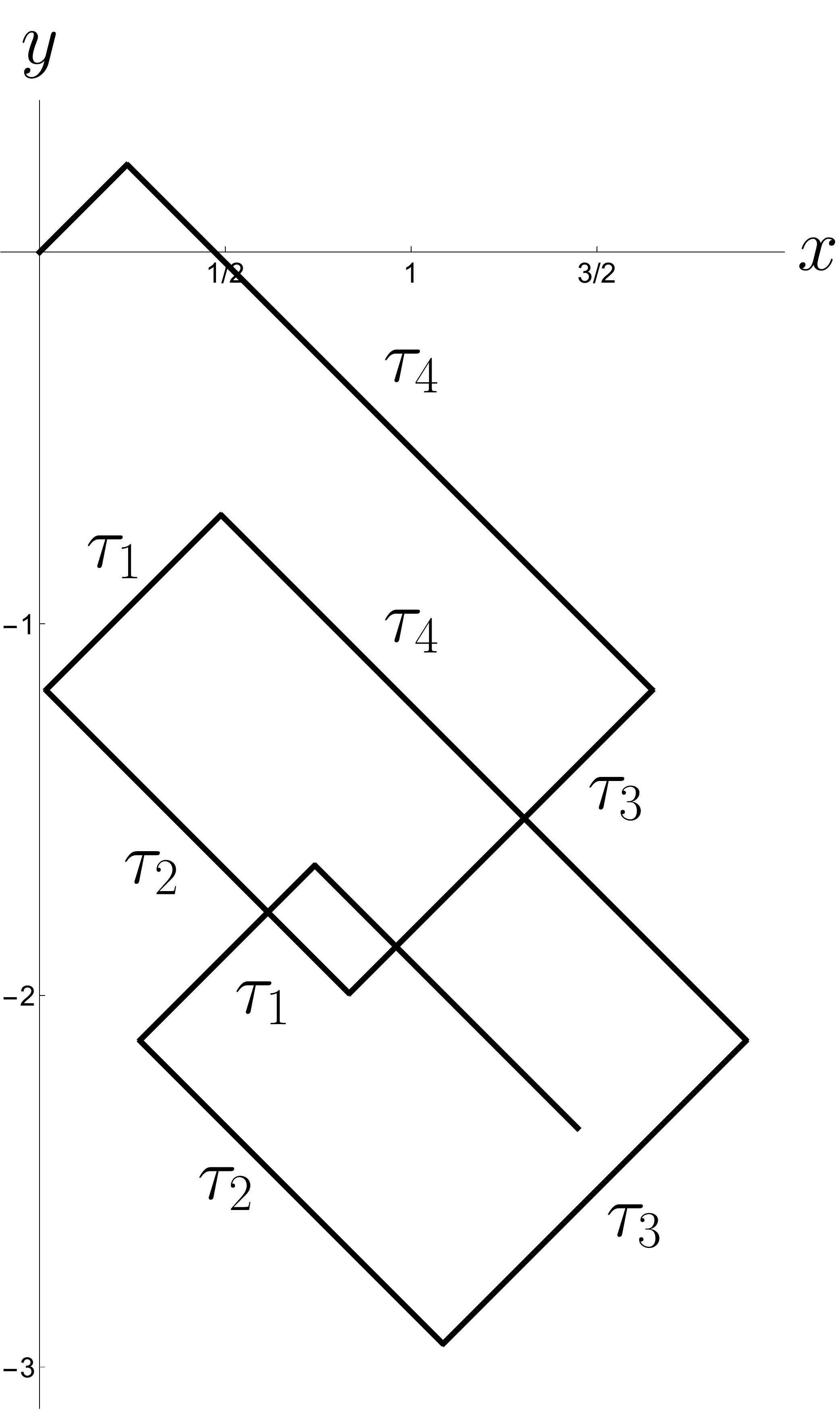}{$(x(t), y(t))$:  Case 1), curve $C_7^-$}{fig:xy1)C7-}{0.5}

\subsubsection{Case $1)$, domain $C_8$}
Let $h_4 > h_5 > 0$, $E > h_4$,  the corresponding trajectory is shown in Figs.~\ref{fig:1)C8}, \ref{fig:xy1)C8}.

\twofiglabelsize
{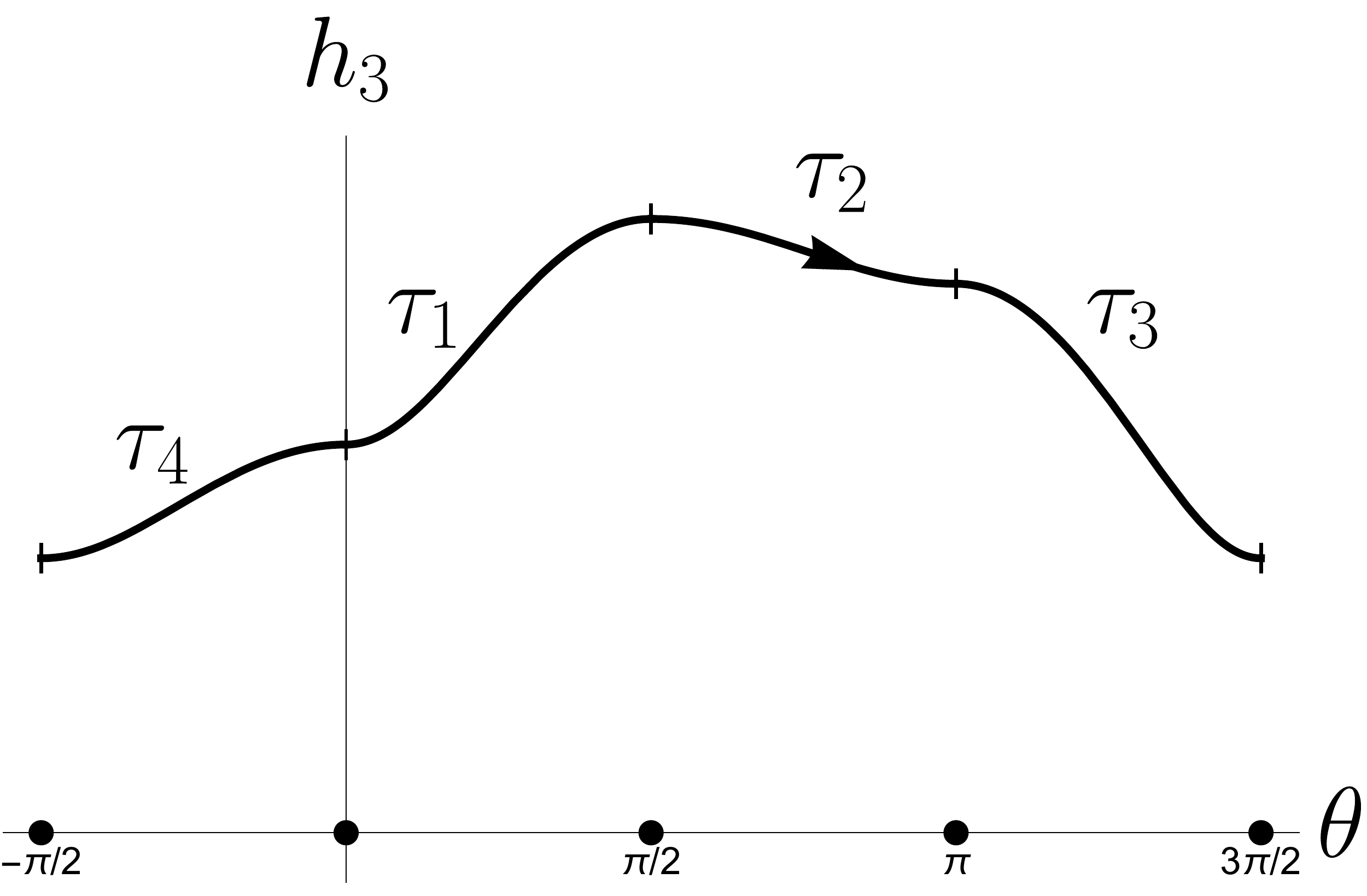}{$(\theta(t), h_3(t))$: Case 1), domain $C_8$}{fig:1)C8}{0.6}
{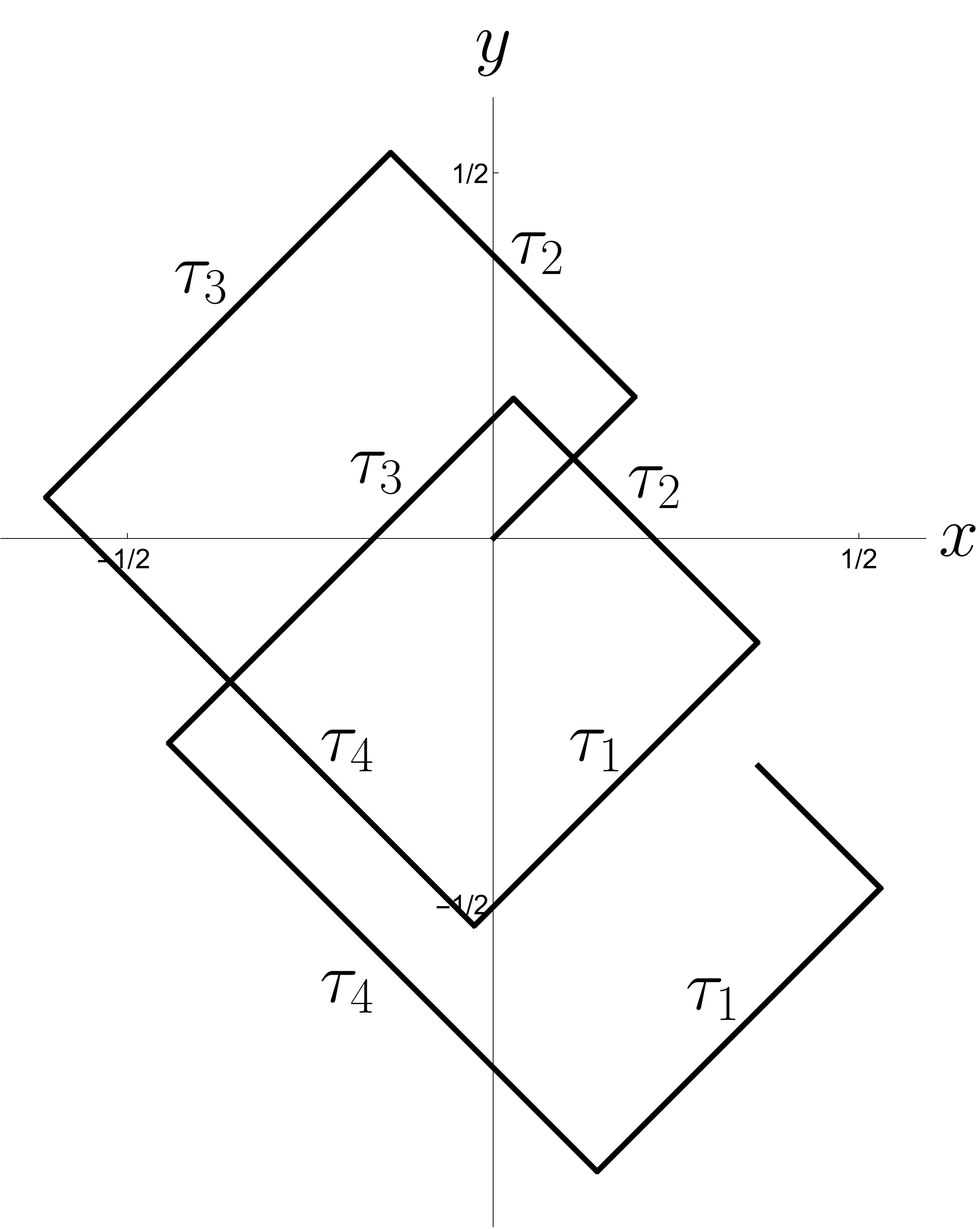}{$(x(t), y(t))$: Case 1), domain $C_8$}{fig:xy1)C8}{0.6}

In the case $h_3 > 0$ the bang-bang control has the form 
\begin{table}[H]
\begin{center}
\begin{tabular}{ccccccccc}
$(u_1, u_2):$ & $\quad$ & $\dots$ & $(+, +)$ & $(-, +)$ & $(-, -)$ & $(+, -)$ & $(+, +)$ &\dots \\
& & & $\tau_1$ & $\tau_2$ & $\tau_3$ & $\tau_4$ & $\tau_1$ &
\end{tabular}
\end{center}
\end{table}
In the case $h_3 < 0$ the order of switchings is opposite. We have
\begin{gather*}
\tau_1 = \frac{\sqrt{2(E + h_4)} - \sqrt{2(E - h_5)}}{h_4 + h_5} = \frac{2}{\sqrt{2(E+h_4)} + \sqrt{2(E-h_5)}}, \\
\tau_2 = \frac{\sqrt{2(E + h_4)} - \sqrt{2(E + h_5)}}{h_4 - h_5} = \frac{2}{\sqrt{2(E+h_4)} + \sqrt{2(E+h_5)}}, \\
\tau_3 = \frac{\sqrt{2(E + h_5)} - \sqrt{2(E - h_4)}}{h_4 + h_5} =  \frac{2}{\sqrt{2(E+h_5)} + \sqrt{2(E-h_4)}}, \\
\tau_4 = \frac{\sqrt{2(E - h_5)} - \sqrt{2(E - h_4)}}{h_4 - h_5} = \frac{2}{\sqrt{2(E-h_5)} + \sqrt{2(E-h_4)}}.
\end{gather*}

\subsection{Case $2)$}
Let $h_4 > h_5 = 0$. Then system~\eq{Hamtheta} has the phase portrait given in Fig.~\ref{fig:2)h3th}, see Subsubsec.~7.2.2~\cite{SFCartan1}.

\onefiglabelsize
{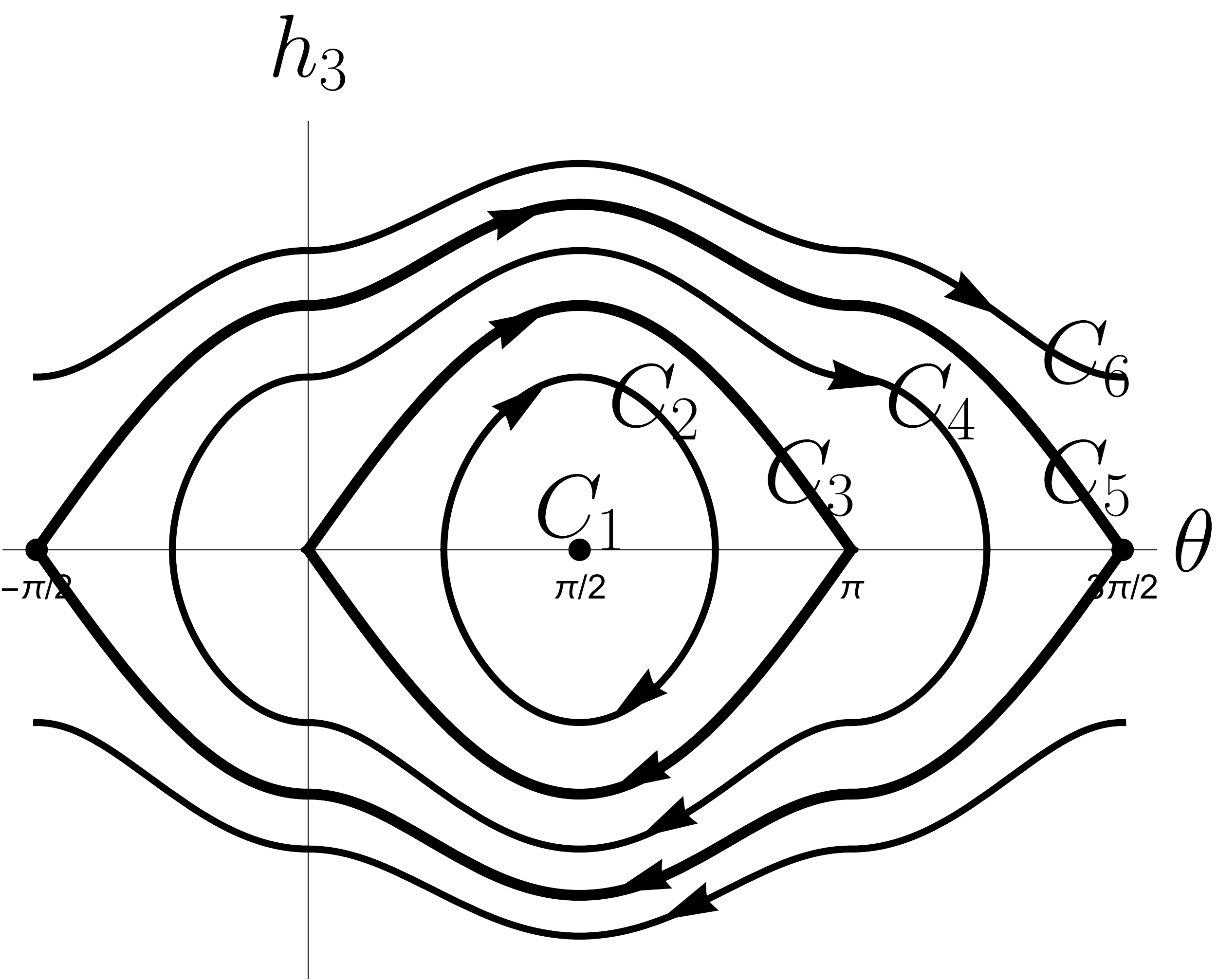}{Phase portrait of system \eq{Hamtheta} in case 2)}{fig:2)h3th}{0.3}

The domain $\{ \lambda \in C \mid h_4 > h_5 = 0 \}$ of the cylinder $C = L^* \cap \{H = 1\}$ admits a decomposition defined by the energy integral $E$:
\begin{align*}
&\{ \lambda \in C \mid h_4 > h_5 = 0 \} = \cup_{i=1}^6 C_i,\\
& C_1 = E^{-1} (-h_4), \qquad C_2 = E^{-1} (-h_4, 0), \qquad C_3 = E^{-1} (0), \\
&C_4 = E^{-1} (0, h_4), \qquad C_5 = E^{-1} (h_4),  \qquad C_6 = E^{-1} (h_4, +\infty). 
\end{align*}

\subsubsection{Case $2)$, level set $C_1$}
Let $h_4 > h_5 = 0$, $E = -h_4$. Then $\theta \equiv \frac{\pi}{2}$, $h_3 \equiv 0$, and the corresponding trajectory is $h_1$-singular.

\subsubsection{Case $2)$, domain $C_2$}
\label{subsec:2)I1}
We have $h_4 > h_5 = 0$, $E \in (-h_4, 0)$,  the corresponding trajectory is shown in Figs.~\ref{fig:2)I1}, \ref{fig:xy2)C2}.

\twofiglabelsize
{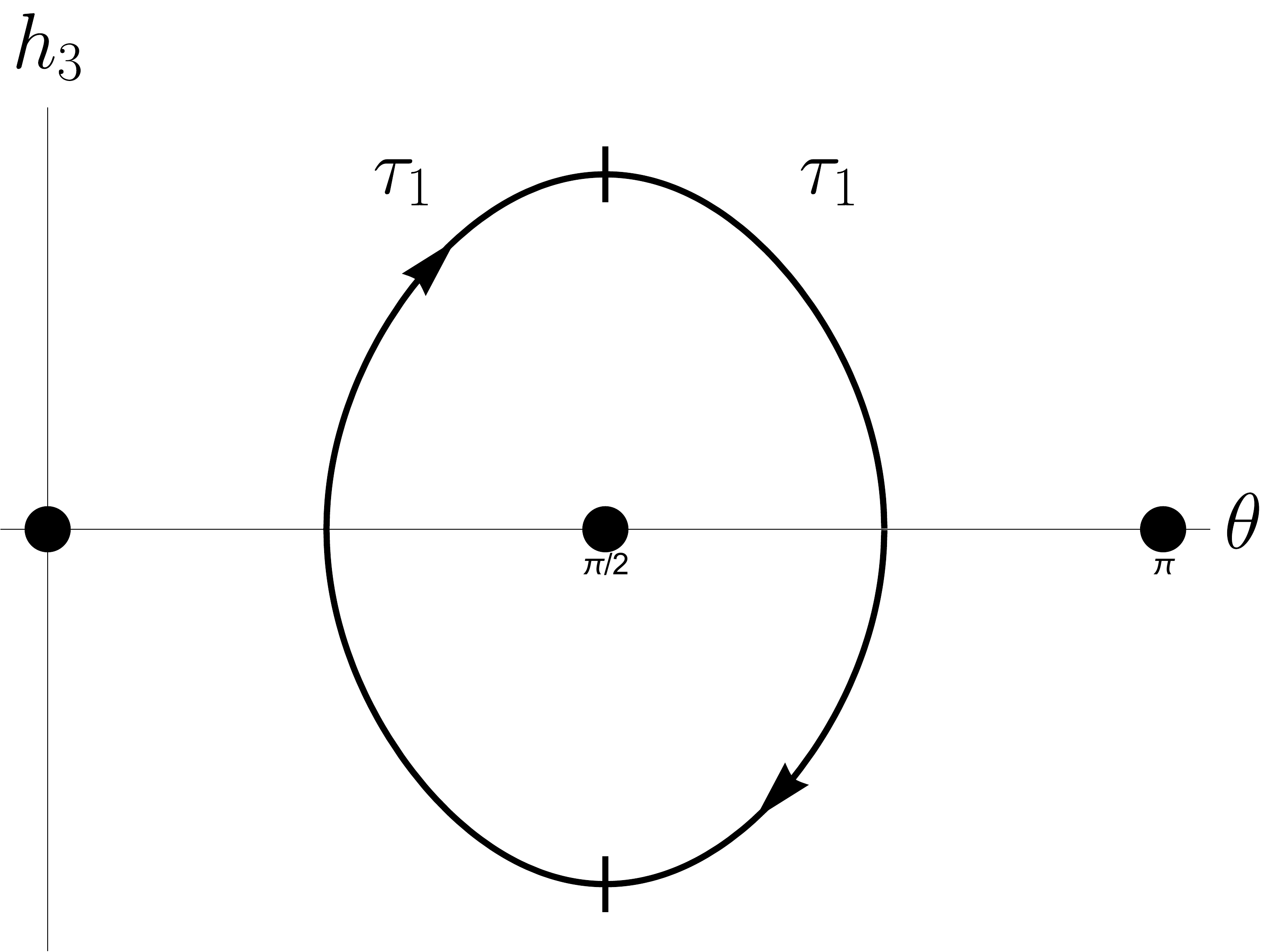}{$(\theta(t), h_3(t))$: Case 2), domain $C_2$}{fig:2)I1}{0.7}
{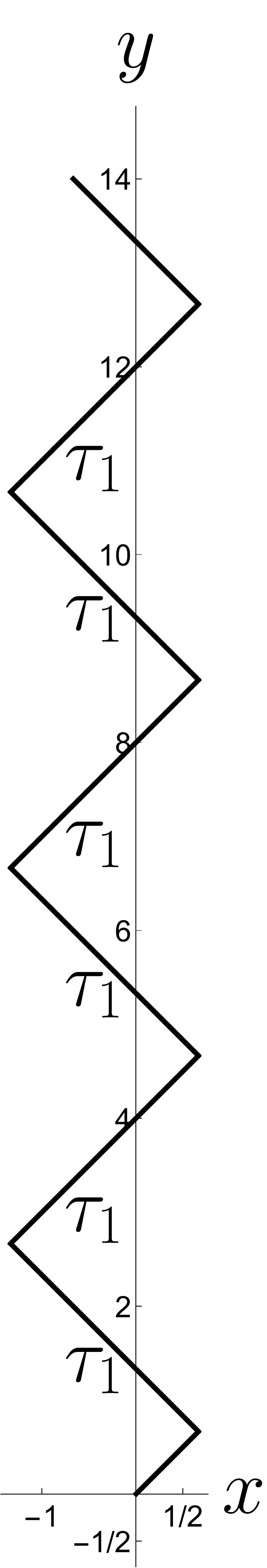}{$(x(t), y(t))$: Case 2), domain $C_2$}{fig:xy2)C2}{0.15}

The control has the form
\begin{table}[H]
\begin{center}
\begin{tabular}{ccccccc}
$(u_1, u_2):$ & $\quad$ & $\dots$ & $(+, +)$ & $(-, +)$ & $(+, +)$ & \dots \\
& & & $\tau_1$ & $\tau_1$ & $\tau_1$ & 
\end{tabular}
\end{center}
\end{table}
with
\begin{gather*}
\tau_1 = \frac{2\sqrt{2(E + h_4)}}{h_4}. 
\end{gather*}

\subsubsection{Case $2)$, level line $C_3$}
\label{subsec:2)C2}
We have $h_4 > h_5 = 0$, $E = 0$,  the corresponding trajectory is shown in Figs.~\ref{fig:2)C2}, \ref{fig:xy2)C3}.

\twofiglabelsize
{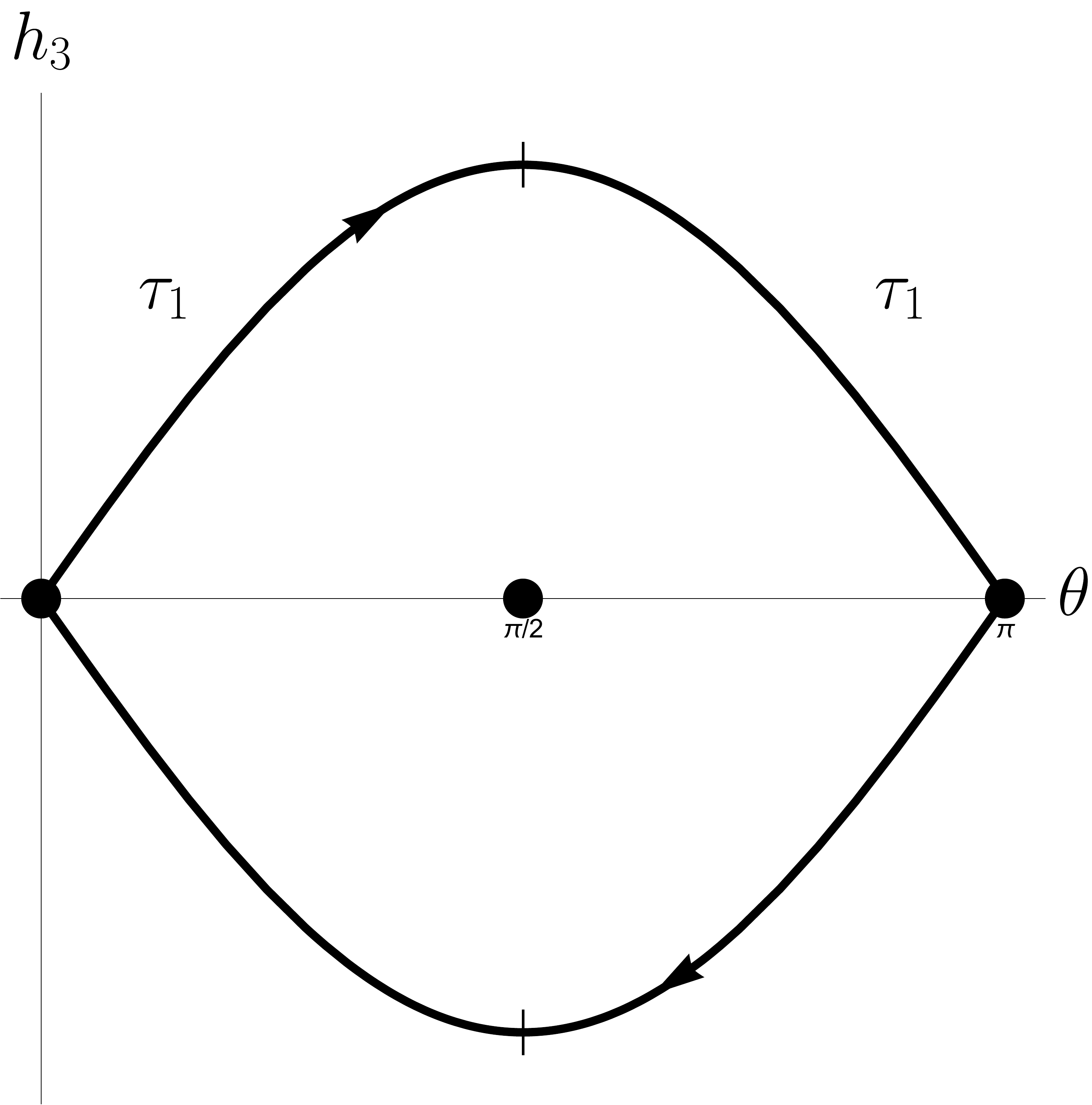}{$(\theta(t), h_3(t))$: Case 2), level line $C_3$}{fig:2)C2}{0.7}
{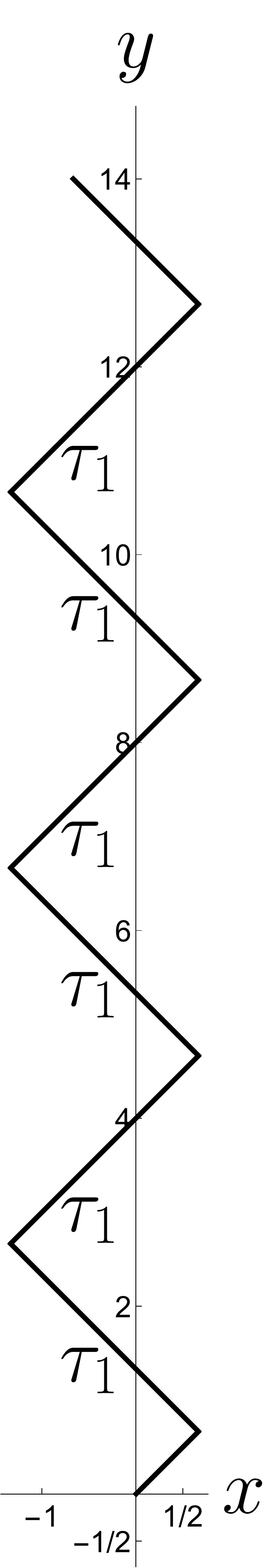}{$(x(t), y(t))$: Case 2), level line $C_3$}{fig:xy2)C3}{0.15}

The control has the form
\begin{table}[H]
\begin{center}
\begin{tabular}{ccccccc}
$(u_1, u_2):$ & $\quad$ & $\dots$ & $(+, +)$ & $(-, +)$ & $(+, +)$ & \dots \\
& & & $\tau_1$ & $\tau_1$ & $\tau_1$ & 
\end{tabular}
\end{center}
\end{table}
with
\begin{gather*}
\tau_1 = 2 \sqrt{\frac{2}{h_4}}.
\end{gather*}

\subsubsection{Case $2)$, domain $C_4$}
We have $h_4 > h_5 = 0$, $E \in (0, h_4)$,  the corresponding trajectory is shown in Figs.~\ref{fig:2)I2}, \ref{fig:xy2)C4}.

\twofiglabelsize
{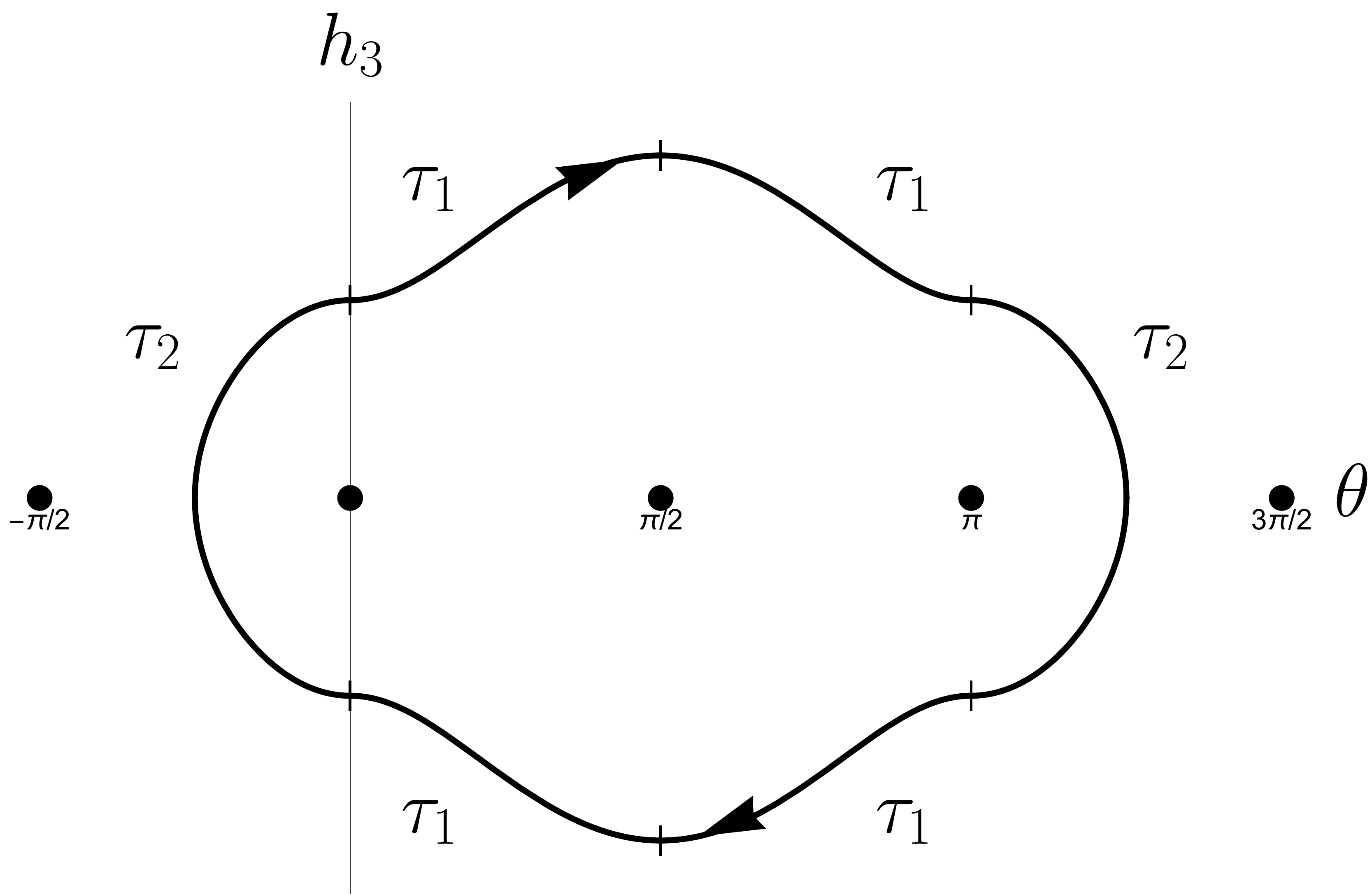}{$(\theta(t), h_3(t))$: Case 2), domain $C_4$}{fig:2)I2}{0.8}
{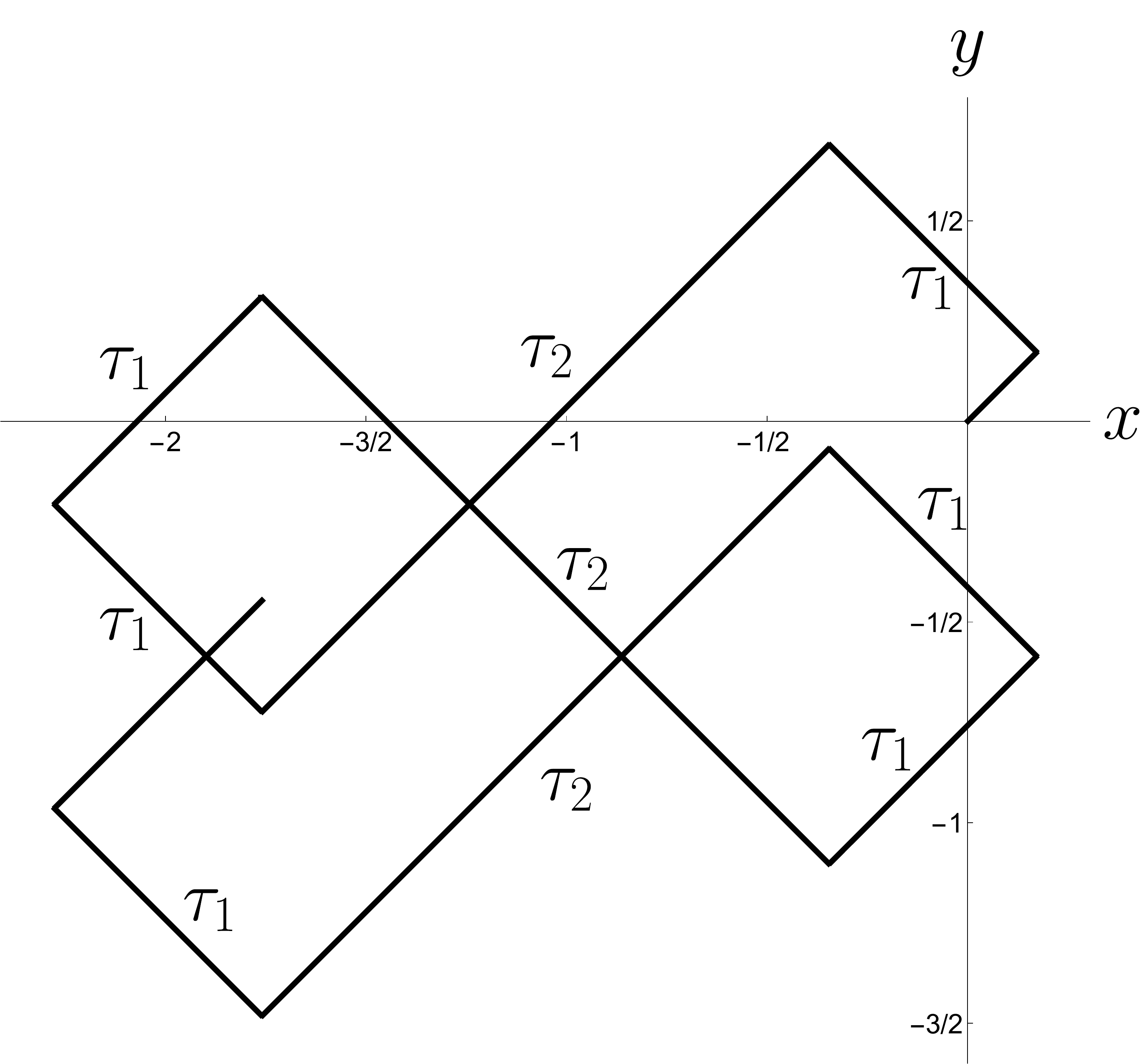}{$(x(t), y(t))$: Case 2), domain $C_4$}{fig:xy2)C4}{0.7}

The control has the form
\begin{table}[H]
\begin{center}
\begin{tabular}{ccccccccccc}
$(u_1, u_2):$ & $\quad$ & $\dots$ & $(+, +)$ & $(-, +)$ & $(-, -)$ & $(-, +)$ & $(+, +)$ & $(+, -)$ & $(+, +)$ &\dots \\
& & & $\tau_1$ & $\tau_1$ & $\tau_2$ & $\tau_1$ & $\tau_1$ & $\tau_2$ & $\tau_1$&
\end{tabular}
\end{center}
\end{table}
with
\begin{gather*}
\tau_1 = \frac{\sqrt{2(E + h_4)} - \sqrt{2E}}{h_4} =\frac{2}{\sqrt{2(E + h_4)} + \sqrt{2E}}, \quad \tau_2 = \frac{2\sqrt{2E}}{h_4}.
\end{gather*}

\subsubsection{Case $2)$, level line $C_5$}
We have $h_4 > h_5 = 0$, $E = h_4$, the corresponding level line is shown in Fig.~\ref{fig:2)C3}.

\twofiglabelsize
{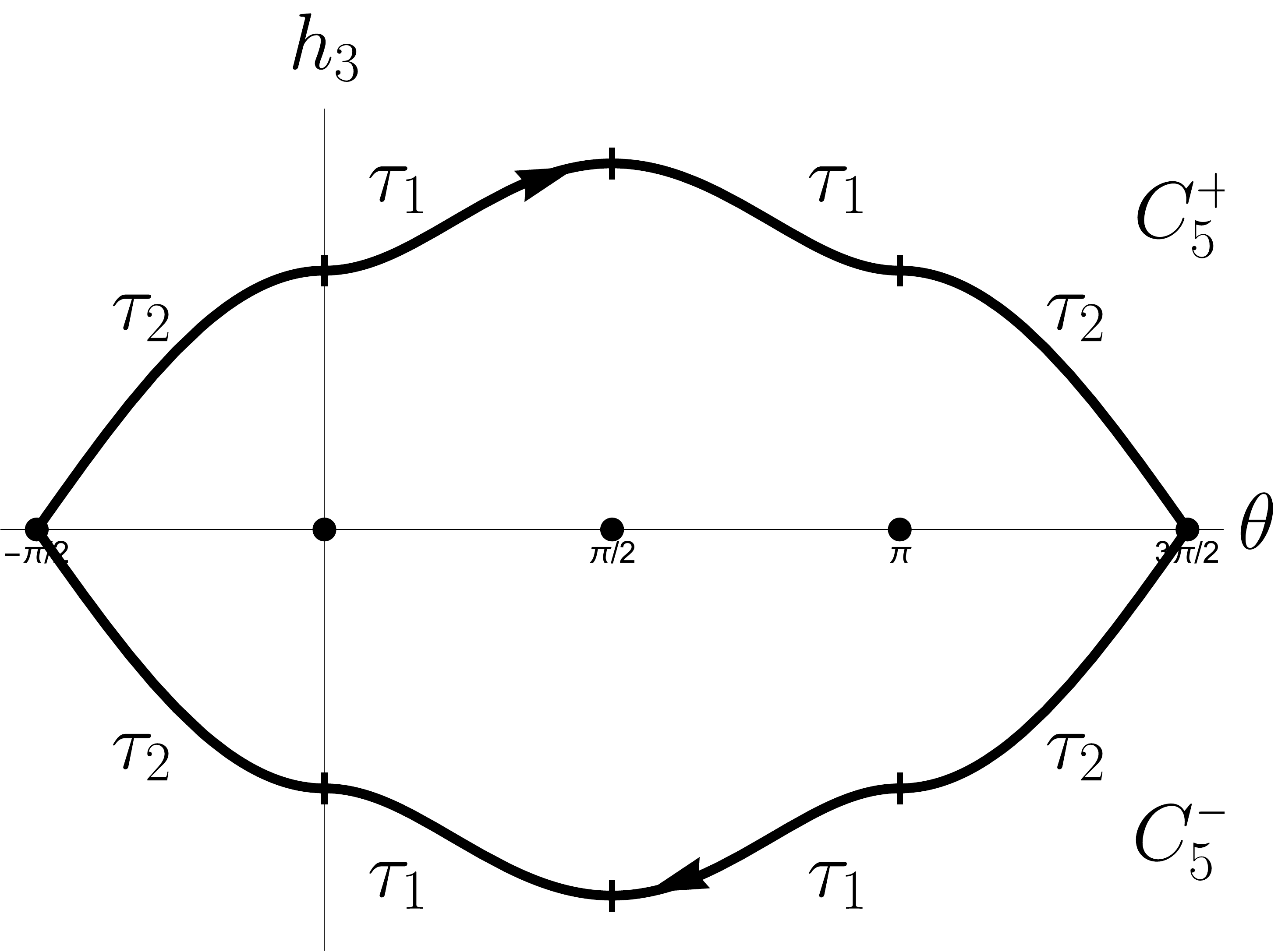}{Case 2), level line $C_5$}{fig:2)C3}{0.7}
{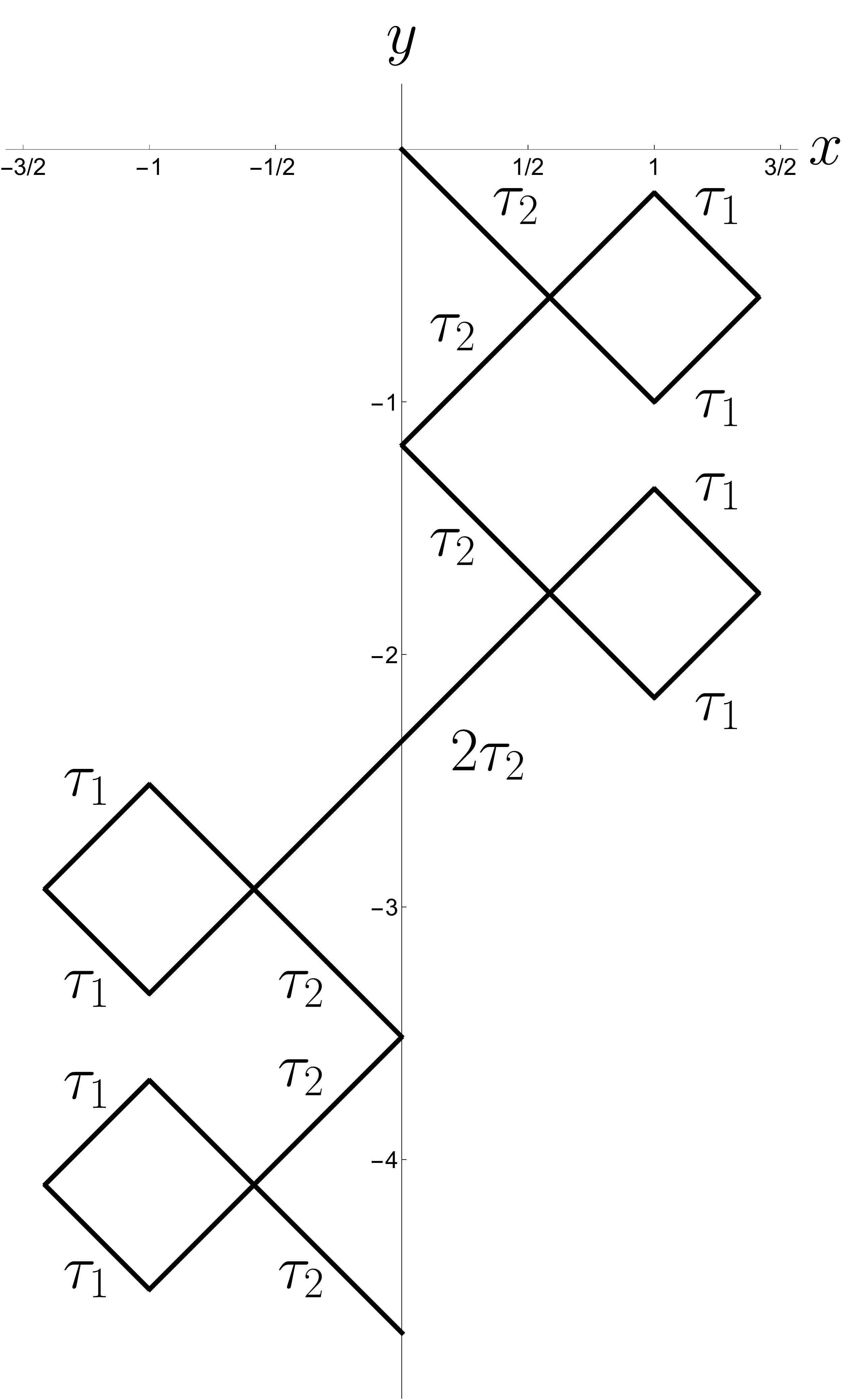}{$(x(t), y(t))$: Case 2), level line $C_5$}{fig:xy2)C5}{0.6}

There are decompositions
\begin{gather*}
C_5 = C_5^+ \cup C_5^-, \quad C_5^+ = C_5 \cap \{ h_3 \ge 0 \}, \quad C_5^- = C_5 \cap \{ h_3 \le 0 \}.
\end{gather*}
The curves $C_5^+$, $C_5^-$ are homeomorphic to $S^1$, with the only singularity --- a corner point $(\theta, h_3) = (\frac{3\pi}{2}, 0)$.

Bang-bang controls are obtained by choosing a finite segment of the following periodic graph:
\medskip
$$
\xymatrix{
&&\tau_1 & \tau_1 & \tau_2 & \tau_2 & \tau_1  & \\
C_5^+ & \cdots \ar@{>}[r] & (+,+) \ar@{>}[r]& (-,+) \ar@{>}[r]& (-,-)\ar@{>}[r] \ar@{>}[dr]& (+,-)\ar@{>}[r]& (+,+)\ar@{>}[r]& \cdots  \\
C_5^- & \cdots \ar@{>}[r] & (-,+) \ar@{>}[r]& (+,+) \ar@{>}[r]& (+,-) \ar@{>}[r] \ar@{>}[ur] & (-,-) \ar@{>}[r]& (-,+)\ar@{>}[r]&\cdots  \\
&&\tau_1 & \tau_1 & \tau_2 & \tau_2 & \tau_1  & 
}
$$
\medskip

We have
\begin{gather*}
\tau_1 = \frac{2 - \sqrt{2}}{\sqrt{h_4}}, \qquad \tau_2 = \sqrt{\frac{2}{h_4}}.
\end{gather*}

The  curves $(x(t), y(t))$ corresponding to the curves $C_5^+$ and $C_5^-$ are shown in Figs.~\ref{fig:xy2)C5+} and  \ref{fig:xy2)C5-} respectively. An example of curve $(x(t), y(t))$ corresponding to two curves $C_5^+$ and two curves  $C_5^-$ is given in Fig.~\ref{fig:xy2)C5}.

\twofiglabelsize
{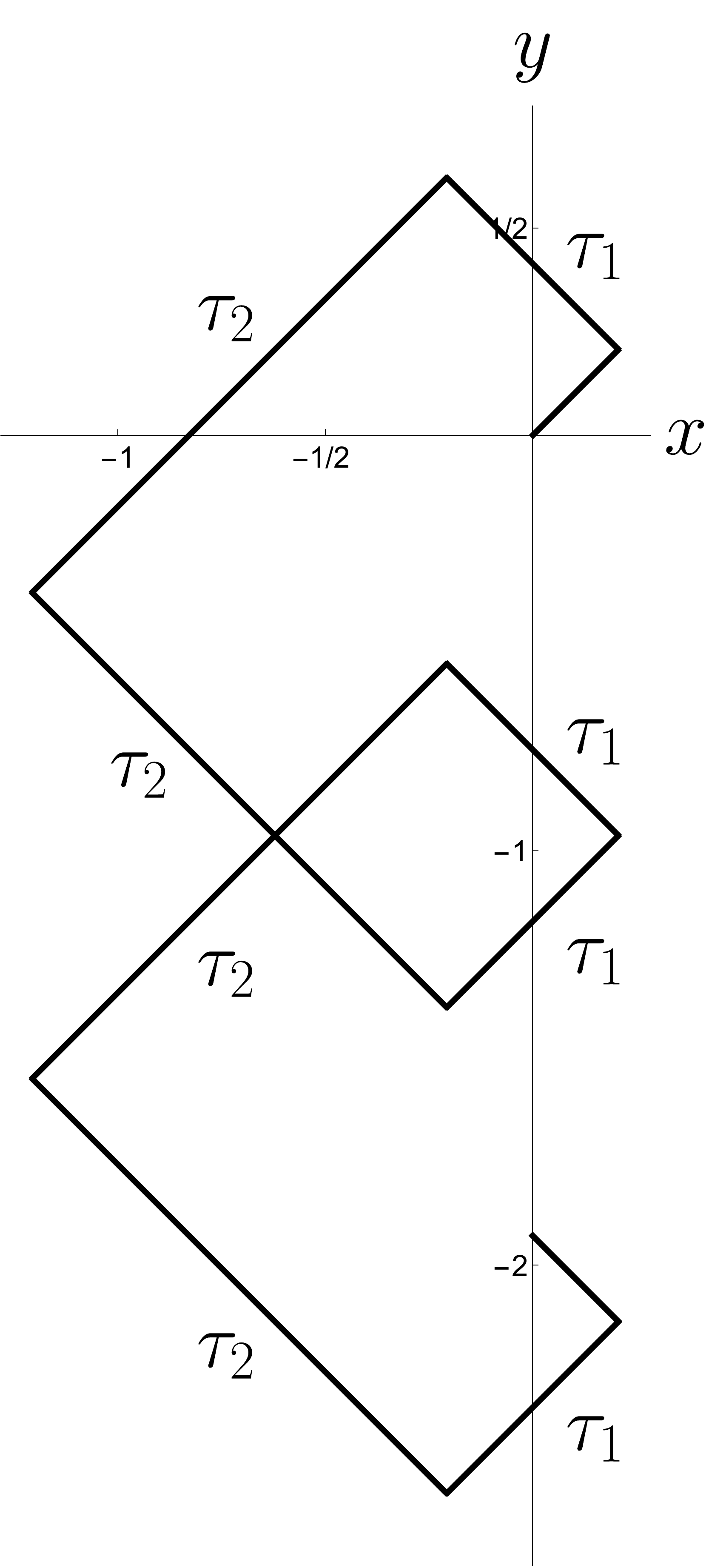}{$(x(t), y(t))$: Case 2), level line $C_5^+$}{fig:xy2)C5+}{0.4}
{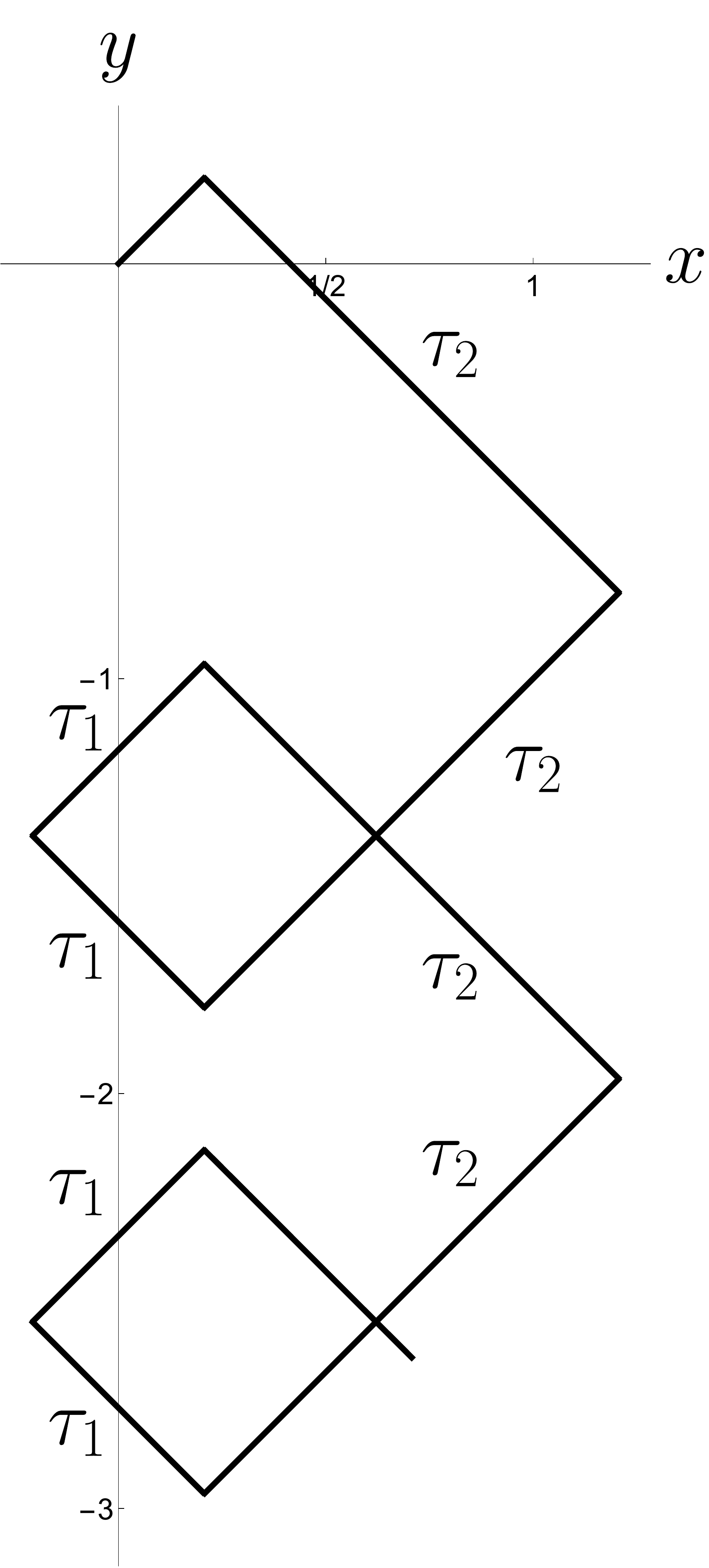}{$(x(t), y(t))$: Case 2), level line $C_5^-$}{fig:xy2)C5-}{0.4}

\subsubsection{Case $2)$, domain $C_6$}
We have $h_4 > h_5 = 0$, $E > h_4$,  the corresponding trajectory is shown in Figs.~\ref{fig:2)N}, \ref{fig:xy2)C6}.

\twofiglabelsize
{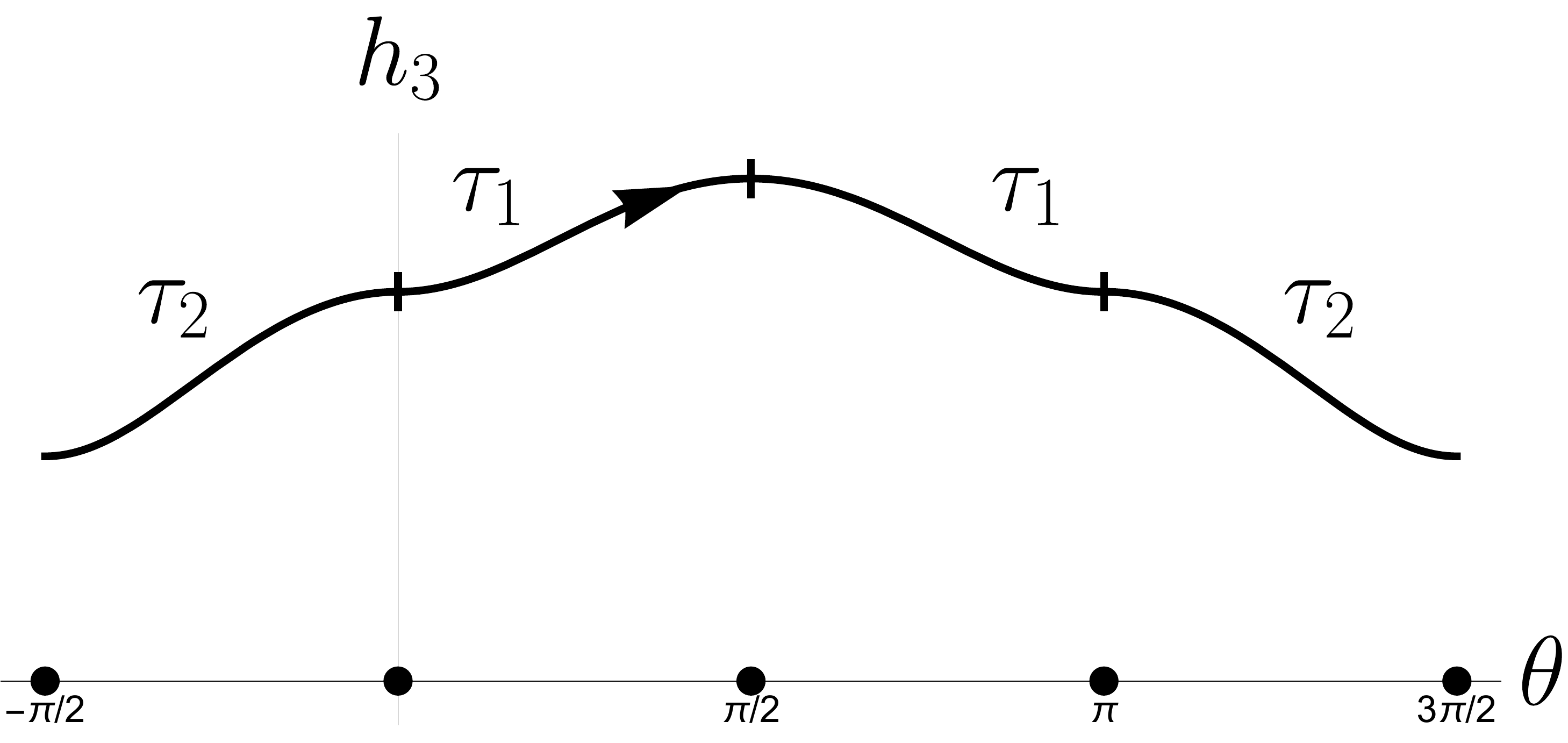}{$(\theta(t), h_3(t))$: Case 2), domain $C_6$}{fig:2)N}{0.7}
{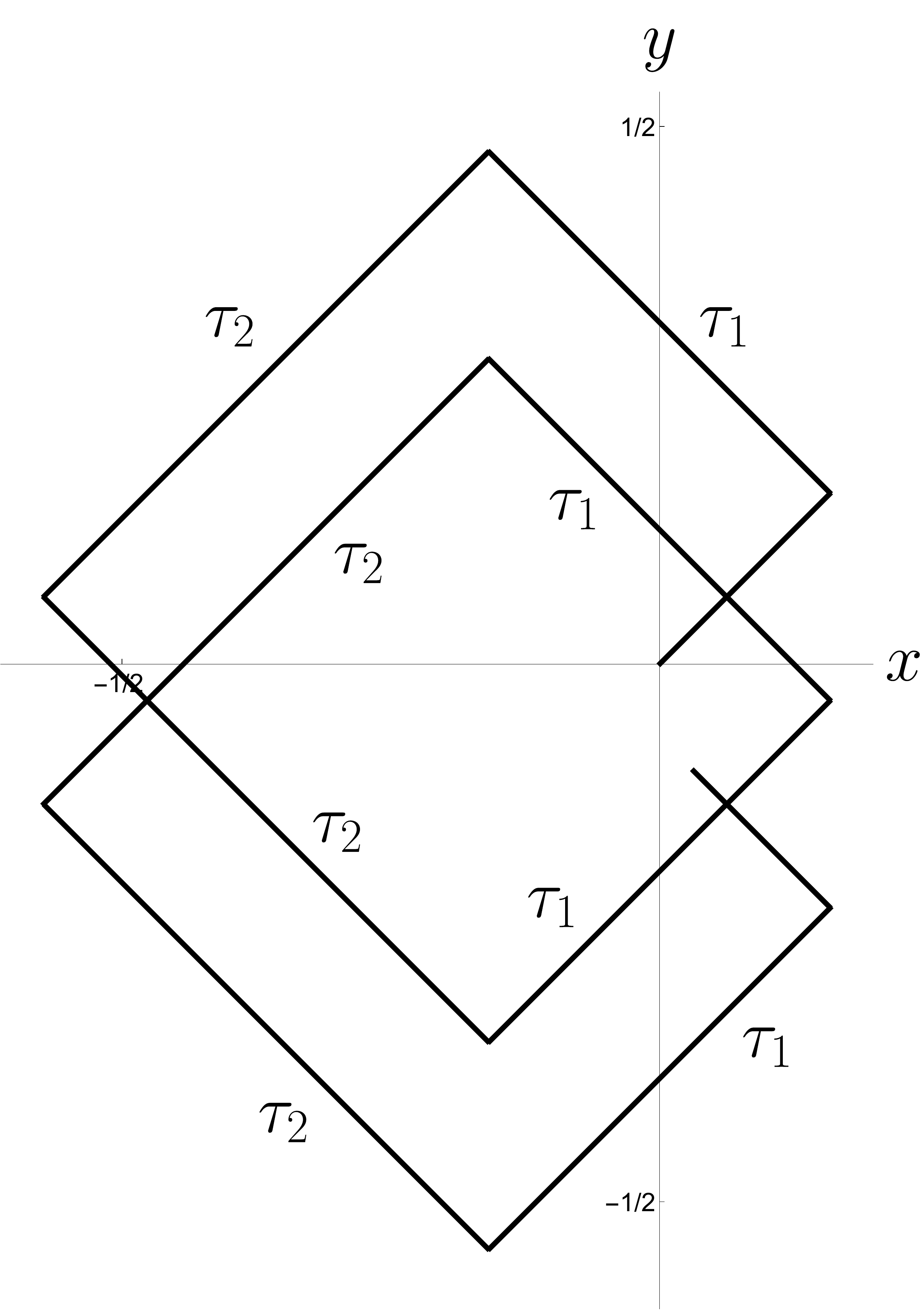}{$(x(t), y(t))$: Case 2), domain $C_6$}{fig:xy2)C6}{0.6}

If $h_3 > 0$, then the control is given by 
\begin{table}[H]
\begin{center}
\begin{tabular}{ccccccccc}
$(u_1, u_2):$ & $\quad$ & $\dots$ & $(+, +)$ & $(-, +)$ & $(-, -)$ & $(+, -)$ & $(+, +)$  &\dots \\
& & & $\tau_1$ & $\tau_1$ & $\tau_2$ & $\tau_2$ & $\tau_1$ &
\end{tabular}
\end{center}
\end{table}
If $h_3 < 0$, then the order of switchings is opposite. We have
\begin{gather*}
\tau_1 = \frac{\sqrt{2(E + h_4)} - \sqrt{2E}}{h_4} =\frac{2}{\sqrt{2(E + h_4)} + \sqrt{2E}}, \\
\tau_2 = \frac{\sqrt{2E} - \sqrt{2(E - h_4)}}{h_4} = \frac{2}{\sqrt{2E} + \sqrt{2(E - h_4)}}.
\end{gather*}

\subsection{Case $3)$}
Let $h_4 = h_5 > 0$. Then system~\eq{Hamtheta} has the phase portrait given in Fig.~\ref{fig:3)h3th}, see Subsubsec.~7.2.3~\cite{SFCartan1}.

\onefiglabelsize
{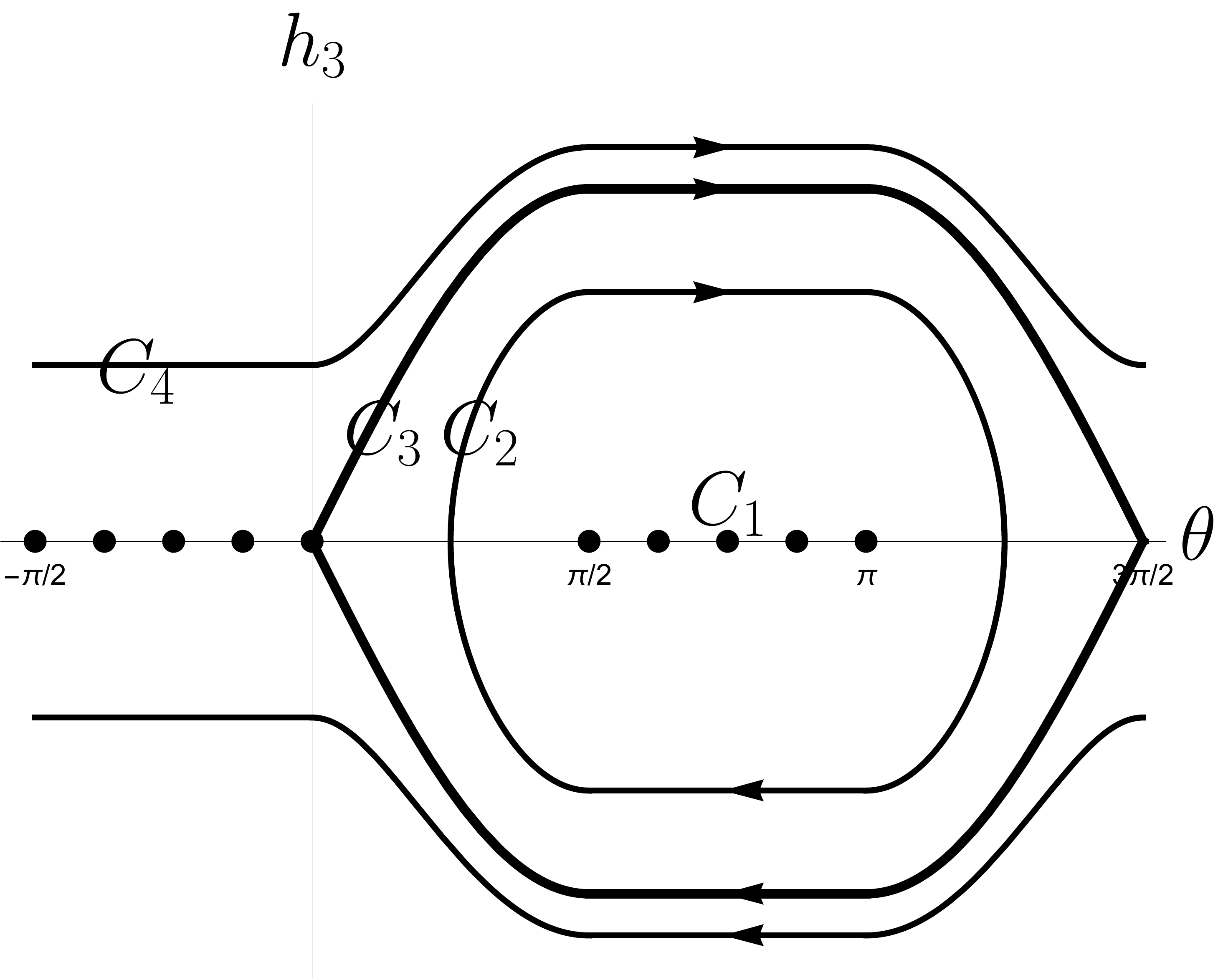}{Phase portrait of system \eq{Hamtheta} in case 3)}{fig:3)h3th}{0.35}

The domain $\{ \lambda \in C \mid h_4 = h_5 > 0 \}$ of the cylinder $C = L^* \cap \{H = 1\}$ admits a decomposition defined by the energy integral $E$:
\begin{align*}
&\{ \lambda \in C \mid h_4 = h_5 > 0 \} = \cup_{i=1}^4 C_i, \\
&C_1 = E^{-1} (-h_4),   \qquad C_2 = E^{-1} (-h_4, 0), \qquad C_3 = E^{-1} (h_4),   \qquad C_4 = E^{-1}   (h_4, +\infty).
\end{align*}

\subsubsection{Case $3)$, level set $C_1$}
Let $h_4 = h_5 > 0$, $E = - h_4$. Then $\theta \equiv \const \in [\pi/2,\pi]$, $h_3 \equiv 0$.

If $\theta \in (\pi/2, \pi)$, then $(x(t), y(t)) = (-t, t)$. 
And if $\theta = \pi/2$ ($\theta = \pi$), we get an $h_1$-singular (resp. $h_2$-singular) trajectory.

\subsubsection{Case $3)$, domain $C_2$}
We have $h_4 = h_5 > 0$, $E \in (-h_4, h_4)$,  the corresponding trajectory is shown in Figs.~\ref{fig:3)I}, \ref{fig:xy3)C2}.

\twofiglabelsize
{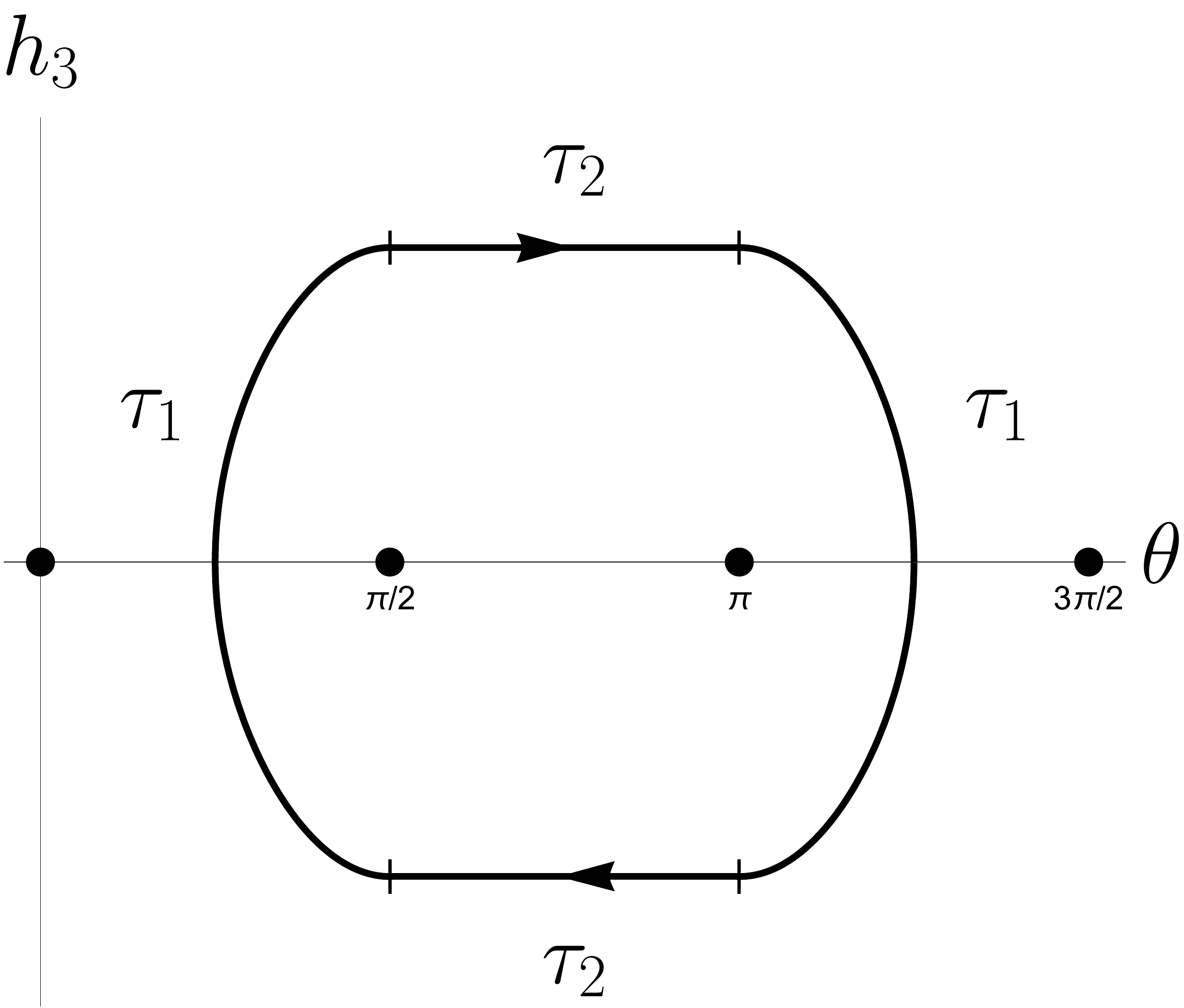}{$(\theta(t), h_3(t))$: Case 3), domain $C_2$}{fig:3)I}{0.6}
{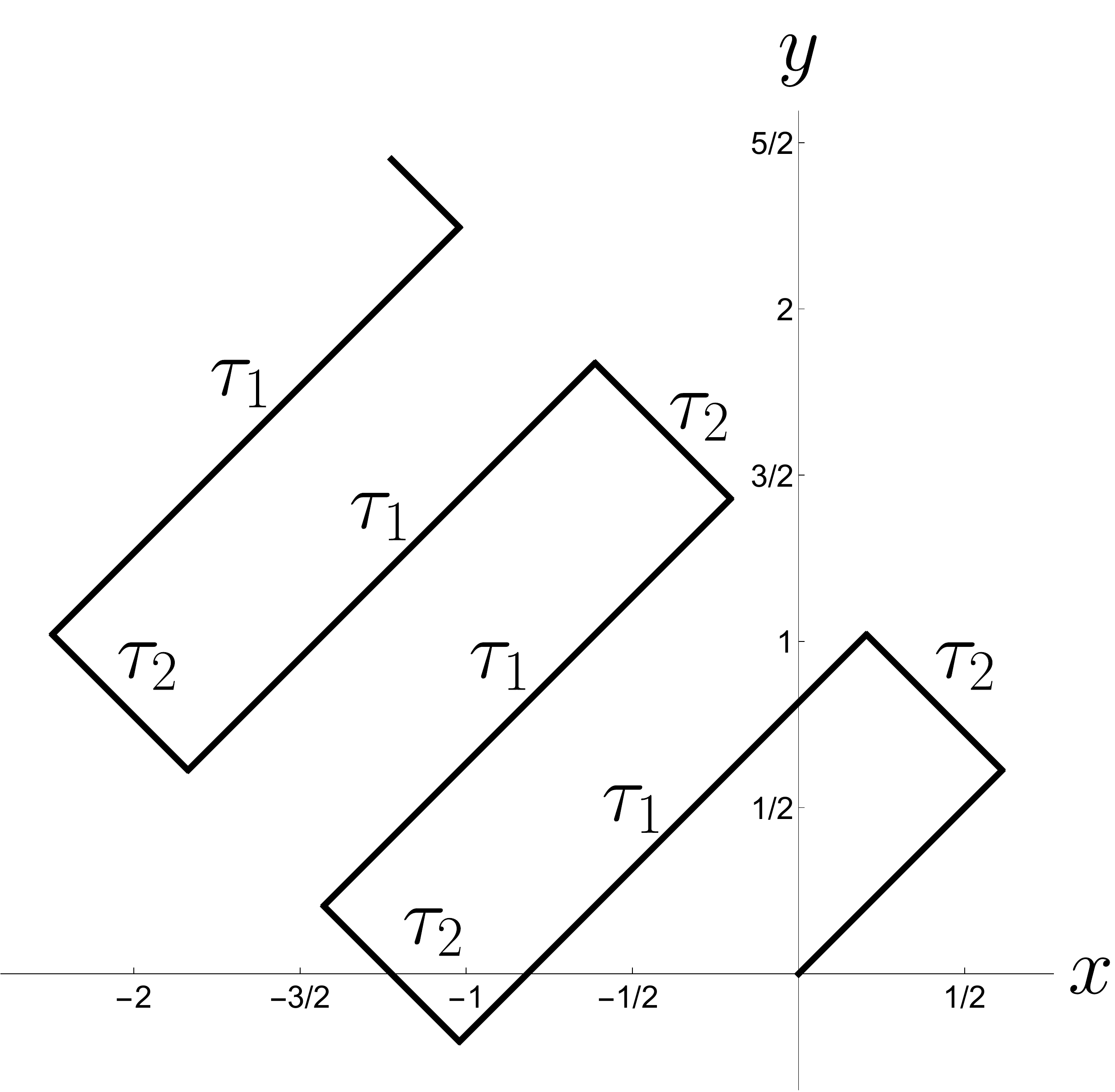}{$(x(t),y(t))$: Case 3), domain $C_2$}{fig:xy3)C2}{0.6}

Then the control has the form
\begin{table}[H]
\begin{center}
\begin{tabular}{ccccccccc}
$(u_1, u_2):$ & $\quad$ & $\dots$ & $(+, +)$ & $(-, +)$ & $(-, -)$ & $(-, +)$ & $(+, +)$  &\dots \\
& & & $\tau_1$ & $\tau_2$ & $\tau_1$ & $\tau_2$ & $\tau_1$ &
\end{tabular}
\end{center}
\end{table}
with
\begin{gather*}
\tau_1 = \frac{\sqrt{2(E + h_4)}}{h_4}, \quad \tau_2 = \frac{1}{\sqrt{2(E + h_4)}}.
\end{gather*}

\subsubsection{Case $3)$, level line $C_3$}
We have $h_4 = h_5 > 0$, $E = h_4$,  the corresponding trajectory is shown in Figs.~\ref{fig:3)C2}, \ref{fig:xy3)C3}.

\twofiglabelsize
{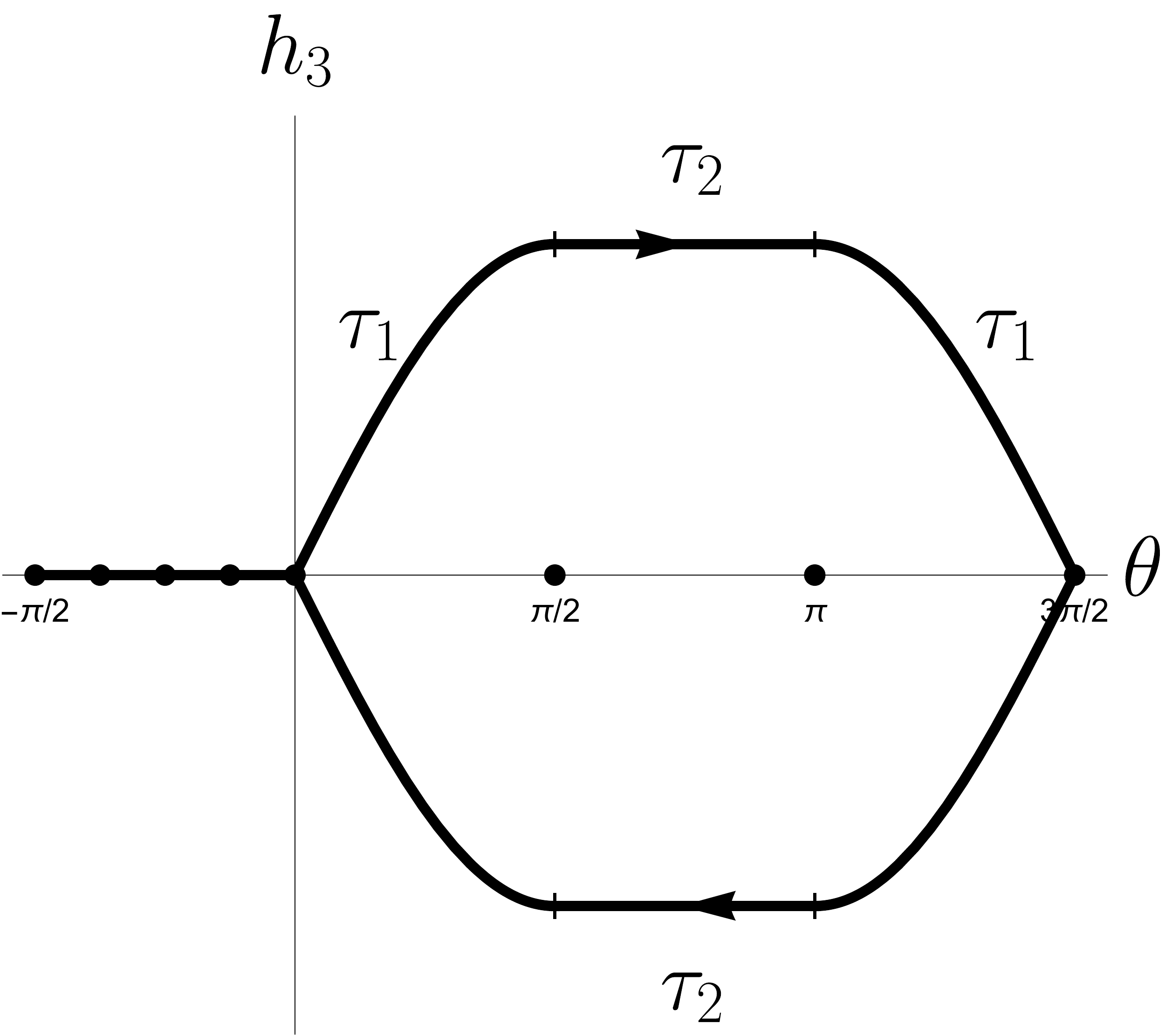}{$(\theta(t), h_3(t))$: Case 3), level line $C_3$}{fig:3)C2}{0.6}
{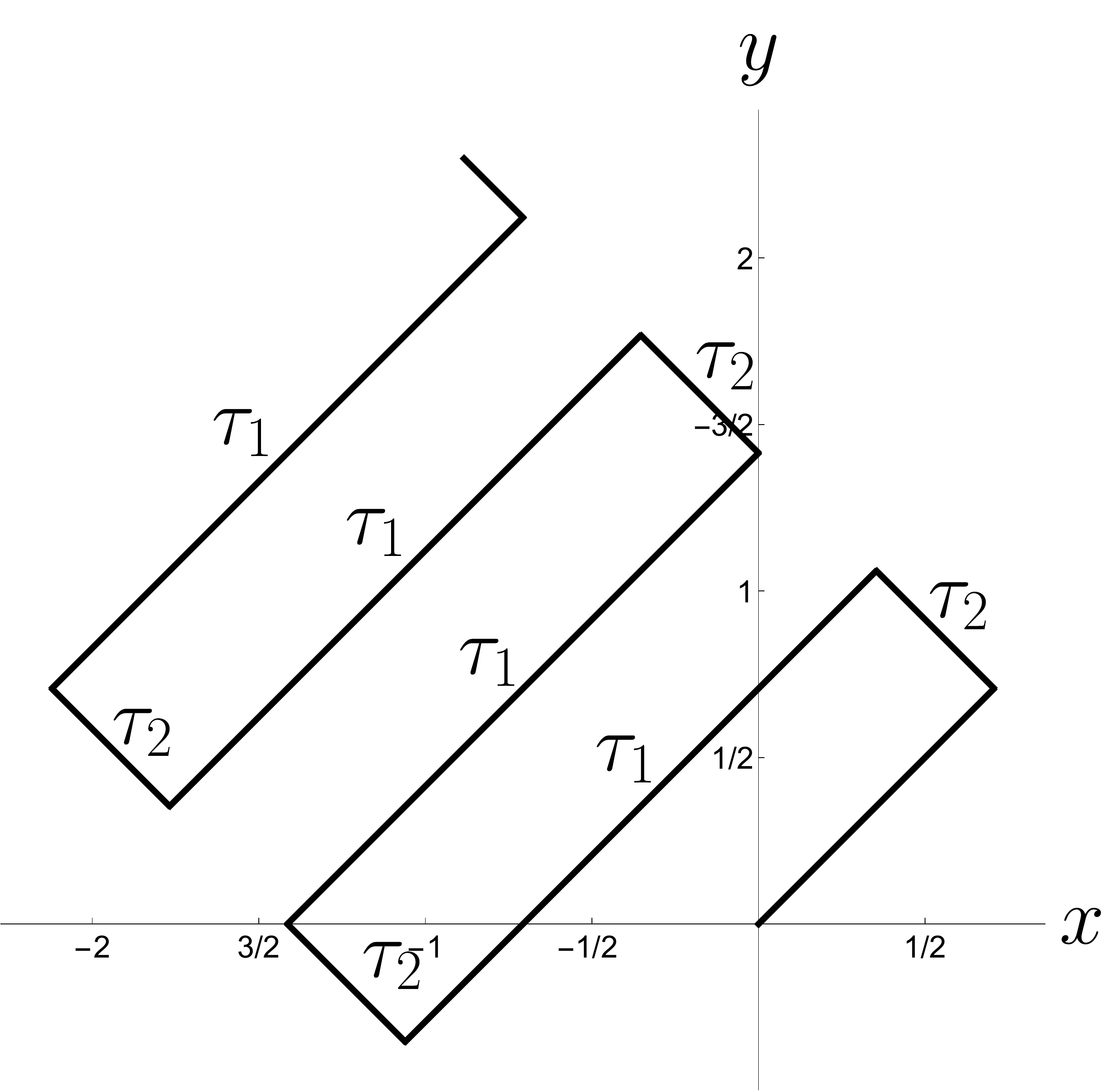}{$(x(t), y(t))$: Case 3), level line $C_3$}{fig:xy3)C3}{0.6}

The level line $C_3$ consists of the segment $\{ \theta \in [-\frac{\pi}{2}, 0], h_3 = 0 \}$ and the curve $\{ (\theta, h_3) \mid h_3 = \pm 2\sqrt{h_4} \sin \theta, \theta \in [0, \frac{3\pi}{2}] \}$ homeomorphic to $S^1$, with two singularities --- corner points $(\theta, h_3) = (0, 0)$ and $(\theta, h_3) = (\frac{3\pi}{2}, 0)$.

The bang-bang control is given by:
\begin{table}[H]
\begin{center}
\begin{tabular}{ccccccccc}
$(u_1, u_2):$ & $\quad$ & $\dots$ & $(+, +)$ & $(-, +)$ & $(-, -)$ & $(-, +)$ & $(+, +)$  &\dots \\
& & & $\tau_1$ & $\tau_2$ & $\tau_1$ & $\tau_2$ & $\tau_1$ &
\end{tabular}
\end{center}
\end{table}
with
\begin{gather*}
\tau_1 = \frac{2}{\sqrt{h_4}}, \quad \tau_2 = \frac{1}{2\sqrt{h_4}}.
\end{gather*}

\subsubsection{Case $3)$, domain $C_4$}
We have $h_4 = h_5 > 0$, $E > h_4$,  the corresponding trajectory is shown in Figs.~\ref{fig:3)N}, \ref{fig:xy3)C4}.

\twofiglabelsize
{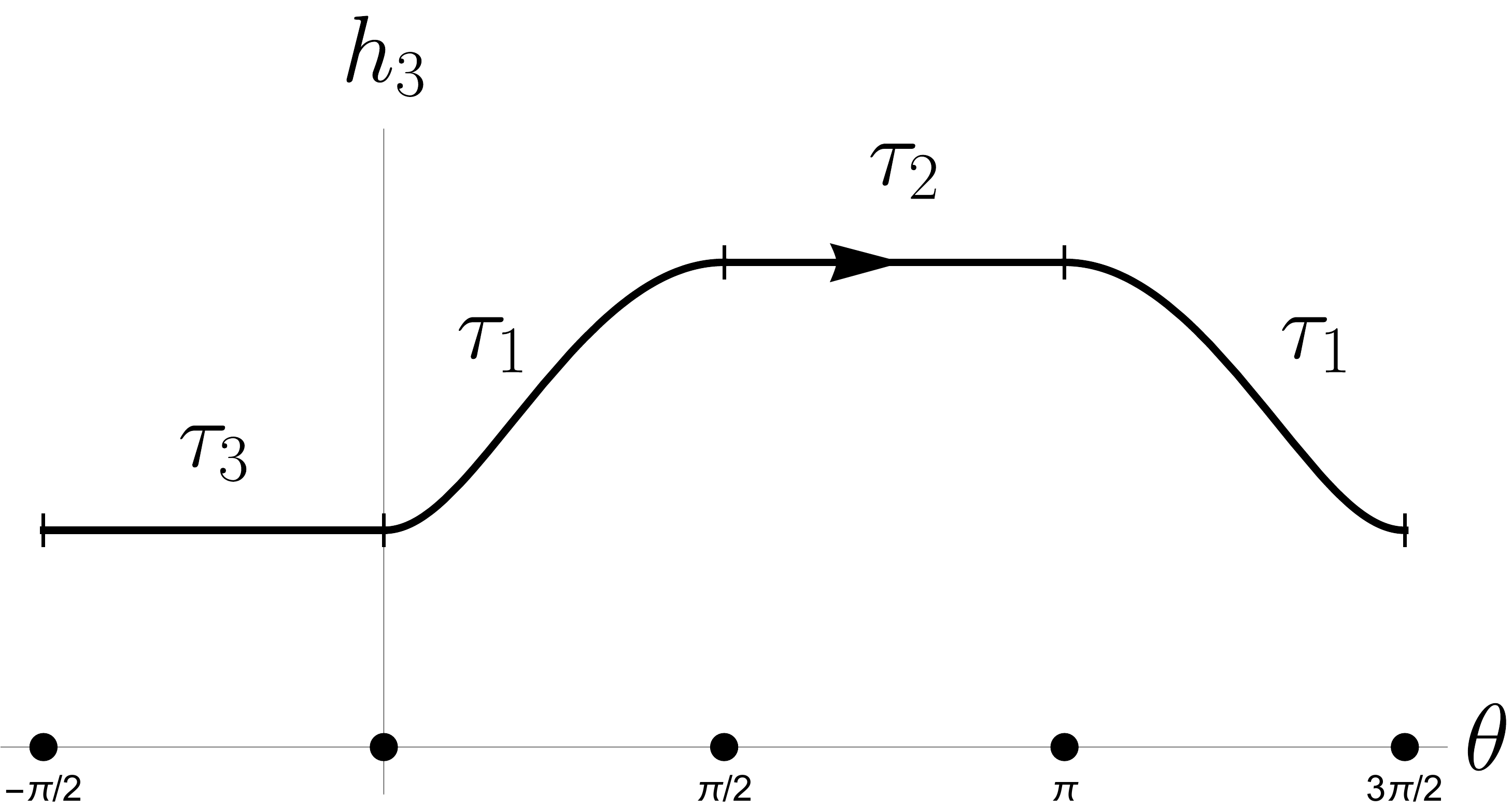}{$(\theta(t), h_3(t))$: Case 3), domain $C_4$}{fig:3)N}{0.7}
{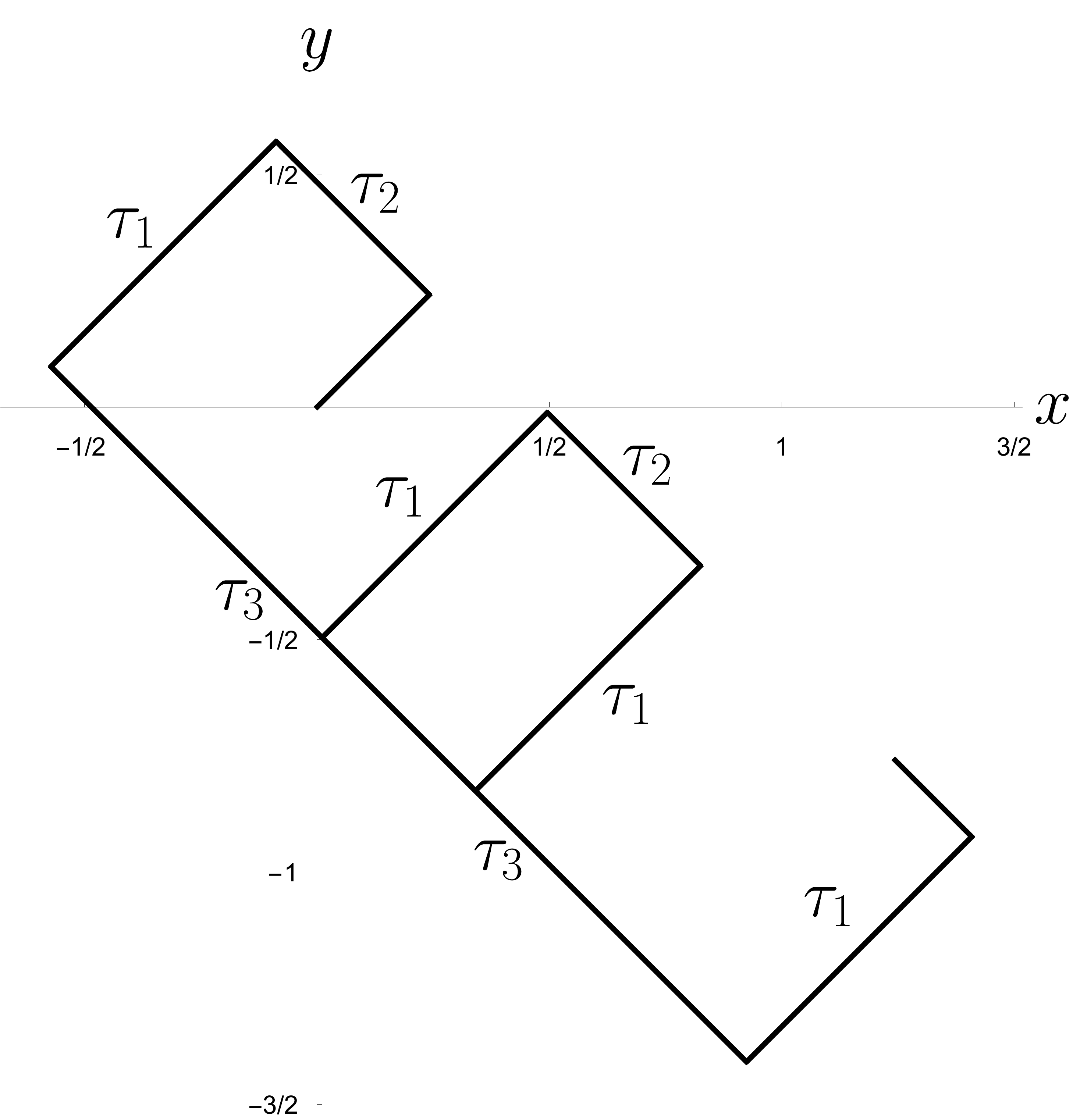}{$(x(t), y(t))$: Case 3), domain $C_4$}{fig:xy3)C4}{0.7}

If $h_3 > 0$, then the control is given by: 
\begin{table}[H]
\begin{center}
\begin{tabular}{ccccccccc}
$(u_1, u_2):$ & $\quad$ & $\dots$ & $(+, +)$ & $(-, +)$ & $(-, -)$ & $(+, -)$ & $(+, +)$  &\dots \\
& & & $\tau_1$ & $\tau_2$ & $\tau_1$ & $\tau_3$ & $\tau_1$ &
\end{tabular}
\end{center}
\end{table}
If $h_3 < 0$, then the order of switchings is opposite. We have
\begin{gather*}
\tau_1 = \frac{\sqrt{2(E + h_4)} - \sqrt{2(E - h_4)}}{2h_4} =\frac{2}{\sqrt{2(E + h_4)} + \sqrt{2(E - h_4)}}, \\
\tau_2 = \frac{1}{\sqrt{2(E + h_4)}}, \quad \tau_3 = \frac{1}{\sqrt{2(E - h_4)}}.
\end{gather*}

\subsection{Case $4)$}\label{subsec:4)}
Let $h_4 = h_5 = 0$. Then system~\eq{Hamtheta} has the phase portrait given in Fig.~\ref{fig:4)h3th}, see Subsubsec.~7.2.4~\cite{SFCartan1}.

\onefiglabelsize
{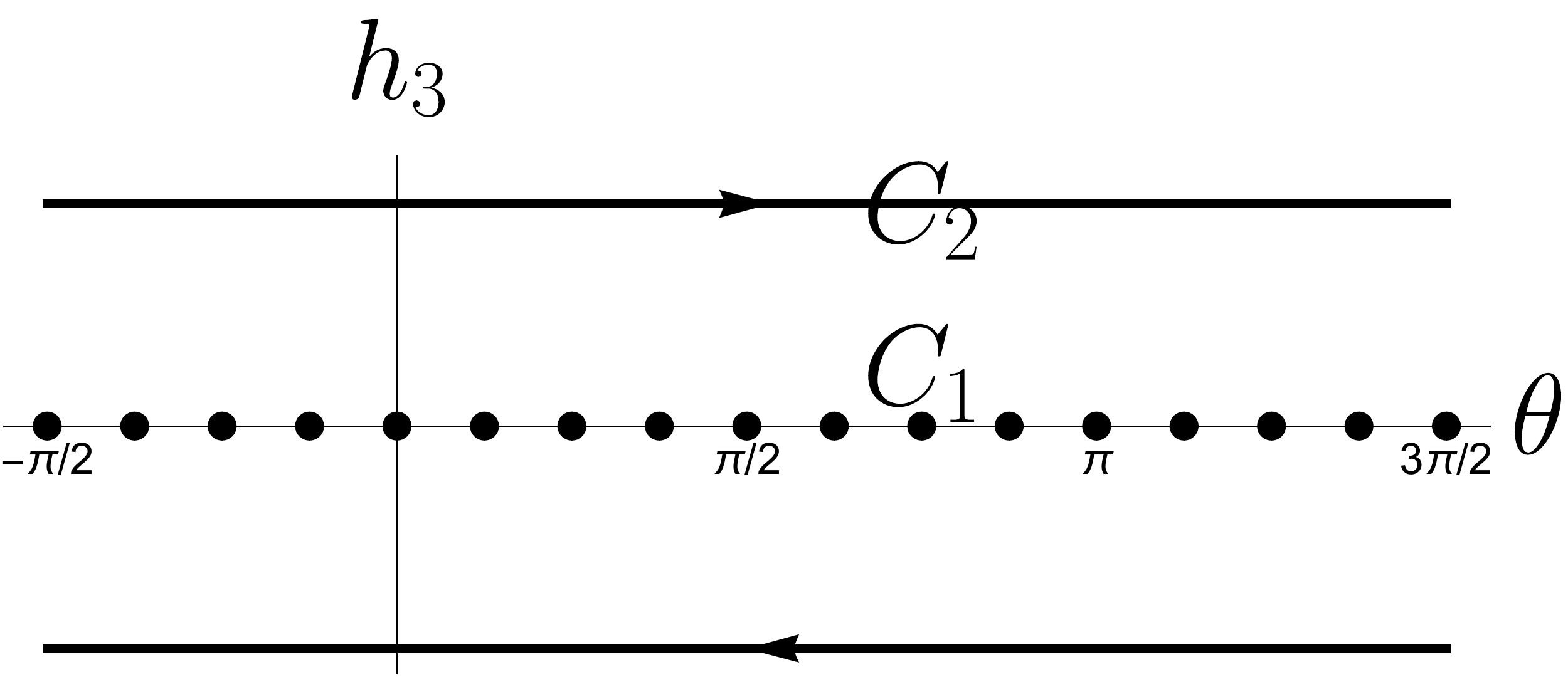}{Phase portrait of system \eq{Hamtheta} in case 4)}{fig:4)h3th}{0.3}

The domain $\{ \lambda \in C \mid h_4 = h_5 = 0 \}$ of the cylinder $C = L^* \cap \{H = 1\}$ admits a decomposition defined by the energy integral $E$:
\begin{align*}
&\{ \lambda \in C \mid h_4 = h_5 = 0 \} = C_1 \cup C_2, \\
&C_1 = E^{-1}(0), \qquad C_2 = E^{-1}(0, +\infty).
\end{align*}

\subsubsection{Case $4)$, level set $C_1$}
We have $h_4 = h_5 = 0$, $E = 0$.
The level set $C_1$ consists of fixed points $(\theta, h_3)$, $h_3 = 0$, $\theta \ne \frac{\pi n}{2}$ that correspond to bang trajectories (which are simultaneously abnormal), and fixed points $(\theta, h_3) = (\frac{\pi n}{2}, 0)$ that correspond to singular trajectories (which are simultaneously abnormal as well).

\subsubsection{Case $4)$, domain $C_2$}
We have $h_4 = h_5 = 0$, $E > 0$,  the corresponding trajectory is shown in Figs.~\ref{fig:4)N}, \ref{fig:xy4)C2}.

\twofiglabelsize
{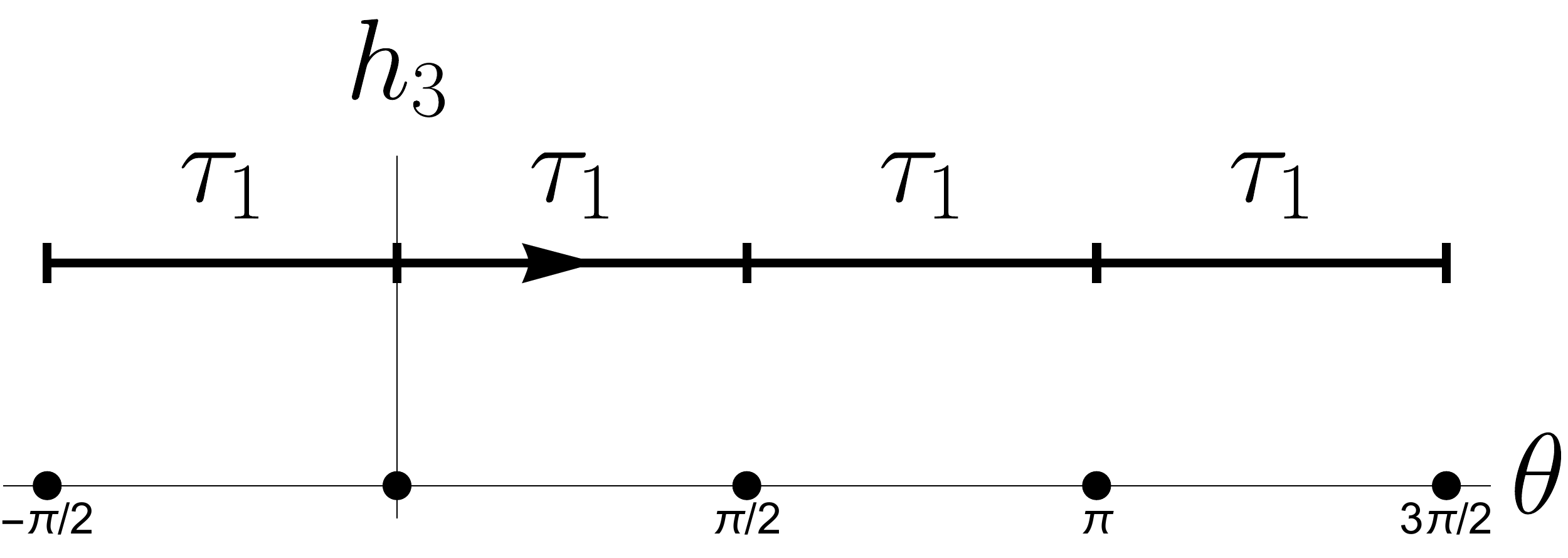}{$(\theta(t), h_3(t))$: Case 4), domain $C_2$}{fig:4)N}{0.6}
{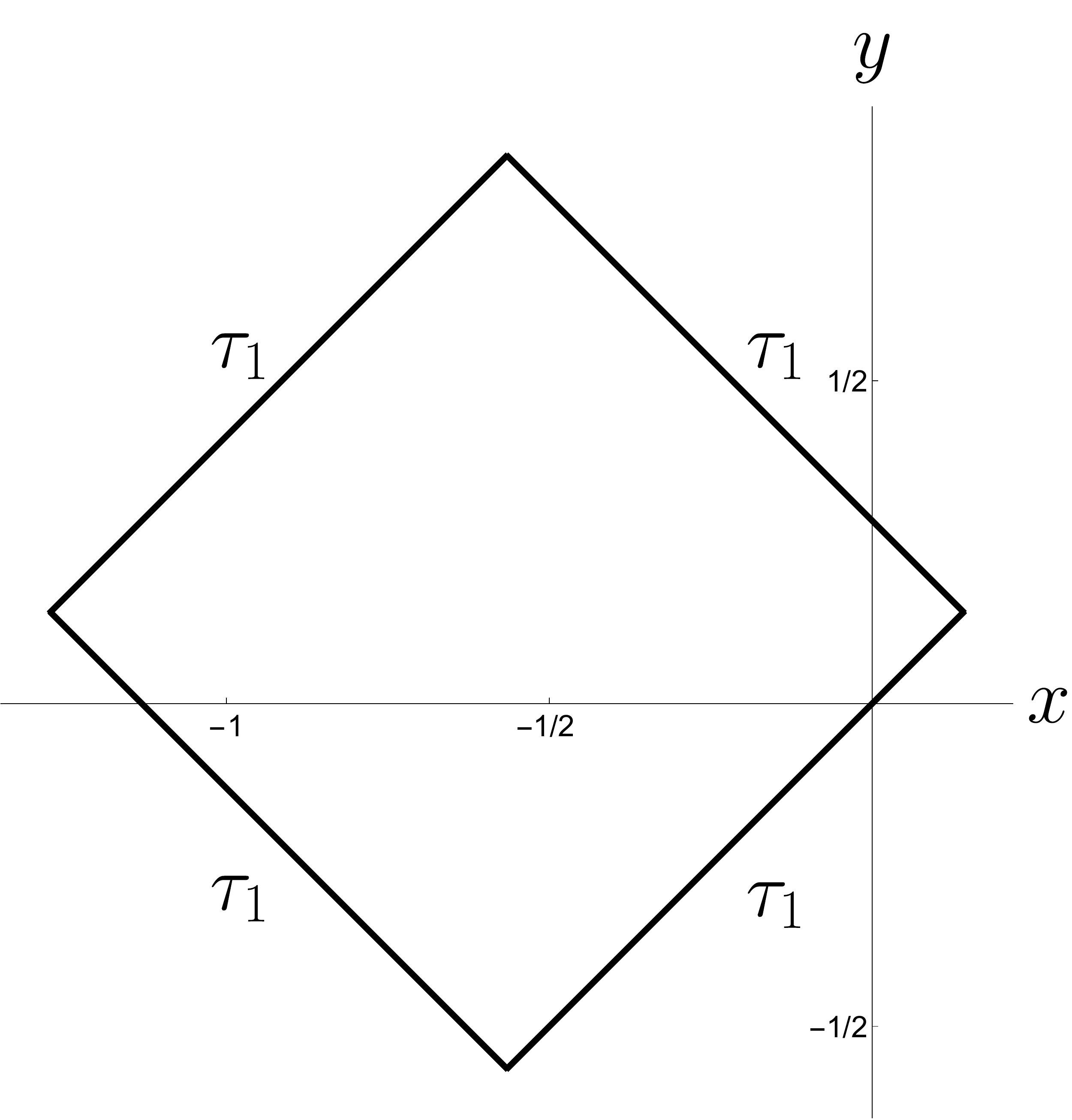}{$(x(t),y(t))$: Case 4), domain $C_2$}{fig:xy4)C2}{0.6}

If $h_3 > 0$, then the control $u$ is given by:
\begin{table}[H]
\begin{center}
\begin{tabular}{ccccccccc}
$(u_1, u_2):$ & $\quad$ & $\dots$ & $(+, +)$ & $(-, +)$ & $(-, -)$ & $(+, -)$ & $(+, +)$  &\dots \\
& & & $\tau_1$ & $\tau_1$ & $\tau_1$ & $\tau_1$ & $\tau_1$ &
\end{tabular}
\end{center}
\end{table}
If $h_3 < 0$, then the order of switchings is opposite. We have
\begin{gather*}
\tau_1 = \frac{1}{\sqrt{2E}}.
\end{gather*}

On the basis of the results obtained in this section we obtain the following statement.

\begin{corollary}\label{cor:tau}
For all bang-bang trajectories, duration of bang arcs is a function of Casimirs: $\tau_i = \tau_i(h_4, h_5, E)$, except the first and the last arcs.
\end{corollary}

\section{Optimality of bang-bang trajectories}
In this section we obtain upper bounds on the number of switchings on bang-bang minimizers.

\subsection{Bang-bang trajectories with low energy $E$}
\begin{theorem}
\label{th:opt_low_en}
If a bang-bang extremal $\lambda_t$, $t \in [0, +\infty)$, satisfies the inequality 
\begin{gather}
\min (-|h_4|, -|h_5|) < E \le \max (-|h_4|, -|h_5|) \label{minEmax}
\end{gather}
then it is optimal, i.e., $t_{\cut} (\lambda_0) = +\infty$.
\end{theorem}
\begin{proof}
All three functions $\min (-|h_4|, -|h_5|)$, $E$, $\max (-|h_4|, -|h_5|)$ are invariant w.r.t. the group of symmetries 
of the square $\{|h_1|+|h_2| = 1\}$ (see Sec.~7.1~\cite{SFCartan1}),
thus it suffices to prove this theorem for the fundamental domain of this group $\{ h_4 \ge h_5 \ge 0 \}$. On this domain inequality \eq{minEmax} turns into the inequality $-h_4 < E \le -h_5 \le 0$, which
is equivalent to the union of the following inequalities:
\begin{enumerate}[label=(\alph*)]
\item $-h_4 < E \le -h_5 \le 0$,
\item $-h_4 < E \le -h_5 = 0$.
\end{enumerate}
Case (a) is exactly Case 1), domain $C_2$ (Subsec. \ref{subsec:1)I1}) and Case 1), level line $C_3$ (Subsec. \ref{subsec:1)C2}). And case (b) is exactly Case 2), domain $C_2$  (Subsec. \ref{subsec:2)I1}) and Case 2), level line $C_3$ (Subsec. \ref{subsec:2)C2}). In all these cases $\theta (t) \in [0, \pi]$, and $\theta (t)$ takes extreme values $0, \pi$ at isolated instants of time $t$. Thus $h_2 (\lambda_t) > 0$ for almost all $t$, whence $u_2 (t) \equiv 1$ for almost all $t$. By Lemma~2~\cite{SFCartan1}, the control $u(t)$ is optimal.
\end{proof}

The domain in the phase cylinder of system~\eq{Hamtheta} corresponding to inequalities~\eq{minEmax} is shown in Fig.~\ref{fig:minEmax}.
\begin{figure}[htbp]
\begin{center}
\includegraphics[width=2.in]{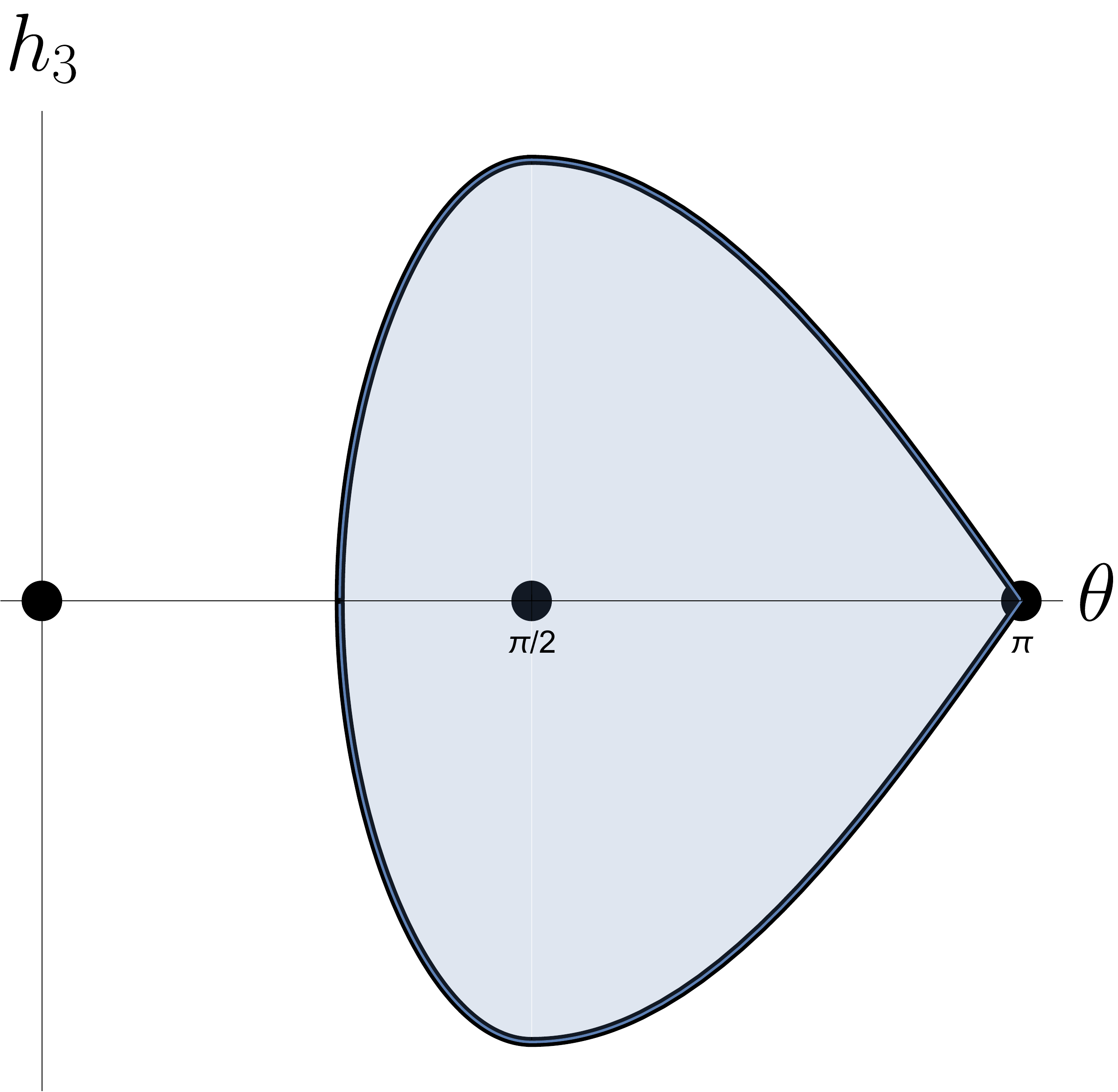}
\caption{Domain~\eq{minEmax}}
\label{fig:minEmax}
\end{center}
\end{figure}

\begin{remark}
It is easy to see that under condition \eq{minEmax} a bang-bang trajectory
is simultaneously a singular trajectory, i.e., $q(t) = \pi(\tilde{\lambda}_t)$, while $\tilde{\lambda}_t$ is an $h_1$-singular extremal linearly independent of $\lambda_t$. One can show  that there exists also a bang-bang extremal $\bar{\lambda}_t$, linearly independent of $\lambda_t$ and $\tilde{\lambda}_t$, such that $q(t) = \pi (\bar{\lambda}_t)$. Thus the trajectory $q(t)$ is a projection of at least three linearly independent extremals (in this case the extremal trajectory is said to have corank not less than $3$~\cite{notes}). Thus the below necessary optimality condition (Th. $\ref{th:agr_gam}$) is not applicable to a trajectory $q(t)$ under condition \eq{minEmax}.
\end{remark}

\subsection{Bound of the number of switchings on bang-bang trajectories \\with high energy $E$}
\subsubsection{Some necessary results}
We obtain an upper bound on the number of switchings on optimal bang-bang trajectories via the following theorem due to A. Agrachev and R. Gamkrelidze.
\begin{theorem}[\cite{agr_gam, bbds}]
\label{th:agr_gam}
Let $(q(\cdot), u(\cdot))$ be an extremal pair for problem \eq{sys}--\eq{T} and let $\lambda_{\cdot}$ be an extremal lift of $q(\cdot)$. Assume that $\lambda_{\cdot}$ is the unique extremal lift of $q(\cdot)$, up to multiplication by a positive scalar. Assume that there exist $0 = t_0 < t_1 < t_2 < \dots < t_k < \tau_{k+1} = T$ and $u^0, \dots, u^k \in U$ such that $u(\cdot)$ is constantly equal to $u^j$ on $(\tau_j, \tau_{j+1})$ for $j = 0, \dots, k$.

Fix $j = 1, \dots, k$. For $i = 0, \dots, k$ let $Y_i = u_1^i X_1 + u_2^i X_2$ and define recursively the operators 
\begin{align*}
&P_j = P_{j-1} = \Id_{\Vec(M)}, \\
&P_i = P_{i-1} \circ e^{(t_i - t_{i-1}) \ad Y_{i-1}}, \qquad i = j +1, \dots, k,\\
&P_i = P_{i+1} \circ e^{-(t_{i+2} - t_{i+1}) \ad Y_{i+1}}, \qquad i = 0, \dots, j-2.\\
\end{align*}
Define the vector fields 
$$ Z_i = P_i (Y_i), \quad i = 0, \dots, k.$$
Let $Q$ be the quadratic form
$$ Q(\alpha) = \sum_{0 \le i < l \le k} \alpha_i \alpha_l \langle \lambda_{t_j}, [Z_i, Z_l] (q(t_j)) \rangle,$$
defined on the space 
$$ W = \left\{ \alpha = (\alpha_0, \dots, \alpha_k) \in \mathbb{R}^{k+1} \mid \sum_{i=0}^k \alpha_i = 0, \quad \sum_{i=0}^k \alpha_i Z_i (q(t_j)) = 0 \right\}. $$
If $Q$ is not negative-semidefinite, then $q(\cdot)$ is not optimal.
\end{theorem}

We will check the sign of the quadratic form $Q|_W$ via the following test.

Consider a quadratic form 
$$ A(x) = \sum_{i, j = 1}^n a_{ij} x_i x_j, \qquad a_{ij} = a_{ji}, \quad x_i \in \mathbb{R}.$$
Denote a minor
\begin{gather*}
A
\begin{pmatrix}
i_1& i_2& \dots& i_p\\
i_1& i_2& \dots& i_p
\end{pmatrix}
=
\begin{vmatrix}
a_{i_1 i_1}& a_{i_1 i_2}& \dots& a_{i_1 i_p}\\
a_{i_2 i_1}& a_{i_2 i_2}& \dots& a_{i_2 i_p}\\
\hdotsfor{4} \\
a_{i_p i_1}& a_{i_p i_2}& \dots& a_{i_p i_p}\\
\end{vmatrix}.
\end{gather*}
\begin{theorem}[\cite{gant}]
\label{th:gant}
A quadratic form $A(x)$ is negative-semidefinite iff the following inequalities hold:
\begin{gather*}
(-1)^p A
\begin{pmatrix}
i_1& i_2& \dots& i_p\\
i_1& i_2& \dots& i_p
\end{pmatrix}
\ge 0, \qquad 1 \le i_1 < i_2 < \dots < i_p \le n, \quad p = 1, 2, \dots, n.
\end{gather*}
\end{theorem}

\subsubsection{Case $1)$, domain $C_4$}
\begin{theorem}
\label{th:1)I2}
Let $h_4 > h_5 > 0$, $E \in (-h_5, h_5)$. Then any bang-bang control with $8$ switchings is not optimal.
\end{theorem}
\begin{proof}
Consider a control starting from $(1, -1)$ and having  $k = 8$ switchings (controls starting from other values are considered similarly). We apply Th.~\ref{th:agr_gam} and show that such control is not optimal. We have $0 = t_0 < t_1 < \dots < t_9 = T$, where
\begin{align*}
&t_1 \in (0, \tau_1], \quad t_2 - t_1 = t_4 - t_3 = t_6 - t_5 = t_8 - t_7 = \tau_2,\\
&t_3 - t_2 = t_7 - t_6 = \tau_3, \quad t_5 - t_4 = \tau_1, \quad t_9 - t_8 \in (0, \tau_1],
\end{align*}
and the values $\tau_1, \tau_2, \tau_3$ are defined in Subsec.~\ref{subsec:1)I2}. Further, we have 
\begin{align*}
&u|_{(t_0, t_1)} = u|_{(t_4, t_5)} = u|_{(t_8, t_9)} = (1, 1),\\
&u|_{(t_1, t_2)} = u|_{(t_3, t_4)} = u|_{(t_5, t_6)} = u|_{(t_7, t_8)} = (-1, 1), \\
&u|_{(t_2, t_3)} = u|_{(t_6, t_7)} = (-1, -1),
\end{align*}
see Fig.~\ref{fig:1)I2}.
We apply Th.~\ref{th:agr_gam} in the case $k = 8$, $j = 1$. We use the basis $(X_+, X_-, X_3, X_{++}, X_{--})$ in the Lie algebra~$L$, where $X_+ = X_1 + X_2$, $X_- = X_1 - X_2$, $X_{++} = X_4 + X_5$, $X_{--} = X_4 - X_5$.
Then 
$$ Y_0 = - Y_2 = Y_4 = -Y_6 = Y_8 = X_+, \quad Y_1 = Y_3 = Y_5 = Y_7 = -X_-. $$
Further,
\begin{align*}
&P_1 = P_0 = \Id, \qquad &&P_5 = P_4 \circ e^{\tau_1 \ad X_+},\\
&P_2 = e^{-\tau_2 \ad X_-}, &&P_6 = P_5 \circ e^{-\tau_2 \ad X_-},\\
&P_3 = P_2 \circ e^{-\tau_3 \ad X_+}, &&P_7 = P_6 \circ e^{-\tau_3 \ad X_+},\\
&P_4 = P_3 \circ e^{-\tau_2 \ad X_-}, &&P_8 = P_7 \circ e^{-\tau_2 \ad X_-}.
\end{align*}
Thus
\begin{align*}
&Z_0 = X_+,\\
&Z_1 = -X_-,\\
&Z_2 = -X_+ + 2\tau_2 X_3 - \tau_2^2 X_{--},\\
&Z_3 = -X_- - 2\tau_3 X_3 + \tau_3^2 X_{++} + 2\tau_2 \tau_3 X_{--},\\
&Z_4 = X_+ - 4 \tau_2 X_3 + 2 \tau_2 \tau_3 X_{++} + 4 \tau_2^2 X_{--},\\
&Z_5 = -X_- + (2\tau_1 - 2\tau_3) X_3 + (\tau_1^2 - 2\tau_1 \tau_3 + \tau_3^2) X_{++} + (2 \tau_2 \tau_3 - 4 \tau_1 \tau_2) X_{--},\\
&Z_6 = -X_+ + 6\tau_2 X_3 + (2 \tau_1 \tau_2 - 4\tau_2 \tau_3) X_{++} - 9\tau_2^2 X_{--},\\
&Z_7 = -X_- + (2\tau_1 - 4\tau_3) X_3 + (4\tau_3^2 - 4\tau_1 \tau_3 + \tau_1^2) X_{++} + (8\tau_2 \tau_3 - 4 \tau_1 \tau_2) X_{--},\\
&Z_8 = X_+ - 8 \tau_2 X_3 + (8 \tau_2 \tau_3 - 4 \tau_2 \tau_1) X_{++} + 16 \tau_2^2 X_{--}.
\end{align*}
Introduce the notation:
$$ c = h_3, \quad a = h_4 + h_5, \quad b = h_4 - h_5.$$
Then $Q(\alpha) = \sum_{i, l = 0}^8 \sigma_{il} \alpha_i \alpha_l$, where
\begin{align*}
&\sigma_{01} = 2c,\qquad  &&\sigma_{16} = 2c - 6\tau_2 b,\qquad &&\sigma_{37} = (2\tau_3 - 2\tau_1) b,\\
&\sigma_{02} = 2\tau_2 a, &&\sigma_{17} = (4\tau_3 - 2\tau_1) b, &&\sigma_{38} = -2c + 2\tau_3 a + 8\tau_2 b,\\
&\sigma_{03} = 2c - 2\tau_3 a, &&\sigma_{18} = -2c + 8\tau_2 b, &&\sigma_{45} = 2c + (2\tau_1 - 2\tau_3)a - 4\tau_2 b,\\
&\sigma_{04} = -4\tau_2 a, &&\sigma_{23} = -2c + 2\tau_3 a + 2\tau_2 b, &&\sigma_{46} = 2\tau_2 a,\\
&\sigma_{05} = 2c + (2\tau_1 - 2\tau_3) a, &&\sigma_{24} = 2\tau_2 a, &&\sigma_{47} = 2c - 4\tau_2 b + (2\tau_1 - 4\tau_3) a,\\
&\sigma_{06} = 6\tau_2 a, &&\sigma_{25} = -2c - (2\tau_1 -2\tau_3) a + 2\tau_2 b, &&\sigma_{48} = -4\tau_2 a,\\
&\sigma_{07} = 2c + (2\tau_1 - 4\tau_3) a, &&\sigma_{26} = -4\tau_2 a, &&\sigma_{56} = 2c+ (2\tau_1 - 2\tau_3)a - 6\tau_2 b,\\
&\sigma_{08} = -8\tau_2 a, &&\sigma_{27} = -2c - (2\tau_1 - 4\tau_3) a +2\tau_2 b, &&\sigma_{57} = 2\tau_3b,\\
&\sigma_{12} = 2c - 2\tau_2 b, &&\sigma_{28} = 6\tau_2 a, &&\sigma_{58} = -2c + (2\tau_3 - 2\tau_1)a + 8\tau_2 b,\\
&\sigma_{13} = 2\tau_3 b, &&\sigma_{34} = -2c + 2\tau_3 a + 4\tau_2 b, &&\sigma_{67} = -2c + (4\tau_3 - 2\tau_1)a + 6\tau_2 b,\\
&\sigma_{14} = -2c + 4\tau_2 b, &&\sigma_{35} = -2\tau_1b, &&\sigma_{68} = 2\tau_2a,\\
&\sigma_{15} = (2\tau_3 - 2\tau_1) b, &&\sigma_{36} = 2c - 2\tau_3a - 6\tau_2 b, &&\sigma_{78} = -2c + (4\tau_3 - 2\tau_1)a + 8\tau_2 b.
\end{align*}
Further,
\begin{align*}
W &= \left\{ (\alpha_0, \dots, \alpha_8) \in \mathbb{R}^9  \mid  \sum_{i=0}^8 \alpha_i = 0, \quad \sum_{i=0}^8 \alpha_i Z_i (q(t_1)) = 0 \right\} \\
&= \{ (\alpha_0, \dots, \alpha_8) \in \mathbb{R}^9 \quad  \mid \quad  \alpha_3 = (4\tau_2 \alpha_0 + (2 \tau_3 - \tau_1) \alpha_1 - 2\tau_2 \alpha_2) / \tau_1, \\
&\qquad \alpha_5 = (-4 \tau_2 \alpha_0 + (\tau_1 - 2\tau_3) \alpha_1 + 2\tau_2 \alpha_2) / \tau_1 - 2\tau_2 \alpha_4 / \tau_3,\\
&\qquad \alpha_6 = -\alpha_2, \quad \alpha_7 = -\alpha_1 + 2 \tau_2 \alpha_4 / \tau_3, \quad \alpha_8 = -\alpha_0 - \alpha_4 \},\\
Q|_W &= \frac{4}{\tau_1 \tau_3} \sum_{i, j = 0,1,2,4} a_{ij} \alpha_i \alpha_j,
\end{align*}
\begin{align*}
&a_{00} = 2\tau_2 \tau_3 (-a \tau_1 + 4b\tau_2),\qquad &&a_{02} = 2\tau_2 \tau_3 (a \tau_1 - 2b\tau_2),\\
&a_{11} = b \tau_3^2 (-\tau_1 + 2\tau_3), &&a_{04} = 2\tau_2 \tau_3 (-c + 2b\tau_2 + a(\tau_3 - \tau_1)),\\
&a_{22} = \tau_2 \tau_3 (-a \tau_1 + 2b \tau_2), &&a_{12} = -\tau_3 (\tau_1 - 2\tau_3) (a\tau_1 - 2b\tau_2)/2,\\
&a_{44} = -\tau_1 \tau_2 (2b\tau_2 + a\tau_3), &&a_{14} = \tau_3 (2b\tau_2 (\tau_1 +2\tau_3) + \tau_3 (-2c - a\tau_1 + 2a\tau_3)) / 2,\\
&a_{01} = 4b\tau_2 \tau_3^2, &&a_{24} = -\tau_2 (-2a \tau_1^2 +6b\tau_1 \tau_2 + a\tau_1 \tau_3 + 4b\tau_2 \tau_3 + 2a \tau_3^2 - 2c(\tau_1 + \tau_3))/2.
\end{align*}
Then
\begin{gather*}
A
\begin{pmatrix}
0& 1\\
0& 1
\end{pmatrix}
 = 
\begin{vmatrix}
a_{00}& a_{01}\\
a_{01}& a_{11}\\
\end{vmatrix}
= \frac{512(1+Y)(X+Y)\left(X+Y-\sqrt{(1+Y)(X+Y)}\right)}{h_4 (1 + X)^4}, \\
X = \frac{h_5}{h_4} \in (0, 1), \quad Y = \frac{E}{h_4} \in (-X, X).
\end{gather*}
Thus $A
\begin{pmatrix}
0& 1\\
0& 1
\end{pmatrix} < 0$, i.e., the quadratic form $Q|_W$ is not negative semidefinite. By Th.~\ref{th:agr_gam}, the control $u$ is not optimal.
\end{proof}
\subsubsection{Bound of the number of switchings in the general case of high energy $E$}
The rest cases are considered similarly to Th.~\ref{th:1)I2}:
\begin{itemize}
\item Case 1), $\cup_{i=5}^8 C_i,$
\item Case 2), $\cup_{i=4}^6 C_i$,
\item Case 3),
\item Case 4).
\end{itemize}
In all these cases Th.~\ref{th:agr_gam} and Th.~\ref{th:gant} imply that $k = 12$ switchings are not optimal.

Passing from the fundamental domain $\{ h_4 \ge h_5 \ge 0 \}$ of the group $G$ to the whole plane $(h_4, h_5)$, we get the following general bound of the number of switchings.
\begin{theorem}\label{th:gen_bound}
If $E > \max(-|h_4|, -|h_5|)$, then optimal bang-bang trajectories have no more than 11 switchings.

In particular, in this case $t_{\cut} (\lambda) < +\infty$.
\end{theorem}

\section{General form of normal extremals}
Now we prove that the list of types of normal extremals given in Sec.~\ref{sec:problem} is complete.

\begin{theorem}
\label{th:gen_controls}
If $\lambda_t, t \in [0, T]$, is a normal extremal, then there exist $0 \le t_1 < t_2 < \dots < t_n \le T$, for which the following conditions hold:
\begin{itemize}
\item $h_1 h_2 (\lambda_{t_i}) = 0, \quad i = 1, \dots, n,$
\item for any $i = 1, \dots, n-1,$ one of the following conditions is satisfied:
\begin{align*}
&h_1 h_2 (\lambda_t)|_{(t_i, t_{i+1})} \ne 0 \textrm{ or}\\
&h_1 (\lambda_t)|_{[t_i, t_{i+1}]} \equiv 0, \quad h_2 (\lambda_t)|_{(t_i, t_{i+1})} \ne 0 \textrm{ or}\\
&h_2 (\lambda_t)|_{[t_i, t_{i+1}]} \equiv 0, \quad h_1 (\lambda_t)|_{(t_i, t_{i+1})} \ne 0.
\end{align*}
\end{itemize}
\end{theorem}
\begin{proof}
Introduce the sets
$$ Z = \{ t \in [0, T] \mid h_1 h_2 (\lambda_t) = 0 \}, 
\qquad N = [0, T] \backslash Z.$$
The set $N$ is open in $[0, T]$, thus it consists of a finite or countable number of open intervals (and, may be, two half-open intervals near the endpoints of $[0, T]$). We prove that $N$ consists of a finite number of intervals. By contradiction, suppose that $N$ is a countable union of non-intersecting intervals. Choosing a point in each interval, construct a sequence $\{ \tau_n \mid n \in \mathbb{N} \} \subset N$. Passing to a subsequence, we can assume that $\exists \lim_{n \to \infty}\tau_n = \bar{t} \in [0, T]$, and $\tau_n < \bar{t}$ for all $n \in \mathbb{N}$ (or $\tau_n > \bar{t}$ for all $n \in \mathbb{N}$, which is considered similarly). Since the points $\tau_n$ belong to different connected components of $N$, there exists a sequence $\{ s_n \}, n \in \mathbb{N}$, such that $s_n \in Z, \quad \tau_n < s_n < \tau_{n+1} < s_{n+1} < \dots < \bar{t}, \quad n \in \mathbb{N}$.
Thus $\lim_{n \to \infty} s_n = \bar{t}$. Since $h_1 h_2 (\lambda_{\tau_n}) \ne 0$, there exists an interval $(\alpha_n, \beta_n) \ni \tau_n$ such that $h_1 h_2 (\lambda_t)|_{(\alpha_n, \beta_n)} \ne 0$, $n \in \mathbb{N}$. Each extremal arc $\lambda_t |_{(\alpha_n, \beta_n)}$, $n \in \mathbb{N}$, is a bang arc, and the Casimirs $h_4$, $h_5$, $E$ take the same value on each of these arcs. Thus duration of all bang arcs is separated from zero: $\beta_n - \alpha_n \ge C > 0$, $C = C (h_4, h_5, E)$, see Cor.~\ref{cor:tau}. Then $s_{n-1} - s_n \ge \beta_n - \alpha_n \ge C > 0$, which contradicts the equality $\lim_{n \to \infty} s_n = \bar{t}$. Thus $N$ consists of a finite number of intervals.

Let $h_1 h_2 (\lambda_t)|_{[t_i, t_{i+1}]} \equiv 0$. Notice that if $h_j (\lambda_t) = 0$, then $h_{3-j} (\lambda_t) \ne 0, j = 1,2$, since $H(\lambda_t) = (|h_1|+|h_2|)(\lambda_t) > 0$ for a normal extremal $\lambda_t$.
Take any $\bar{t} \in [t_i, t_{i+1}]$, then we can assume that $h_1 (\lambda_{\bar{t}}) = 0, h_2 (\lambda_{\bar{t}}) \ne 0$ (the case $h_2  (\lambda_{\bar{t}}) = 0, h_1  (\lambda_{\bar{t}}) \ne 0$ is considered similarly). Then there exists a neighborhood $O(\bar{t}) \subset [t_1, t_2]$ such that $h_2 (\lambda_t)|_{O(\bar{t})} \ne 0$, thus $h_1 (\lambda_t)|_{O(\bar{t})} \equiv 0$. Thus the set $\{ t \in [t_i, t_{i+1}] \mid h_1(\lambda_t) = 0 \}$ is open and closed in $[t_i, t_{i+1}]$, so it coincides with $[t_i, t_{i+1}]$. In other words,
$$ h_1 (\lambda_t)|_{[t_i, t_{i+1}]} \equiv 0, \qquad h_2 (\lambda_t)|_{(t_i, t_{i+1})} \ne 0.$$
Summing up, any normal extremal is either bang-bang, or singular, or mixed.
\end{proof}

\section{Mixed extremals}
Consider an extremal $\lambda_t, t \in [0, T]$, and let $0 \le \alpha < \beta < \gamma \le T$. Let the
arc $\lambda_t |_{[\alpha, \beta]}$ be bang-bang, and let the arc $\lambda_t |_{[\beta, \gamma]}$ be singular. Then we say that the bang-bang arc $\lambda_t |_{[\alpha, \beta]}$  adjoins the singular arc $\lambda_t |_{[\beta, \gamma]}$ at the point $\lambda_\beta$. Similarly in the case when a singular arc precedes a bang-bang arc.

Notice that singular arcs of types (a), (b) were described in Theorems~3, 4~\cite{SFCartan1}.

\begin{prop}
Let $h_4 \ge h_5 \ge 0$.

A singular arc can adjoin a bang-bang arc only at  points $\lambda_{\bar{t}}$ that satisfy the following conditions:
\begin{itemize}
\item
$\theta = \frac{3\pi}{2}$, $h_3 = 0$, $0\le  h_5 \le h_4$\\
($h_1$-singular arc of type $(b)$ adjoins a bang-bang arc),
\item
$\theta = 0$, $h_3 = 0$, $0< h_5 = h_4$\\
($h_2$-singular arc of type $(b)$ adjoins a bang-bang arc).
\end{itemize}
\end{prop}
\begin{proof}
Singular arcs of type (a) cannot adjoin bang-bang arcs since these singular arcs satisfy the equalities $h_3 = h_4 = h_5 = 0$ (see Th.~3, 4~\cite{SFCartan1}), but these equalities cannot hold on bang-bang extremals (see Subsec.~\ref{subsec:4)}).

$h_1$-singular arcs of type (b) satisfy the conditions:
$$ \theta = \frac{\pi}{2} + \pi n, \quad h_3 = 0, \quad 0< h_5 \le h_4.$$
The point $\theta = \frac{\pi}{2}, h_3 = 0$ is an equilibrium point of the phase portrait of the reduced Hamiltonian system of PMP, thus the equality $(\theta, h_3) = (\frac{\pi}{2}, 0)$ cannot hold on a bang-bang extremal.
Similarly, $h_2$-singular arc of type  (b) satisfies the conditions
$$ \theta = \pi n, \quad h_3 = 0, \quad h_4 = h_5 = 0,$$
and the equality $(\theta, h_3) = (\pi, 0)$ cannot hold on a bang-bang extremal.
\end{proof}

Notice that singular controls that adjoin bang-bang controls are constant. Thus all mixed controls are piecewise constant, and Th.~\ref{th:agr_gam} can be used for bounding the number of switchings on optimal mixed trajectories.

Mixed extremals are schematically shown in Figs.~\ref{fig:mix1)}-\ref{fig:mix3)}. Small dashed circles near the points $(\theta, h_3) = (\frac{3\pi}{2}, 0)$ and $(\theta, h_3) = (0, 0)$ denote singular arcs that adjoin
bang-bang arcs. Singular arcs are shown by dashed segments.

\twofiglabelsize
{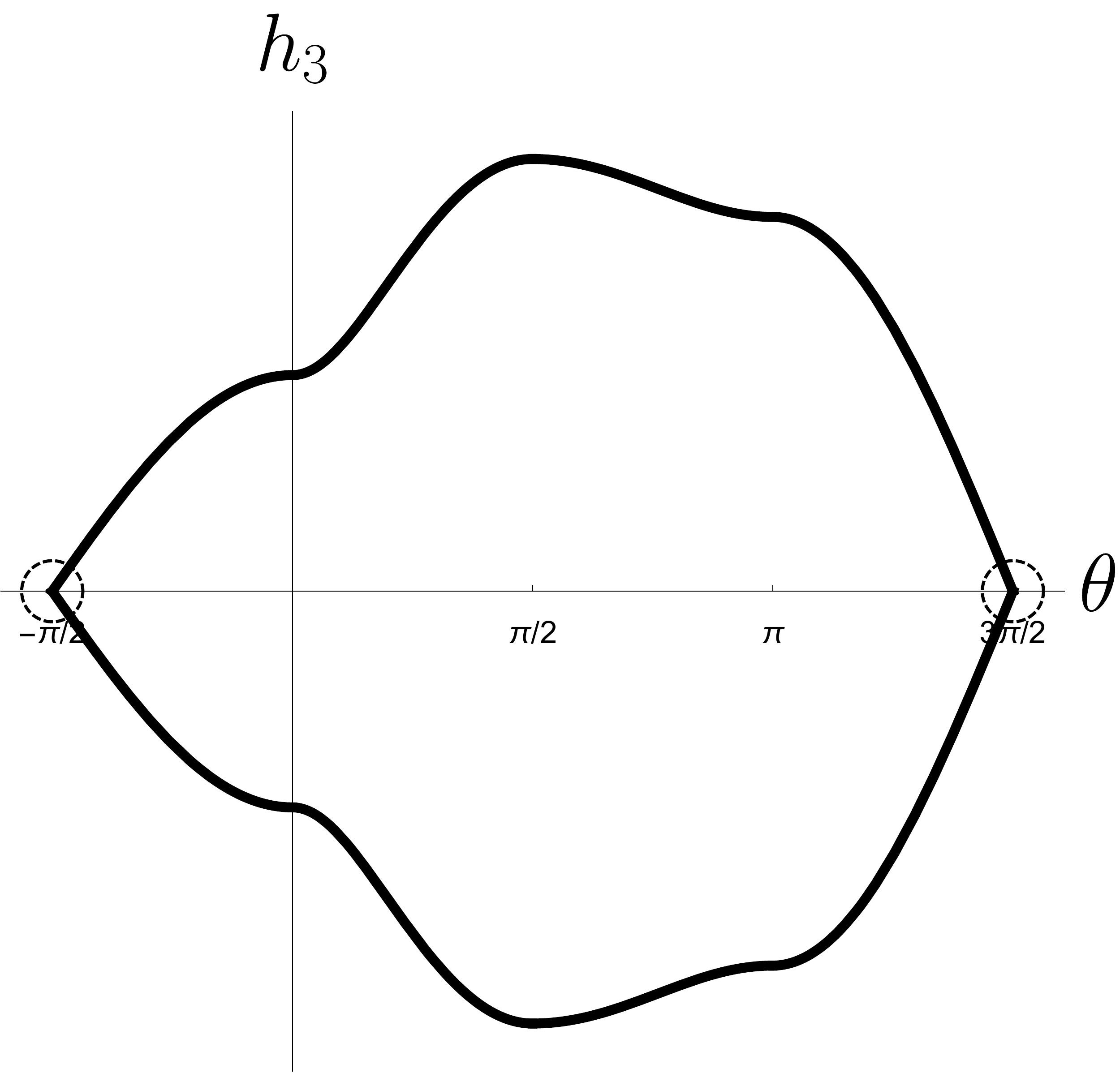}{Mixed extremals $(\theta(t), h_3(t))$ with $h_4 > h_5 > 0$}{fig:mix1)}{0.6}
{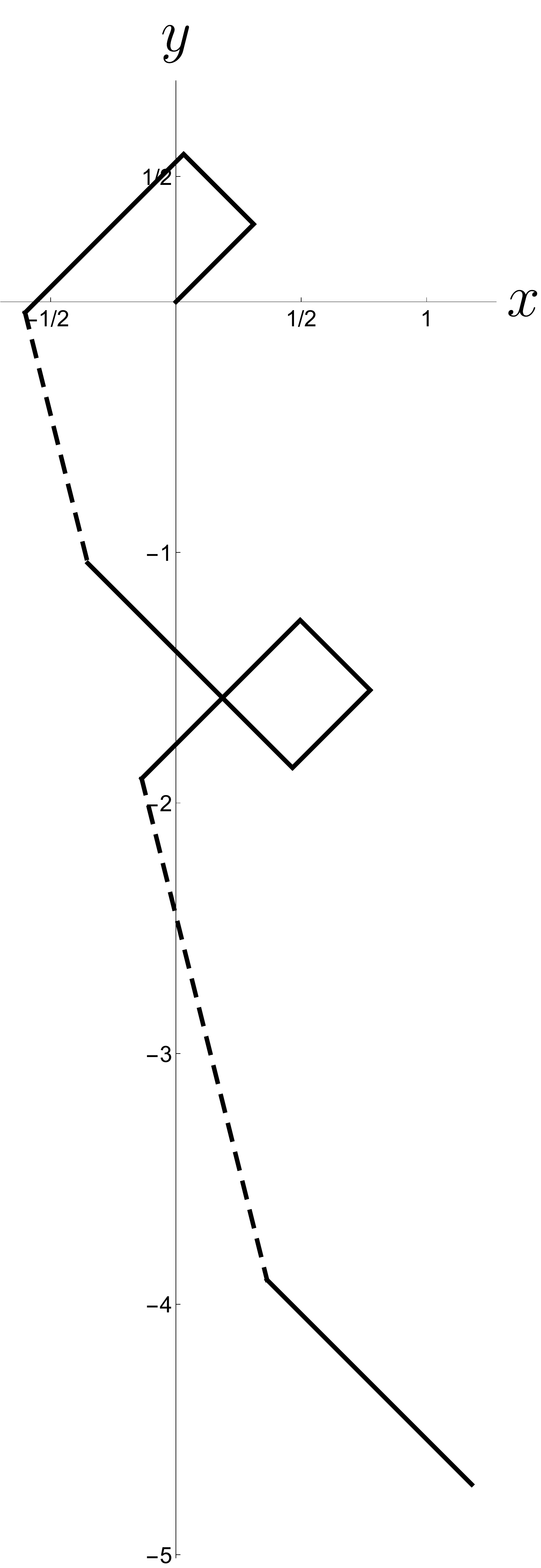}{Example of mixed trajectory $(x(t), y(t))$ with $h_4 > h_5 > 0$}{fig:xym1)}{0.4}

\twofiglabelsize
{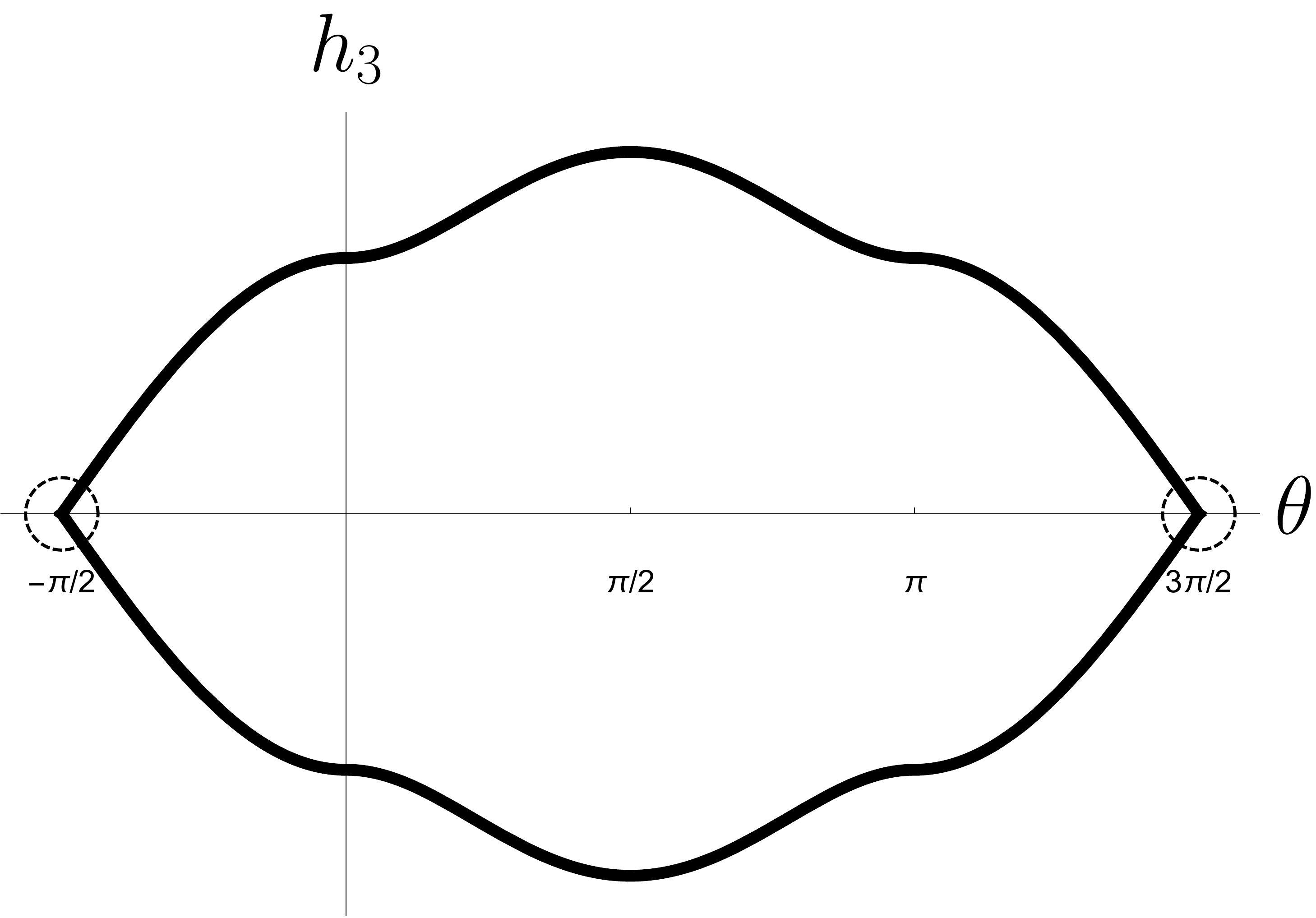}{Mixed extremals $(\theta(t), h_3(t))$  with $h_4 > h_5 = 0$}{fig:mix2)}{0.7}
{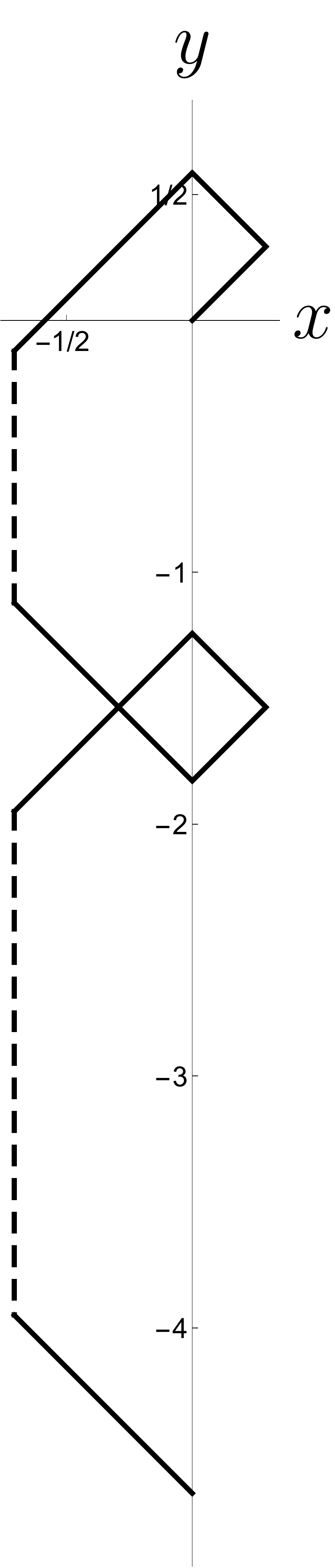}{Example of mixed trajectory $(x(t), y(t))$ with $h_4 > h_5 = 0$}{fig:xym2)}{0.2}

\twofiglabelsize
{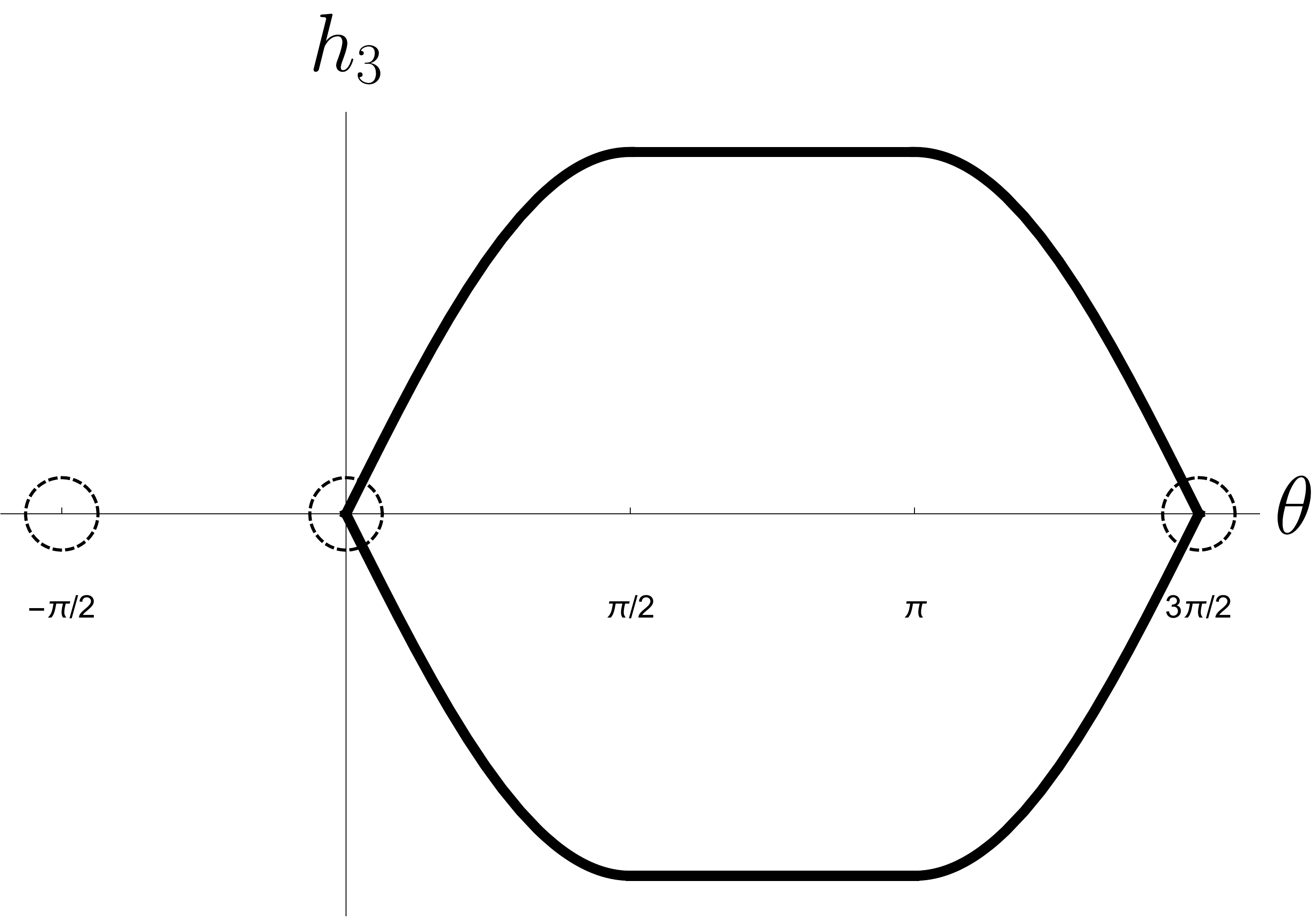}{Mixed extremals $(\theta(t), h_3(t))$  with $h_4 = h_5 > 0$}{fig:mix3)}{0.7}
{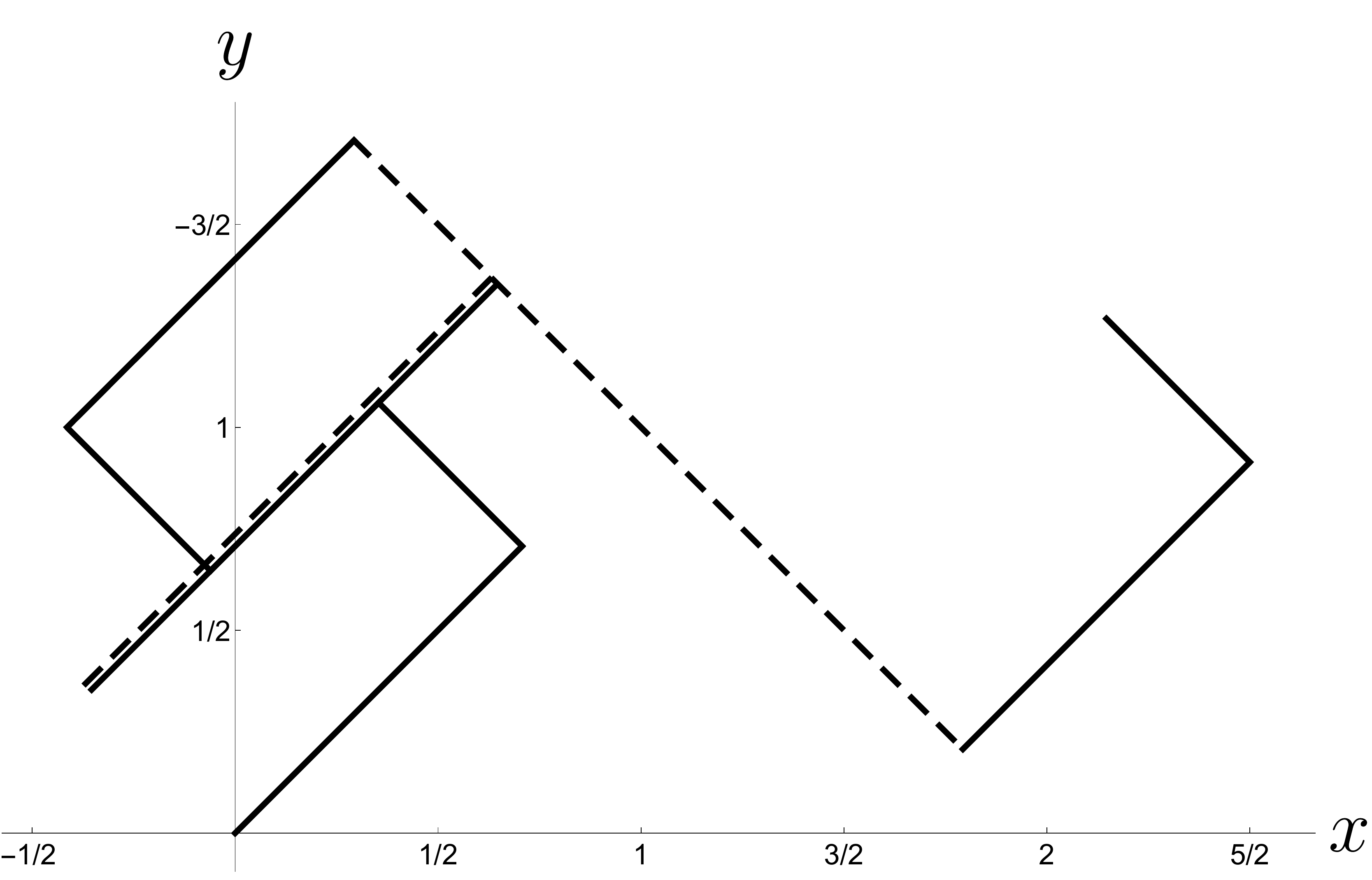}{Example of mixed trajectory $(x(t), y(t))$ with $h_4 = h_5 > 0$}{fig:xym3)}{0.8}


Notice that singular arcs, unlike bang-bang ones, may have arbitrary durations.

Mixed extremals for $h_4 \ge h_5 \ge 0$ arise in the following cases:
\begin{itemize}
\item Case 1): $h_4 > h_5 > 0$, $\theta = \frac{3\pi}{2}$, $h_3 = 0$, level line $C_6$, see Figs.~\ref{fig:mix1)}, \ref{fig:xym1)},
\item Case 2): $h_4 > h_5 = 0$, $\theta = \frac{3\pi}{2}$, $h_3 = 0$, level line $C_5$, see Figs.~\ref{fig:mix2)}, \ref{fig:xym2)},
\item Case 3): $h_4 = h_5 > 0$, $\theta \in \{ 0,\frac{3\pi}{2}\}$, $h_3 = 0$, level line $C_2$, see Figs.~\ref{fig:mix3)}, \ref{fig:xym3)}.
\end{itemize}

Theorem \ref{th:agr_gam} yields the following bound.
\begin{theorem}
\label{th:mix_bound}
Optimal mixed controls have not more that $13$ switchings.
\end{theorem}
Notice that mixed extremals $\lambda_t$ are 
not uniquely determined by the initial covector $\lambda_0$ and time $t$, because of arbitrary duration of singular arcs.
 Thus exponential mapping cannot be defined for mixed extremals, as it was defined for bang-bang ones.

\section{Bound on the number of arcs of minimizers}
Important questions for applications of sub-Finsler geometry in metric group theory are the following: 
\begin{itemize}
\item
given any pair of points in a sub-Finsler manifold, does there exist a piecewise-smooth minimizer that connects these points?
\item
is there a uniform bound on the number of smooth arcs for piecewise smooth minimizers that connect arbitrary points in the manifold?
\end{itemize}

On the basis of our results we can provide affirmative answer for the both questions for the $\ell_{\infty}$ sub-Finsler problem studied in this paper.

\begin{corollary}
Any two points in the Cartan group can be connected by a piecewise smooth minimizer with not more than $14$ smooth arcs.
\end{corollary}
\begin{proof}
By Th.~\ref{th:gen_controls},
any two points in the Cartan group can be connected by a minimizer that belongs to the following (mutually not excluding) classes:
\begin{enumerate}
\item
abnormal,
\item
singular,
\item
bang-bang,
\item
mixed.
\end{enumerate}

Abnormal trajectories are smooth (see Th.~2~\cite{SFCartan1}).
If two points can be connected by a singular trajectory, then they can be connected by a piecewise smooth singular trajectory with not more than 5 smooth arcs (see Prop.~1~\cite{SFCartan1}).
Bang-bang minimizers are piecewise smooth with up to 12 smooth arcs (see Th.~\ref{th:gen_bound}).
Finally, mixed minimizers are piecewise smooth with up to 14 smooth arcs (see Th.~\ref{th:mix_bound}).
\end{proof}

Moreover, we can now prove Th.~\ref{th:result} as a corollary of previously obtained results.
\begin{proof}
Classification of minimizers into types $(i)$--$(iii)$ follows from Th.~\ref{th:gen_controls}.
The bound on the number of switchings on minimizers of type $(ii)$ and not of type $(i)$ is given by Th.~\ref{th:gen_bound};
a similar bound for type $(iii)$ is obtained by Th.~\ref{th:mix_bound}. The length-minimizing property of trajectories of type $(i)$ follows from Lemma~2~\cite{SFCartan1}. Existence of a piecewise smooth minimizer with up to 5 smooth arcs for every trajectory of type $(i)$ follows from Prop.~1~\cite{SFCartan1}.
\end{proof}

\section{Conclusion}
In this paper we continued a study of the $\ell_{\infty}$ sub-Finsler problem on the Cartan group. Many questions remain unsolved, e.g.:
\begin{itemize}
\item
cut time along  bang-bang and mixed trajectories,
\item
cut locus,
\item
regularity of sub-Finsler distance and sphere.
\end{itemize}
We postpone study of these questions to forthcoming papers.

\section*{Acknowledgments}

The authors thank Prof. Andrei Agrachev and Prof. Lev Lokutsievsky for fruitful discussions of sub-Finsler geometry.

\end{document}